\documentclass[a4paper, 11pt, twoside]{amsart}
\usepackage[top=3cm, bottom=3cm, left=3cm, right=3cm]{geometry}
\usepackage[pdftex]{graphicx}
\usepackage{tikz}
\usetikzlibrary{shapes, arrows, arrows.meta}
\usepackage{array}
\usepackage{amsfonts,amstext,amscd,bezier,amsthm,amssymb}
\usepackage[centertags]{amsmath}
\usepackage{mathtools}
\usepackage[integrals]{wasysym}
\usepackage[nice]{nicefrac}
\usepackage{stmaryrd}
\usepackage{bbm}
\usepackage[mathcal]{euscript}
\usepackage[utf8]{inputenc}
\usepackage[T1]{fontenc}
\usepackage[english]{babel}
\usepackage{cite}
\usepackage{mathptmx}
\usepackage[section]{algorithm}
\usepackage{longtable}
\usepackage{url}
\usepackage[inline]{enumitem}
\usepackage[linktocpage, hidelinks, bookmarks, bookmarksnumbered, pdfstartview={XYZ null null 1.00}]{hyperref}
\newcommand{\theoname}{Theorem}
\newcommand{\lemmname}{Lemma}
\newcommand{\coroname}{Corollary}
\newcommand{\propname}{Proposition}
\newcommand{\definame}{Definition}
\newcommand{\hyponame}{Hypotheses}
\newcommand{\remkname}{Remark}
\newcommand{\explname}{Example}

\theoremstyle{plain}
\newtheorem{theo}{\theoname}[section]
\newtheorem{lemm}[theo]{\lemmname}
\newtheorem{coro}[theo]{\coroname}
\newtheorem{prop}[theo]{\propname}
\theoremstyle{definition}
\newtheorem{defi}[theo]{\definame}
\newtheorem{hypo}[theo]{\hyponame}
\newtheorem{remk}[theo]{\remkname}

\DeclareMathOperator{\Lip}{Lip}
\DeclareMathOperator{\diam}{diam}
\DeclareMathOperator{\diverg}{div}
\DeclareMathOperator{\conv}{co}
\DeclareMathOperator{\Proj}{Pr}
\DeclareMathOperator{\Adm}{Adm}
\DeclareMathOperator{\MFG}{MFG}
\DeclareMathOperator{\OCP}{OCP}
\DeclareMathOperator{\Opt}{Opt}
\DeclareMathOperator{\OOpt}{\mathbf{Opt}}

\newcommand*\diff{\mathop{}\!\mathrm{d}}

\newcommand{\dist}{\mathbf{d}}
\newcommand{\R}{\mathbb{R}}

\newcommand{\suchthat}{\;|\:}
\newcommand{\midsuchthat}{\;\middle|\:}

\newcommand{\gengrad}{\partial^{\mathrm{C}}}

\newcommand{\abs}[1]{\left\lvert #1\right \rvert}

\numberwithin{figure}{section}
\numberwithin{equation}{section}

\begin{document}

\setlength{\parskip}{1pt plus 1pt minus 1pt}

\setlist[enumerate, 1]{label={\textnormal{(\alph*)}}, ref={(\alph*)}, leftmargin=0pt, itemindent=*}
\setlist[description, 1]{leftmargin=0pt, itemindent=*}
\setlist[enumerate, 2]{label={\textnormal{(\roman*)}}, ref={(\roman*)}}

\def\paperTitle{{Minimal-time mean field games}}
\newcommand{\paperKeywords}{Optimal control, Nash equilibrium, time-dependent eikonal equation, congestion games, Pontryagin Maximum Principle, MFG system}
\newcommand{\paperMSC}{91A13, 49N70, 49K15, 35Q91}

\expandafter\title\paperTitle

\author{Guilherme Mazanti}
\address{Laboratoire de Math\'ematiques d'Orsay, Univ.\ Paris-Sud, CNRS, Univ.\ Paris-Saclay, 91405 Orsay, France.}
\email{guilherme.mazanti@math.u-psud.fr}

\author{Filippo Santambrogio}
\address{Institut Camille Jordan, Universit\'e Claude Bernard Lyon 1, 43 Boulevard du 11 novembre 1918, 69622 Villeurbanne cedex, France.}
\email{santambrogio@math.univ-lyon1.fr}

\subjclass[2010]{\paperMSC}

\thanks{This work was partially supported by a public grant as part of the ``Investissement d'avenir'' project, reference ANR-11-LABX-0056-LMH, LabEx LMH, PGMO project VarPDEMFG, and by the ANR (Agence Nationale de la Recherche) project ANR-16-CE40-0015-01. The first author was also partially supported by the Hadamard Mathematics LabEx (LMH) through the grant number ANR-11-LABX-0056-LMH in the ``Investissement d'avenir'' project.}

\begin{abstract}
This paper considers a mean field game model inspired by crowd motion where agents want to leave a given bounded domain through a part of its boundary in minimal time. Each agent is free to move in any direction, but their maximal speed is bounded in terms of the average density of agents around their position in order to take into account congestion phenomena.

After a preliminary study of the corresponding minimal-time optimal control problem, we formulate the mean field game in a Lagrangian setting and prove existence of Lagrangian equilibria using a fixed point strategy. We provide a further study of equilibria under the assumption that agents may leave the domain through the whole boundary, in which case equilibria are described through a system of a continuity equation on the distribution of agents coupled with a Hamilton--Jacobi equation on the value function of the optimal control problem solved by each agent. This is possible thanks to the semiconcavity of the value function, which follows from some further regularity properties of optimal trajectories obtained through Pontryagin Maximum Principle. Simulations illustrate the behavior of equilibria in some particular situations.
\end{abstract}

\maketitle

\hypersetup{pdftitle=\paperTitle, pdfauthor={Guilherme Mazanti and Filippo Santambrogio}, pdfkeywords={\paperKeywords}, pdfsubject={\paperMSC}}

\tableofcontents

\section{Introduction}

Introduced around 2006 by Jean-Michel Lasry and Pierre-Louis Lions \cite{Lasry2006JeuxI, Lasry2006JeuxII, Lasry2007Mean} and independently by Peter E.\ Caines, Minyi Huang, and Roland P.\ Malhamé \cite{Huang2006Large, Huang2007Large, Huang2003Individual}, mean field games (written simply MFGs in this paper for short) are differential games with a continuum of players, assumed to be rational, indistinguishable, individually neglectable, and influenced only by the average behavior of other players through a mean-field type interaction. Their original purpose was to provide approximations of Nash equilibria of games with a large number of symmetric players, with motivations from economics \cite{Lasry2006JeuxI, Lasry2006JeuxII, Lasry2007Mean} and engineering \cite{Huang2006Large, Huang2007Large, Huang2003Individual}. In this paper, we use the words ``player'' and ``agent'' interchangeably to refer to those taking part in a game.

Since their introduction, mean field games have attracted much research effort and several works have investigated subjects such as approximation results (how games with a large number of symmetric players converge to MFGs \cite{Kolokoltsov2014Rate, Cardaliaguet2017Convergence}), numerical approximations \cite{Gueant2012New, Carlini2014Fully, Achdou2010Mean, Achdou2016Convergence}, games with large time horizon \cite{Cardaliaguet2013Long1, Cardaliaguet2013Long2}, variational mean field games \cite{Cardaliaguet2016First, Benamou2017Variational, Meszaros2015Variational, Prosinski2017Global}, games on graphs or networks \cite{Gueant2015Existence, Camilli2015Model, Gomes2013Continuous, Cacace2017Numerical}, or the characterization of equilibria using the master equation \cite{CardaliaguetMaster, Bensoussan2017Interpretation, Carmona2014Master}. We refer to \cite{Gomes2014Mean, Gueant2011Mean, CardaliaguetNotes} for more details and further references on mean field games.

The goal of this paper is to study a simple mean field game model for crowd motion. The mathematical modeling and analysis of crowd motion has been the subject of a very large number of works from many different perspectives, motivated not only by understanding but also by controlling and optimizing crowd behavior \cite{Henderson1971Statistics, Piccoli2011Time, Maury2011Handling, Helbing2000Simulating, Faure2015Crowd, Maury2010Macroscopic, CriPicTos}. Some mean field game models inspired by crowd motion have been considered in the literature, such as in \cite{Lachapelle2011Mean, Benamou2017Variational, Cardaliaguet2016First, Burger2013Mean}. Most of these models, as well as most mean field games model in general, consider that the movement of agents takes place in a fixed time interval, and that each agent is free to choose their speed, which is only penalized in the cost. The goal of this paper is to propose and study a novel model where the final time for the movement of an agent is free, and is actually the agent's minimization criterion, and an agent's maximal speed is constrained in terms of the average distribution of agents around their position.  This is motivated by the fact that, in some crowd motion situations, an agent may not be able to move faster by simply paying some additional cost, since the congestion provoked by other agents may work as a physical barrier for the agent to increase their speed. In some sense the closest MFG model to ours is the one in \cite{AchPor-cong}, where congestion effects are modeled via a cost which is multiplicative in the speed and the density (and a constraint could be obtained in the limit); yet, also the model in \cite{AchPor-cong} is set on a fixed time horizon and cannot catch the limit case where speed and density are related by a constraint. As a result, the model that we propose is novel and deserves the preliminary study that we develop in the present paper.

The model we consider in this paper is related to Hughes' model for crowd motion \cite{Hughes2002Continuum, Hughes2003Flow}. As the model we propose here, Hughes' model also considers agents who aim at leaving in minimal time a bounded domain under a congestion-dependent constraint on their speeds. In the present paper, for modeling and existence purposes, we consider both the case where the agents exit through all the boundary and the case where the exit is only a part of it, while some results characterizing equilibria in terms of a system of PDEs require regularity which cannot be obtained when the exit is only trhough a part of the boundary.

The main difference between our model and Hughes' is that, in the latter, at each time, an agent moves in the optimal direction to the boundary assuming that the distribution of agents remains constant, whereas in our model, agents take into account the future evolution of the distribution of agents in the computation of their optimal trajectories. This accounts for the time derivative in the Hamilton--Jacobi equation from \eqref{SystMFG}, which is the main difference between \eqref{SystMFG} and the equations describing the motion of agents in Hughes' model. The knowledge of the future necessary for solving the minimization problem in our model can be interpreted, similarly to other mean field games, as an anticipation of future behavior of other agents based on past experiences in the same situation. In particular, a mean field game model such as ours should be suitable for modeling the behavior of a crowd walking on the streets or subway tunnels, but not for panic situations such as emergency evacuations.

We formulate the notion of equilibrium of a mean field game in this paper in a Lagrangian setting. Contrarily to the classical approach for mean field games consisting on defining an equilibrium in terms of a time-varying measure $m_t$ describing the distribution of agents at time $t$, the Lagrangian approach relies instead on describing the motion of agents by a measure on the set of all possible trajectories. This is a classical approach in optimal transport problems (see, e.g., \cite{Santambrogio2015Optimal, Ambrosio2005Gradient, Villani2009Optimal}), which has been used for instance in \cite{Brenier1989Least} to study incompressible flows, in \cite{Carlier2008Optimal} for Wardrop equilibria in traffic flow, or in \cite{Bernot2009Optimal} for branched transport problems. The Lagrangian approach has also been used for defining equilibria of mean field games, for instance in \cite{Benamou2017Variational, CannarsaExistence, Cardaliaguet2016First, Cardaliaguet2015Weak}.

Our main results are \begin{enumerate*}[label={(\roman*)}]\item Theorem \ref{MainTheoExist}, stating the existence of an equilibrium for the mean field game model we propose in this paper, and \item Theorem \ref{MainTheoCharact}, stating that equilibria satisfy a continuity equation on the time-dependent measure representing the distribution of agents coupled with a Hamilton--Jacobi equation on the value function of the time-minimization problem solved by each agent.\end{enumerate*} The proof of Theorem \ref{MainTheoExist} relies on a reformulation of the notion of equilibrium (provided in Definition \ref{DefiEquilibriumMFG}) in terms of a fixed point and on Kakutani fixed point theorem. The Hamilton--Jacobi equation from Theorem \ref{MainTheoCharact} can be obtained by standard methods on optimal control, but the continuity equation relies on further properties of the value function, and in particular its semiconcavity, obtained thanks to a detailed study of optimal trajectories. 

The notion of equilibrium used in this paper is standard in non-atomic congestion games, and a typical example of these games can be observed in Wardrop equilibria. This notion of equilibrium, introduced in \cite{Wardrop1952Theoretical} for vehicular traffic flow models, has many connections with MFGs and in particular with minimal-time MFGs, and has been studied in several other works \cite{Carlier2011Continuous, Carlier2008Optimal, Haurie1985Relationship, Fischer2010Fast}. In Wardrop equilibria the players are the vehicles and their goal is to minimize their traveling time; a major difference with respect to the notion used in this paper is that Wardrop equilibria are usually considered in the stationary case where the flow of vehicles is constant, whereas MFGs explicitly address the case of a time-dependent flow of agents.

The paper is organized as follows. Section \ref{SecNotations} provides the main notations used in this paper and recalls some classical results and definitions. Section \ref{SecMFG} presents the mean field game model that we consider, together with an associated optimal control problem, presenting and discussing the main assumptions that are used in this paper. The associated optimal control problem is studied in Section \ref{SecOCP}, its main results being used in Section \ref{SecExistence} to obtain existence of an equilibrium to our mean field game and in Section \ref{SecCharacterization} to prove that the distribution of agents and the value function of the optimal control problem solved by each agent satisfy a system of partial differential equations. Examples and numerical simulations are provided in Section \ref{SecExamples}.

\section{Notations and preliminary definitions}
\label{SecNotations}

Let us set the main notations used in this paper. The usual Euclidean norm in $\mathbb R^d$ is denoted by $\abs{\cdot}$, and we write $\mathbb S^{d-1}$ for the usual unit Euclidean sphere in $\mathbb R^d$, i.e., $\mathbb S^{d-1} = \{x \in \mathbb R^d \suchthat \abs{x} = 1\}$. Given two sets $A, B$, the notation $f: A \rightrightarrows B$ indicates that $f$ is a set-valued map from $A$ to $B$, i.e., $f$ maps a point $x \in A$ to a subset $f(x) \subset B$.

For a metric space $X$, $x \in X$, and $r \in \mathbb R_+$, the open and closed balls of center $x$ and radius $r$ are denoted respectively by $B_X(x, r)$ and $\overline B_X(x, r)$, these notations being simplified respectively to $B_d(x, r)$ and $\overline B_d(x, r)$ when $X = \mathbb R^d$, with a further simplification to $B_d$ and $\overline B_d$ if $x = 0$ and $r = 1$. Given two metric spaces $X$ and $Y$, we denote by $\mathcal C(X, Y)$ the set of continuous functions from $X$ to $Y$ and by $\Lip(X, Y)$ the set of those which are Lipschitz continuous. When $X \subset \mathbb R^d$, we will also make use of the space $\mathcal C^{1, 1}(X, \mathbb R)$.

Given a complete, separable, and bounded metric space $X$ with metric $\dist$, let $\mathcal C_X = \mathcal C(\mathbb R_+, X)$\label{MathcalCX} be endowed with the topology of uniform convergence on compact sets ($\mathcal C_X$ is a Polish space for this distance, see, e.g., \cite[Corollary 3, page X.9; Corollary, page X.20; and Corollary, page X.25]{Bourbaki2007Topologie}). Whenever needed, we endow $\mathcal C_X$ with the complete distance
\[
\dist_{\mathcal C_X}(\gamma_1, \gamma_2) = \sum_{n=1}^\infty \frac{1}{2^n} \sup_{t \in [0, n]}\dist(\gamma_1(t), \gamma_2(t)).
\]
For $c \geq 0$, we denote by $\Lip_c(X)$\label{LipcX} the subset of $\mathcal C_X$ containing all $c$-Lipschitz continuous functions. Recall that, if $X$ is compact, then, thanks to Arzelà--Ascoli Theorem \cite[Corollary 3, page X.19]{Bourbaki2007Topologie}, $\Lip_c(X)$ is compact. For $t \in \mathbb R_+$, we denote by $e_t: \mathcal C_X \to X$\label{DefiET} the evaluation map $e_t(\gamma) = \gamma(t)$.

For a given topological space $X$, let $\mathcal P(X)$\label{MathcalPX} denote the set of all Borel probability measures on $X$. If $X$, endowed with a metric $\dist$, is a complete, separable, and bounded metric space, we endow $\mathcal P(X)$ with the usual Wasserstein distance $W_1$\label{DefiW1} defined by its dual formulation (see, e.g., \cite[Chapter 7]{Ambrosio2005Gradient} and \cite[Chapter 5]{Santambrogio2015Optimal}):
\begin{equation}
\label{WassersteinDual}
W_1(\mu, \nu) = \sup\left\{\int_X \phi(x) \diff(\mu - \nu)(x) \midsuchthat \phi \in \Lip(X, \mathbb R) \text{ is $1$-Lipschitz}\right\}.
\end{equation}

Let $X$ be a metric space with metric $\dist$, $a, b \in \mathbb R$ with $a < b$, and $\gamma: (a, b) \to X$. The \emph{metric derivative}\label{AbsDotGamma} of $\gamma$ at a point $t \in (a, b)$ is defined by
\[\abs{\dot\gamma}(t) = \lim_{s \to t} \frac{\dist(\gamma(s), \gamma(t))}{\abs{s - t}}\]
whenever this limit exists. Recall that, if $\gamma$ is absolutely continuous, then $\abs{\dot\gamma}(t)$ exists for almost every $t \in (a, b)$ (see, e.g., \cite[Theorem 1.1.2]{Ambrosio2005Gradient}).

We shall also need in this paper some classical tools from non-smooth analysis, which we briefly recall now, following the presentation from \cite[Chapter 2]{Clarke1990Optimization}.

\begin{defi}
Let $O \subset \mathbb R^d$ be open, $\phi: O \to \mathbb R$ be Lipschitz continuous, $x \in O$, and $v \in \mathbb R^d$. We define \emph{generalized directional derivative} $\phi^\circ(x, v)$ of $\phi$ at $x$ in the direction $v$ by
\[
\phi^\circ(x, v) = \limsup_{\substack{y \to x \\ h \searrow 0}} \frac{\phi(y + h v) - \phi(y)}{h},
\]
and the \emph{generalized gradient}\label{GenGradPhi} $\gengrad \phi(x)$ of $\phi$ at $x$ by
\begin{equation}
\label{EqDefiGeneralizedGradient}
\gengrad \phi(x) = \{\xi \in \mathbb R^d \suchthat \phi^\circ(x, v) \geq \xi \cdot v \text{ for all } v \in \mathbb R^d\}.
\end{equation}
\end{defi}

The generalized gradient is also sometimes referred to as \emph{Clarke's gradient} in the literature, which explains the notation $\gengrad$ used here. We shall also need the notion of partial generalized gradients.

\begin{defi}
Let $O_1 \subset \mathbb R^{d_1}$, $O_2 \subset \mathbb R^{d_2}$ be open sets, $\phi: O_1 \times O_2 \to \mathbb R$ be Lipschitz continuous, and $(x_1, x_2) \in O_1 \times O_2$. The \emph{partial generalized gradient} $\gengrad_1\phi(x_1, x_2) = \gengrad_{x_1}\phi(x_1, x_2)$ is defined as the generalized gradient of the function $x_1 \mapsto \phi(x_1, x_2)$. The partial generalized gradient $\gengrad_{2} \phi = \gengrad_{x_2}\phi$ is defined similarly.
\end{defi}

Recall that, in general, there is no link between $\gengrad\phi(x_1, x_2)$ and $\gengrad_1\phi(x_1, x_2) \times \gengrad_2\phi(x_1, x_2)$ (see, e.g., \cite[Example 2.5.2]{Clarke1990Optimization}).

\begin{defi}
Let $C \subset \mathbb R^d$ be non-empty and $\dist_C: \mathbb R^d \to \mathbb R_+$ be the distance to $C$, defined by $\dist_C(x) = \inf_{y \in C}\abs{x-y}$. Let $x \in C$. We define the \emph{tangent cone} $T_C(x)$ to $C$ at $x$ by
\[
T_C(x) = \{v \in \mathbb R^d \suchthat \dist_C^\circ(x, v) = 0\},
\]
and the \emph{normal cone} $N_C(x)$ to $C$ at $x$ by
\[
N_C(x) = \{\xi \in \mathbb R^d \suchthat \xi \cdot v \leq 0 \text{ for all } v \in T_C(x)\}.
\]
\end{defi}

We refer to \cite[Chapter 10]{Clarke2013Functional} and \cite[Chapter 2]{Clarke1990Optimization} for standard properties of generalized directional derivatives, generalized gradients, and tangent and normal cones.

\section{The MFG model}
\label{SecMFG}

We provide in this section a mathematical description of the mean field game model considered in this paper. Let $(X,\dist)$ be a complete and separable metric space, $\Gamma \subset X$ be non-empty and closed, and $K: \mathcal P(X) \times X \to \mathbb R_+$. We consider the following mean field game, denoted by $\MFG(X, \Gamma, K)$\label{MFGXGammaK}. Agents evolve in $X$, their distribution at time $t \in \mathbb R_+$ being given by a probability measure $m_t \in \mathcal P(X)$. The goal of each agent is to reach the exit $\Gamma$ in minimal time, and, in order to model congestion, we assume the speed of an agent at a position $x$ in time $t$ to be bounded by $K(m_t, x)$.

Notice that, for a given agent, their choice of trajectory $\gamma$ depends on the distribution of all agents $m_t$, since the speed of $\gamma$, i.e.\ its metric derivative $\abs{\dot\gamma}$, should not exceed $K(m_t, x)$. On the other hand, the distribution of the agents $m_t$ itself depends on how agents choose their trajectories $\gamma$. We are interested here in the equilibrium situations, meaning that, starting from a time evolution of the distribution of agents $m: \mathbb R_+ \to \mathcal P(X)$, the trajectories $\gamma$ chosen by agents induce an evolution of the initial distribution of agents $m_0$ that is precisely given by $m$.

In order to provide a more mathematically precise definition of this mean field game and the notion of equilibrium, we first introduce an optimal control problem where agents evolving in $X$ want to reach $\Gamma$ in minimal time, their speed being bounded by some time- and state-dependent function $k: \mathbb R_+ \times X \to \mathbb R_+$. Here $k$ will not depend on the density of the agents, it will be considered as given. This optimal control problem is denoted in the sequel by $\OCP(X, \Gamma, k)$\label{OCPXGammak}.

\begin{defi}[$\OCP(X, \Gamma, k)$]
\label{DefiOCP}
Let $(X,\dist)$ be a complete and separable metric space, $\Gamma \subset X$ be non-empty and closed, and $k: \mathbb R_+ \times X \to \mathbb R_+$ be continuous.
\begin{enumerate}
\item A curve $\gamma \in \Lip(\mathbb R_+, X)$ is said to be $k$-\emph{admissible} for $\OCP(X, \Gamma, k)$ if its metric derivative $\abs{\dot\gamma}$ satisfies $\abs{\dot\gamma}(t) \leq k(t, \gamma(t))$ for almost every $t \in \mathbb R_+$. The set of all $k$-admissible curves is denoted by $\Adm(k)$\label{DefiAdmk}.

\item Let $t_0 \in \mathbb R_+$. The \emph{first exit time}\label{Taut0gamma} after $t_0$ of a curve $\gamma \in \Lip(\mathbb R_+, X)$ is the number $\tau(t_0, \gamma) \in [t_0, +\infty]$ defined by
\[
\tau(t_0, \gamma) = \inf\{t \geq 0 \mid \gamma(t + t_0) \in \Gamma\}.
\]

\item Let $t_0 \in \mathbb R_+$ and $x_0 \in X$. A curve $\gamma \in \Lip(\mathbb R_+, X)$ is said to be a \emph{time-optimal curve} or \emph{time-optimal trajectory} (or simply \emph{optimal curve} or \emph{optimal trajectory}) for $(k, t_0, x_0)$ if $\gamma \in \Adm(k)$, $\gamma(t) = x_0$ for every $t \in [0, t_0]$, $\tau(t_0, \gamma) < +\infty$, $\gamma(t) = \gamma(t_0 + \tau(t_0, \gamma)) \in \Gamma$ for every $t \in [t_0 + \tau(t_0, \gamma),\allowbreak +\infty)$, and
\begin{equation}
\label{EqMinimalTime}
\tau(t_0, \gamma) = \min_{\substack{\beta \in \Adm(k) \\ \beta(t_0) = x_0}} \tau(t_0, \beta).
\end{equation}
The set of all optimal curves for $(k, t_0, x_0)$ is denoted by $\Opt(k, t_0, x_0)$\label{Optkt0x0}.
\end{enumerate}
\end{defi}

\begin{remk}
\label{RemkControlSyst}
If $X\subset\R^d$, the metric derivative of a curve $\gamma \in \Lip(\mathbb R_+, X)$ coincides with the norm of the usual derivative when it exists (cf.\ e.g.\ \cite[Remark 1.1.3]{Ambrosio2005Gradient}), and one obtains that a curve $\gamma$ is $k$-admissible if and only if there exists a measurable function $u: \mathbb R_+ \to \overline B_{d}$ such that
\begin{equation}
\label{AdmissibleIsControlSystem}
\dot\gamma(t) = k(t, \gamma(t)) u(t).
\end{equation}
System \eqref{AdmissibleIsControlSystem} can be seen as a control system, where $\gamma(t) \in X \subset \R^d$ is the state and $u(t) \in \overline B_{d}$ is the control input. This point of view allows one to formulate \eqref{EqMinimalTime} as an optimal control problem, justifying the terminology used in this paper. Optimal control techniques are a key point in the study of equilibria for $\MFG(X, \Gamma, K)$ carried out in the sequel.
\end{remk}

The relation between the optimal control problem $\OCP(X, \Gamma, k)$ and the mean field game $\MFG\allowbreak(X,\allowbreak \Gamma,\allowbreak K)$ is that, given $K: \mathcal P(X) \times X \to \mathbb R_+$, players from the mean field game $\MFG(X, \Gamma, K)$ solve the optimal control problem $\OCP(X, \Gamma, k)$ with $k: \mathbb R_+ \times X \to \mathbb R_+$ given by $k(t, x) = K(m_t, x)$, where $m_t$ is the distribution of players at time $t$. Using this relation, one can provide the definition of an equilibrium for $\MFG(X, \Gamma, K)$.

\begin{defi}[Equilibrium of $\MFG(X, \Gamma, K)$]
\label{DefiEquilibriumMFG}
Let $(X,\dist)$ be a complete and separable metric spa\-ce, $\Gamma \subset X$ be non-empty and closed, and $K: \mathcal P(X) \times X \to \mathbb R_+$ be continuous. Let $m_0 \in \mathcal P(X)$. A measure $Q \in \mathcal P(\mathcal C_X)$ is said to be a \emph{Lagrangian equilibrium} (or simply \emph{equilibrium}) of $\MFG(X, \Gamma,\allowbreak K)$ with initial condition $m_0$ if ${e_0}_\# Q = m_0$ and $Q$-almost every $\gamma \in \mathcal C_X$ is an optimal curve for $(k, 0, \gamma(0))$, where $k: \mathbb R_+ \times X \to \mathbb R_+$ is defined by $k(t, x) = K({e_t}_\# Q, x)$.
\end{defi}

In order to simplify the notations in the sequel, given $K: \mathcal P(X) \times X \to \mathbb R_+$ and $m: \mathbb R_+ \to \mathcal P(X)$, we define $k: \mathbb R_+ \times X \to \mathbb R_+$ by $k(t, x) = K(m_t, x)$ for $(t, x) \in \mathbb R_+ \times X$ and say that $\gamma \in \Lip(\mathbb R_+, X)$ is $m$-admissible for $\MFG(X, \Gamma, K)$ if it is $k$-admissible for $\OCP(X, \Gamma, k)$, and denote $\Adm(k)$ simply by $\Adm(m)$\label{DefiAdmm}. We also say that, given $(t_0, x_0) \in \mathbb R_+ \times X$, $\gamma \in \Lip(\mathbb R_+, X)$ is an optimal trajectory for $(m, t_0, x_0)$ if it is an optimal trajectory for $(k, t_0, x_0)$ for the optimal control problem $\OCP(X, \Gamma, k)$, and denote the set of optimal trajectories for $(m, t_0, x_0)$ by $\Opt(m, t_0, x_0)$\label{Optmt0x0}. Given $Q \in \mathcal P(\mathcal C_X)$, we also consider the time-dependent measure $\mu^Q: \mathbb R_+ \to \mathcal P(X)$ given by $\mu^Q_t = {e_t}_\# Q$, and denote $\Adm(\mu^Q)$ and $\Opt(\mu^Q, t_0, x_0)$ simply by $\Adm(Q)$\label{DefiAdmQ} and $\Opt(Q, t_0, x_0)$\label{OptQt0x0}, respectively.

\begin{remk}
\label{RemkOptimalTrajectoriesRemainStoppedAfterFinalTime}
Since players do not necessarily arrive at the target set $\Gamma$ all at the same time, the behavior of players who have not yet arrived may be influenced by the players who already arrived at $\Gamma$ since $K$ is not necessarily local, i.e., $K(m_t, x)$ may depend on the measure $m_t$ on points other than only $x$. On the other hand, after arriving at $\Gamma$, players are no longer submitted to the time minimization criterion, and thus their trajectory might in principle be arbitrary after their arrival time. In order to avoid ambiguity, we have decided to fix the behavior of players who already arrived at $\Gamma$ by saying that a trajectory is optimal only when it remains at its arrival position after its arrival time. Similarly, we assume that optimal trajectories starting at a time $t_0 > 0$ remain constant on the interval $[0, t_0]$.
\end{remk}

The study of $\OCP(X, \Gamma, k)$ and $\MFG(X, \Gamma, K)$ carried out in this paper requires some assumptions on $X$, $\Gamma$, $K$ and $k$. For simplicity, we state all such assumptions here, and refer to them wherever needed in the sequel.

\begin{hypo}
\label{MainHypo}\mbox{}
\begin{enumerate}
\item\label{HypoXCompact} The metric space $(X,\dist)$ is compact and $\Gamma \subset X$ is non-empty and closed.
\item\label{Hypok1Lip} The function $k: \mathbb R_+ \times X \to \mathbb R_+$ is Lipschitz continuous and there exist $K_{\min}, K_{\max} \in \mathbb R_+^\ast$ such that, for all $(t, x) \in \mathbb R_+ \times X$, one has $k(t, x) \in [K_{\min}, K_{\max}]$.
\item\label{HypoXDist} There exists $D > 0$ such that, for every $x, y \in X$, there exist $T \in [0, D\dist(x, y)]$ and $\gamma \in \Lip([0, T],\allowbreak X)$ such that $\gamma(0) = x$, $\gamma(T) = y$, and $\abs{\dot\gamma}(t) = 1$ for almost every $t \in [0, T]$.
\item\label{HypoXOverlineOmega} One has $X = \overline\Omega$ for $\Omega \subset \mathbb R^d$ open, bounded, and non-empty.
\item\label{HypoExitOnPartialOmega} \ref{HypoXOverlineOmega} holds and $\Gamma = \partial\Omega$.
\item\label{HypokGlobLip} \ref{HypoXOverlineOmega} holds and $k \in \mathcal \Lip(\mathbb R \times \mathbb R^d, \mathbb R_+)$.
\item\label{HypokGlobC11} \ref{HypoXOverlineOmega} holds and $k \in \mathcal C^{1, 1}(\mathbb R \times \mathbb R^d, \mathbb R_+)$.
\item\label{HypoPartialOmegaESP} \ref{HypoXOverlineOmega} holds and $\Omega$ satisfies the \emph{uniform exterior sphere property}: there exists $r > 0$ and $C \subset \mathbb R^d \setminus \Omega$ such that
\[\mathbb R^d \setminus \Omega = \bigcup_{x \in C} \overline B_d(x, r).\]
\item\label{HypoK2Lip} The function $K: \mathcal P(X) \times X \to \mathbb R_+$ is Lipschitz continuous and there exist $K_{\min}, K_{\max} \in \mathbb R_+^\ast$ such that, for all $(\mu, x) \in \mathcal P(X) \times X$, one has $K(\mu, x) \in [K_{\min}, K_{\max}]$.
\item\label{HypoKConvolution} \ref{HypoXOverlineOmega} holds and $K: \mathcal P(\overline\Omega) \times \mathbb R^d$ is given by $K = g \circ E$, where $g \in \mathcal C^{1, 1}(\mathbb R_+, \mathbb R_+^\ast)$ and $E: \mathcal P(\overline\Omega) \times \mathbb R^d \to \mathbb R_+$ is given by
\begin{equation}
\label{DefiE}
E(\mu, x) = \int_{\overline\Omega} \chi(x - y) \eta(y) \diff \mu(y),
\end{equation}
with $\chi \in \mathcal C^{1, 1}(\mathbb R^d, \mathbb R_+)$, $\eta \in \mathcal C^{1, 1}(\mathbb R^d, \mathbb R_+)$, $\chi$ and $\eta$ bounded, and $\eta(x) = 0$ and $\nabla \eta(x) = 0$ for every $x \in \partial\Omega$.
\end{enumerate}
\end{hypo}

Hypotheses \ref{MainHypo}\ref{HypoXCompact}--\ref{HypoXDist} are the standard assumptions needed in all the results from Section \ref{SecOCP} for studying $\OCP(X, \Gamma, k)$. Note that \ref{HypoXDist} provides a relation between the distance $\dist(x, y)$ and the length of curves from $x$ to $y$ in $X$, stating that the former is, up to a constant, an upper bound on the latter. This assumption is satisfied when the geodesic metric induced by $\dist$ is equivalent to $\dist$ itself. In particular, it is satisfied if $X$ is a length space. Moreover, \ref{HypoXDist} implies that $X$ is path-connected.

One uses Hypotheses \ref{MainHypo}\ref{HypoXOverlineOmega}--\ref{HypoPartialOmegaESP} in Section \ref{SecOCP} to obtain more properties of $\OCP(X, \Gamma, k)$, such as the facts that the value function satisfies a Hamilton--Jacobi equation and is semiconcave, and that optimal trajectories can be described via the Pontryagin Maximum Principle. Assumption \ref{HypoXOverlineOmega} allows one to work on a subset of an Euclidean space, making it easier to give a meaning to the Hamilton--Jacobi equation, for instance.

Assumption \ref{HypoExitOnPartialOmega} states that the goal of an agent is to leave the domain $\Omega$ through any part of its boundary. This simplifying assumption is first used when applying Pontryagin Maximum Principle to $\OCP(X, \Gamma, k)$, and, even though it is not strictly needed at this point, as remarked in the beginning of Section \ref{SecPMP}, it is important in Section \ref{SecNormGradientAndOptimalControl} to deduce the semiconcavity of the value function in Proposition \ref{PropVarphiSemiconcave}. Indeed, without assumption \ref{HypoExitOnPartialOmega}, one must take into account in $\OCP(X, \Gamma, k)$ the state constraint $\gamma(t) \in \overline\Omega$, and value functions of optimal control problems with state constraints may fail to be semiconcave (see, e.g., \cite[Example 4.4]{Cannarsa2007Lipschitz}).

Notice that \ref{HypoExitOnPartialOmega} is not needed to obtain existence of equilibria for $\MFG(X, \Gamma, K)$ in Section \ref{SecExistence}, nor for establishing the Hamilton--Jacobi equation on the value function $\varphi$ of $\OCP(X, \Gamma, k)$ in Section \ref{SecHJ}. However, to prove that $m$ satisfies the continuity equation from \eqref{SystMFG} and hence complete the description of equilibria by the MFG system in Section \ref{SecCharacterization}, one uses the characterization of optimal trajectories of $\OCP(X, \Gamma, k)$ in terms of the normalized gradient of $\varphi$ carried out in Section \ref{SecNormGradientAndOptimalControl}, the semiconcavity of $\varphi$ being a key ingredient in the proof of the main result of that section, Theorem \ref{TheoNormGradIffSingleton}.

Assumption \ref{HypokGlobLip} is not restrictive with respect to \ref{Hypok1Lip}, stating that one considers $k$ to be extended to the whole space $\mathbb R \times \mathbb R^d$ in a Lipschitz manner, and is included in the list of hypotheses only for simplifying the statements of the results. A classical technique to obtain Lipschitz extensions of Lipschitz continuous functions is by inf-convolution (see, e.g., \cite{Hiriart-Urruty1980Extension}). Assumptions \ref{HypokGlobC11} and \ref{HypoPartialOmegaESP} are important to obtain the semiconcavity of the value function of $\OCP(X, \Gamma, k)$ in Section \ref{SecNormGradientAndOptimalControl}.

Concerning Hypotheses \ref{MainHypo}\ref{HypoK2Lip} and \ref{HypoKConvolution}, they are used in Sections \ref{SecExistence} and \ref{SecCharacterization} to study $\MFG(X, \Gamma, K)$. Assumption \ref{HypoK2Lip} is the counterpart of \ref{Hypok1Lip}, with \ref{Hypok1Lip} being obtained from \ref{HypoK2Lip} when $k$ is given by $k(t, x) = K({e_t}_{\#} Q, x)$ (see Corollary \ref{CorokFromK}). Assumption \ref{HypoKConvolution}, which is reasonable for modeling reasons (the function $g$ is typically required to be non-increasing, but this is not crucial for the mathematical results we present), will be important in order for this $k$ to satisfy also \ref{HypokGlobC11} when $Q$ is an equilibrium (see Proposition \ref{PropK0C11}). Notice that \ref{HypoKConvolution} implies \ref{HypoK2Lip}; more precisely, we have the following result.

\begin{prop}
\label{PropFConvolution}
Let $\Omega \subset \mathbb R^d$ be open, bounded, and non-empty, $X = \overline\Omega$, $g \in \Lip(\mathbb R_+, \mathbb R_+^\ast)$, $\chi \in \Lip(\mathbb R^d,\allowbreak \mathbb R_+)$, $\eta \in \Lip(\overline\Omega, \mathbb R_+)$, and define $K: \mathcal P(\overline\Omega) \times \overline\Omega \to \mathbb R_+$ by
\[
K(\mu, x) = g\left[\int_{\overline\Omega} \chi(x - y) \eta(y) \diff\mu(y)\right].
\]
Then $K$ satisfies Hypothesis \ref{MainHypo}\ref{HypoK2Lip}.
\end{prop}

\begin{proof}
Let $E: \mathcal P(\overline\Omega) \times \overline\Omega \to \mathbb R_+$ be defined by \eqref{DefiE}. Let $M_1 > 0$ be a Lipschitz constant for $\chi$ and $\eta$ and $M_2 = \max\{\max_{x \in \overline\Omega} \eta(x),\allowbreak \max_{(x, y) \in {\overline\Omega}^2} \chi(x - y)\}$. Then, for every $y \in \overline\Omega$, the function $z \mapsto \chi(y - z) \eta(z)$ is $2 M_1 M_2$-Lipschitz continuous on $\overline\Omega$ and, in particular, $z \mapsto \frac{\chi(y - z) \eta(z)}{2 M_1 M_2}$ is $1$-Lipschitz continuous on $\overline\Omega$.

Let us first show that $K \in \Lip(\mathcal P(\overline\Omega) \times \overline\Omega, \mathbb R_+)$. Take $\mu, \nu \in \mathcal P(\overline\Omega)$ and $x, y \in \overline\Omega$. Then
\begin{align*}
& E(\mu, x) - E(\nu, y) \\
{} = {} & \int_{\overline\Omega} \chi(x - z) \eta(z) \diff \mu(z) - \int_{\overline\Omega} \chi(y - z) \eta(z) \diff \nu(z) \displaybreak[0] \\
{} = {} & \int_{\overline\Omega} \left[\chi(x - z) - \chi(y - z)\right] \eta(z) \diff\mu(z) + 2 M_1 M_2 \int_{\overline\Omega} \frac{\chi(y - z) \eta(z)}{2 M_1 M_2} \diff(\mu - \nu)(z) \displaybreak[0] \\
{} \leq {} & M_1 M_2 \abs{x - y} + 2 M_1 M_2 W_1(\mu, \nu),
\end{align*}
where we use \eqref{WassersteinDual}. A similar computation starting from $E(\nu, y) - E(\mu, x)$ completes the proof of the fact that $E \in \Lip(\mathcal P(\overline\Omega) \times \overline\Omega, \mathbb R_+)$, and then $K = g \circ E \in \Lip(\mathcal P(\overline\Omega) \times \overline\Omega, \mathbb R_+)$ since $g \in \Lip(\mathbb R_+, \mathbb R_+^\ast)$.

Finally, it follows from the definition of $E$ that $E(\mu, x) \leq M_2^2$ for all $(\mu, x) \in \mathcal P(\overline\Omega) \times \overline\Omega$, and thus
\begin{align*}
K_{\min} & = \inf_{(\mu, x) \in \mathcal P(\overline\Omega) \times \overline\Omega} g \circ E(\mu, x) \geq \min_{x \in [0, M_2^2]} g(x) > 0, \displaybreak[0] \\
K_{\max} & = \sup_{(\mu, x) \in \mathcal P(\overline\Omega) \times \overline\Omega} g \circ E(\mu, x) \leq \max_{x \in [0, M_2^2]} g(x) < +\infty,
\end{align*}
as required.
\end{proof}

\section{Preliminary study of the minimal-time optimal control problem}
\label{SecOCP}

In this section, we provide several properties for the optimal control problem $\OCP(X, \Gamma, k)$. Minimal-time optimal control problems are a classical subject in the optimal control literature (see, e.g., \cite{Cannarsa2004Semiconcave, Pontryagin1962Mathematical, Clarke1990Optimization}). Most works consider the case of an autonomous control system with smooth dynamics in the Euclidean space or a smooth manifold, dealing with the non-autonomous case by a classical state-augmentation technique. Since the study of $\MFG(X, \Gamma, K)$ requires some properties of $\OCP(X, \Gamma, k)$ in less regular cases (for instance when $k$ is only Lipschitz continuous), we provide here a detailed presentation of the properties of $\OCP(X, \Gamma, k)$ in order to highlight which hypotheses are required for each result.

The proofs of classical properties for $\OCP(X, \Gamma, k)$ are provided when the particular structure of $\OCP(X, \Gamma, k)$ allows for simplifications with respect to classical proofs in the literature, and omitted otherwise. Several properties presented here are new and rely on the particular structure of $\OCP(X, \Gamma, k)$. This is the case, for instance, of the lower bound on the time variation of the value function (Proposition \ref{PropMonotoneOptimalTime} and Lemma \ref{LemmP0GreaterThanMinusOne}) and the characterization of optimal trajectories in terms of the normalized gradient presented in Section \ref{SecNormGradientAndOptimalControl}.

We start in Section \ref{SecValueFunction} by studying the value function $\varphi$ corresponding to $\OCP(X, \Gamma, k)$. We then specialize to the case of an optimal control problem on $\overline\Omega \subset \mathbb R^d$ and establish a Hamilton--Jacobi equation for $\varphi$ in Section \ref{SecHJ}, before proving further properties of optimal trajectories obtained from Pontryagin Maximum Principle in Section \ref{SecPMP}. We conclude this section by a characterization of the optimal control in terms of a normalized gradient in Section \ref{SecNormGradientAndOptimalControl}, which also considers the continuity of the normalized gradient.

\subsection{Elementary properties of the value function}
\label{SecValueFunction}

As stated in Remark \ref{RemkControlSyst}, at least in the case where $X$ is a subset of a normed vector space, admissible trajectories for some $k: \mathbb R_+ \times X \to \mathbb R_+$ can be seen as trajectories of the control system \eqref{AdmissibleIsControlSystem} and the minimization problem \eqref{EqMinimalTime}, as an optimal control problem. Inspired by this interpretation, we make use of an usual strategy in optimal control, namely that of considering the value function associated with the optimal control problem. We start by recalling the classical definition of the value function for $\OCP(X, \Gamma, k)$.

\begin{defi}
\label{DefiVarphi}
Let $(X, \Gamma, k)$ be as in Definition \ref{DefiOCP}. The \emph{value function} of $\OCP(X, \Gamma, k)$ is the function $\varphi: \mathbb R_+ \times X \to \mathbb R_+ \cup \{+\infty\}$ defined by
\begin{equation}
\label{EqDefiVarphi}
\varphi(t, x) = \inf_{\substack{\gamma \in \Adm(k) \\ \gamma(t) = x}} \tau(t, \gamma).
\end{equation}
\end{defi}

The goal of this section is to provide some properties of $\varphi$, in particular its Lipschitz continuity and the fact that it satisfies a Hamilton--Jacobi equation. We gather in the next proposition some elementary properties of $\varphi$.

\begin{prop}
\label{Prop1Varphi}
Consider the optimal control problem $\OCP(X, \Gamma, k)$ and assume that Hypotheses \ref{MainHypo}\ref{HypoXCompact}--\ref{HypoXDist} hold.
\begin{enumerate}
\item\label{PropBoundTau} There exists $M > 0$ such that, for every $(t, x) \in \mathbb R_+ \times X$, one has $\varphi(t, x) \leq M$.
\item\label{PropExistOptim} For every $(t, x) \in \mathbb R_+ \times X$, there exists an optimal trajectory $\gamma \in \Lip(\mathbb R_+, X)$ for $(k, t, x)$. In particular, the infimum in \eqref{EqDefiVarphi} is attained at $\gamma$.
\end{enumerate}
\end{prop}

The proof of Proposition \ref{Prop1Varphi} follows from standard techniques: the upper bound $M$ in \ref{PropBoundTau} can be taken as $M = \frac{D \diam(X)}{K_{\min}}$, where $\diam(X) = \sup_{x, y \in X}\dist(x, y) < +\infty$ is the diameter of $X$, and the existence of an optimal trajectory proven using compactness of minimizing sequences.

We now turn to the proof of Lipschitz continuity of $\varphi$. We first show, in Proposition \ref{PropTauLipT}, that $\varphi$ is Lipschitz continuous with respect to $t$, uniformly with respect to $x$, before completing the proof of Lipschitz continuity of $\varphi$ in Proposition \ref{PropTauLip}.

\begin{prop}
\label{PropTauLipT}
Consider the optimal control problem $\OCP(X, \Gamma, k)$ and assume that Hypotheses \ref{MainHypo}\ref{HypoXCompact}--\ref{HypoXDist} hold. Then the map $t \mapsto \varphi(t, x)$ is Lipschitz continuous, uniformly with respect to $x \in X$.
\end{prop}

\begin{proof}
Let $t_1, t_2 \in \mathbb R_+$ and assume, with no loss of generality, that $t_1 \leq t_2$. We prove the result by showing that one has both
\begin{equation}
\label{TauLipTIneq1}
\varphi(t_1, x) - \varphi(t_2, x) \leq t_2 - t_1
\end{equation}
and
\begin{equation}
\label{TauLipTIneq2}
\varphi(t_2, x) - \varphi(t_1, x) \leq C (t_2 - t_1)
\end{equation}
for some constant $C > 0$ independent of $t_1$, $t_2$, and $x$.

Let us first prove \eqref{TauLipTIneq1}. Let $\gamma_2 \in \Opt(k, t_2, x)$. Define $\gamma_1 \in \Lip(\mathbb R_+, X)$ by
\[
\gamma_1(s) = 
\begin{dcases*}
x, & if $s \leq t_2$, \\
\gamma_2(s), & if $s > t_2$.
\end{dcases*}
\]
Then $\gamma_1(t) = x$ for $t \in [0, t_1]$, $\gamma_1(t) = \gamma_2(t) = \gamma_2(t_2 + \varphi(t_2, x)) \in \Gamma$ for $t > t_2 + \varphi(t_2, x)$, $\abs{\dot\gamma_1}(s) = 0$ for $s < t_2$, and $\abs{\dot\gamma_1}(s) \leq k (s, \gamma_1(s))$ for almost every $s > t_2$, which proves that $\gamma_1 \in \Adm(k)$ and $\varphi(t_1, x) \leq \tau(t_1, \gamma_1) \leq t_2 - t_1 + \varphi(t_2, x)$, yielding \eqref{TauLipTIneq1}.

We now turn to the proof of \eqref{TauLipTIneq2}. Let $M > 0$ be as in Proposition \ref{Prop1Varphi}\ref{PropBoundTau} and take $\gamma_1 \in \Opt(k, t_1, x)$. Notice that, thanks to Hypothesis \ref{MainHypo}\ref{Hypok1Lip}, the map
\[
\mathbb R_+ \times \mathbb R_+ \ni (t, s) \mapsto \frac{k(t, \gamma_1(s))}{k(s, \gamma_1(s))}
\]
is lower bounded by $\frac{K_{\min}}{K_{\max}}$, upper bounded by $\frac{K_{\max}}{K_{\min}}$, and globally Lipschitz continuous. Let $M_0 > 0$ be a Lipschitz constant for this map. Let $\phi: [t_2, +\infty) \to \mathbb R$ be the unique function satisfying
\[
\left\{
\begin{aligned}
\dot\phi(t) & = \frac{k(t, \gamma_1(\phi(t)))}{k(\phi(t), \gamma_1(\phi(t)))} \\
\phi(t_2) & = t_1.
\end{aligned}
\right.
\]
Notice that $\phi$ is strictly increasing and maps $[t_2, +\infty)$ onto $[t_1, +\infty)$, its inverse $\phi^{-1}$ being defined on $[t_1, +\infty)$. Let $\gamma_2: \mathbb R_+ \to X$ be given by
\[
\gamma_2(s) = 
\begin{dcases*}
x, & if $s \leq t_2$, \\
\gamma_1(\phi(s)), & if $s > t_2$.
\end{dcases*}
\]
Then $\gamma_2 \in \Lip(\mathbb R_+, X)$, $\gamma_2(t) = x$ for every $t \in [0, t_2]$, $\gamma_2(t) = \gamma_2(\phi^{-1}(t_1 + \varphi(t_1, x))) \in \Gamma$ for every $t \geq \phi^{-1}(t_1 + \varphi(t_1, x))$, $\abs{\dot\gamma_2}(s) = 0$ whenever $0 < s < t_2$ or $s > \phi^{-1}(t_1 + \varphi(t_1, x))$, and, for almost every $s \geq t_2$, one has $\abs{\dot\gamma_2}(s) = \abs{\dot\phi(s)} \abs{\dot\gamma_1}(\phi(s)) \leq k(s, \gamma_2(s))$. Hence $\gamma_2 \in \Adm(k)$ and $t_2 + \varphi(t_2, x) \leq \phi^{-1}(t_1 + \varphi(t_1, x))$, i.e., $\phi(t_2 + \varphi(t_2, x)) \leq t_1 + \varphi(t_1,\allowbreak x)$. Yet, for every $\sigma \geq 0$,
\[
\phi(t_2 + \sigma) - (t_2 + \sigma) = t_1 - t_2 + \int_{t_2}^{t_2 + \sigma} \left[\frac{k(s, \gamma_1(\phi(s)))}{k(\phi(s), \gamma_1(\phi(s)))} - 1\right] \diff s,
\]
and thus
\[
\abs{\phi(t_2 + \sigma) - (t_2 + \sigma)} \leq \abs{t_1 - t_2} + M_0 \int_{t_2}^{t_2 + \sigma} \abs{\phi(s) - s} \diff s,
\]
which yields, by Gronwall's inequality, that $\abs{\phi(t_2 + \sigma) - (t_2 + \sigma)} \leq \abs{t_1 - t_2} e^{M_0 \sigma}$. Then
\[
\abs{\phi(t_2 + \varphi(t_2, x)) - (t_2 + \varphi(t_2, x))} \leq \abs{t_1 - t_2} e^{M M_0},
\]
which proves that
\[
t_1 + \varphi(t_1, x) \geq \phi(t_2 + \varphi(t_2, x)) \geq t_2 + \varphi(t_2, x) - e^{M M_0}(t_2 - t_1),
\]
and thus
\[
\varphi(t_2, x) - \varphi(t_1, x) \leq (e^{M M_0} - 1)(t_2 - t_1),
\]
which concludes the proof of \eqref{TauLipTIneq2}.
\end{proof}

\begin{prop}
\label{PropTauLip}
Consider the optimal control problem $\OCP(X, \Gamma, k)$ and assume that Hypotheses \ref{MainHypo}\ref{HypoXCompact}--\ref{HypoXDist} hold. Then $\varphi$ is Lipschitz continuous on $\mathbb R_+ \times X$.
\end{prop}

\begin{proof}
Let $D$ be as in Hypothesis \ref{MainHypo}\ref{HypoXDist}, $x_1, x_2 \in X$, and $t_1, t_2 \in \mathbb R_+$. According to Hypothesis \ref{MainHypo}\ref{HypoXDist}, there exist $T_0 \in [0, D\dist(x_1, x_2)]$ and $\gamma_0 \in \Lip([0, T_0], X)$ such that $\gamma_0(0) = x_1$, $\gamma_0(T_0) = x_2$, and $\abs{\dot\gamma_0}(s) = 1$ for almost every $s \in [0, T_0]$. Denote by $M > 0$ the Lipschitz constant of the map $t \mapsto \varphi(t, x)$, which is independent of $x \in X$ according to Proposition \ref{PropTauLipT}.

Set $T = \frac{T_0}{K_{\min}}$ and $\sigma_2 = \max\{t_2, T\}$. Let $\gamma_2 \in \Opt(k, \sigma_2, x_2)$ and define $\gamma_1 \in \Lip(\mathbb R_+, X)$ by
\[
\gamma_1(t) = 
\begin{dcases*}
x_1, & if $t \leq \sigma_2 - T$, \\
\gamma_0(K_{\min}(t + T - \sigma_2)), & if $\sigma_2 - T < t \leq \sigma_2$, \\
\gamma_2(t), & if $t > \sigma_2$.
\end{dcases*}
\]
Then $\gamma_1(t) = x_1$ for every $t \in [0, \sigma_2 - T]$, $\gamma_1(t) = \gamma_1(\sigma_2 + \varphi(\sigma_2, x_2)) \in \Gamma$ for every $t \geq \sigma_2 + \varphi(\sigma_2, x_2)$, and $\abs{\dot\gamma_1}(t) \leq k(t, \gamma_1(t))$ for almost every $t \geq 0$, which proves that $\gamma_1 \in \Adm(k)$ and $\varphi(\sigma_2 - T, x_1) \leq T + \varphi(\sigma_2, x_2)$. Hence, using Proposition \ref{PropTauLipT}, one obtains that
\begin{align*}
 & \varphi(t_1, x_1) - \varphi(t_2, x_2) \\
{} \leq {} & M \abs{t_1 - \sigma_2 + T} + T + M \abs{\sigma_2 - t_2} \displaybreak[0] \\
{} \leq {} & M \abs{t_1 - t_2} + M \abs{t_2 - \sigma_2 + T} + M\abs{\sigma_2 - t_2} + T = M \abs{t_1 - t_2} + (M + 1) T \displaybreak[0] \\
{} = {} & M \abs{t_1 - t_2} + \frac{M+1}{K_{\min}}T_0 \leq M \abs{t_1 - t_2} + D \frac{M+1}{K_{\min}}\dist(x_1, x_2).
\end{align*}
One can bound $\varphi(t_2, x_2) - \varphi(t_1, x_1)$ in exactly the same manner by exchanging the roles of $(t_1, x_1)$ and $(t_2, x_2)$ and replacing $\gamma_0$ by $\gamma_0(T_0 - \cdot)$.
\end{proof}

Another important property of the value function is presented in the next proposition, and is useful for providing a lower bound on the time derivative of the value function when it exists.

\begin{prop}
\label{PropMonotoneOptimalTime}
Consider the optimal control problem $\OCP(X, \Gamma, k)$ and assume that Hypotheses \ref{MainHypo}\ref{HypoXCompact}--\ref{HypoXDist} hold. There exists $c > 0$ such that, for every $x \in X$ and $t_1, t_2 \in \mathbb R_+$ with $t_1 \neq t_2$, one has
\begin{equation}
\label{DtTauQGreaterThanMinusOne}
\frac{\varphi(t_2, x) - \varphi(t_1, x)}{t_2 - t_1} \geq c-1.
\end{equation}
In particular, if $t_1 < t_2$, then $t_1 + \varphi(t_1, x) < t_2 + \varphi(t_2, x)$.
\end{prop}

The last statement of the proposition means that, if two optimal trajectories start at the same point $x$ on different times, the one which started sooner will arrive first at $\Gamma$.

\begin{proof}
It suffices to prove \eqref{DtTauQGreaterThanMinusOne} in the case $t_2 > t_1$, the other case being obtained by exchanging the role of $t_1$ and $t_2$. We then assume from now on that $t_2 > t_1$. Notice also that, if \eqref{DtTauQGreaterThanMinusOne} holds, then, for $t_2 > t_1$, one has $\varphi(t_2, x) - \varphi(t_1, x) \geq (c-1) (t_2 - t_1) > -(t_2 - t_1)$, which yields $t_1 + \varphi(t_1, x) < t_2 + \varphi(t_2, x)$.

Let $\gamma_2 \in \Opt(k, t_2, x)$ and $\phi: \mathbb [t_1, +\infty) \to \mathbb R$ be the unique function satisfying
\[
\left\{
\begin{aligned}
\dot\phi(t) & = \frac{k(t, \gamma_2(\phi(t)))}{k(\phi(t), \gamma_2(\phi(t)))} \\
\phi(t_1) & = t_2.
\end{aligned}
\right.
\]
Notice that $\phi$ is strictly increasing and maps $[t_1, +\infty)$ onto $[t_2, +\infty)$, its inverse $\phi^{-1}$ being defined on $[t_2, +\infty)$. Moreover, since $\psi(t) = t$ is a solution of $\dot\psi(t) = \frac{k(t, \gamma_2(\psi(t)))}{k(\psi(t), \gamma_2(\psi(t)))}$ and $\phi(t_1) = t_2 > t_1 = \psi(t_1)$, one obtains that $\phi(t) > \psi(t) = t$ for every $t \in \mathbb R_+$.

Let $M > 0$ be as in Proposition \ref{Prop1Varphi}\ref{PropBoundTau} and $M_0 > 0$ be a Lipschitz constant for $k$. Notice that, for every $t \in \mathbb R_+$, one has
\[
\dot\phi(t) - 1 = \frac{k(t, \gamma_2(\phi(t))) - k(\phi(t), \gamma_2(\phi(t)))}{k(\phi(t), \gamma_2(\phi(t)))}.
\]
Fix $t_0 \in [t_1, t_1 + M]$. Then, for every $t \in [t_1, t_0]$,
\[
\phi(t_0) - t_0 = \phi(t) - t + \int_t^{t_0} \frac{k(s, \gamma_2(\phi(s))) - k(\phi(s), \gamma_2(\phi(s)))}{k(\phi(s), \gamma_2(\phi(s)))} \diff s,
\]
which yields
\[
\abs{\phi(t) - t} \leq \abs{\phi(t_0) - t_0} + \frac{M_0}{K_{\min}}\int_t^{t_0} \abs{\phi(s) - s} \diff s,
\]
and thus, by Gronwall's inequality, for every $t \in [t_1, t_0]$,
\[
\abs{\phi(t) - t} \leq \abs{\phi(t_0) - t_0} e^{\frac{M_0}{K_{\min}} (t_0 - t)}.
\]
In particular, using the fact that $\phi(t) > t$ for every $t \in \mathbb R_+$, one obtains from the above, setting $t_0 = t_1 + \varphi(t_1, x) \in [t_1, t_1 + M]$ and $t = t_1$, that
\begin{equation}
\label{LowerBoundVarphiMinusT}
\phi(t_1 + \varphi(t_1, x)) - \left(t_1 + \varphi(t_1, x)\right) \geq c (t_2 - t_1),
\end{equation}
where $c = e^{-\frac{M M_0}{K_{\min}}}$.

Let $\gamma_1 \in \Lip(\mathbb R_+, X)$ be defined by $\gamma_1(t) = \gamma_2(\phi(t))$ for $t > t_1$ and $\gamma_1(t) = x$ for $t \leq t_1$. Then, for almost every $t \in [t_1, +\infty)$, one has $\abs{\dot\gamma_1}(t) = \abs{\dot\phi(t)} \abs{\dot\gamma_2}(\phi(t)) \leq k(t, \gamma_1(t))$, which proves that $\gamma_1 \in \Adm(k)$. Moreover, $\gamma_1(t_1) = x$ and $\gamma_1(\phi^{-1}(t_2 + \varphi(t_2, x))) \in \Gamma$, which proves that the first exit time after $t_1$ of $\gamma_1$ satisfies $\tau(t_1, \gamma_1) \leq \phi^{-1}(t_2 + \varphi(t_2, x))$, and hence $\varphi(t_1, x) \leq \phi^{-1}(t_2 + \varphi(t_2, x)) - t_1$. Thus $\phi(t_1 + \varphi(t_1,\allowbreak x)) \allowbreak \leq t_2 + \varphi(t_2, x)$, and, using \eqref{LowerBoundVarphiMinusT}, we get that $t_1 + \varphi(t_1, x) + c(t_2 - t_1) \leq t_2 + \varphi(t_2, x)$, yielding \eqref{DtTauQGreaterThanMinusOne}.
\end{proof}

We gather in the next result two further properties of the value function and of optimal trajectories; the proof is straightforward and thus omitted here.

\begin{prop}
\label{Prop2Varphi}
Consider the optimal control problem $\OCP(X, \Gamma, k)$ and assume that Hypotheses \ref{MainHypo}\ref{HypoXCompact}--\ref{HypoXDist} hold. Let $(t_0, x_0) \in \mathbb R_+ \times X$ and $\gamma \in \Opt(k, t_0, x_0)$.
\begin{enumerate}
\item\label{PropFlowPreservesOptimality} For every $h \in [0, \varphi(t_0,\allowbreak x_0)]$, one has $\varphi(t_0 + h, \gamma(t_0 + h)) + h = \varphi(t_0, x_0)$.
\item\label{LemmOptimalAllAlong} Let $t_1 \in [t_0, t_0 + \varphi(t_0, x_0)]$ and consider the trajectory $\widetilde\gamma: \mathbb R_+ \to X$ defined by $\widetilde\gamma(t) = \gamma(t)$ for $t > t_1$ and $\widetilde\gamma(t) = \gamma(t_1)$ for $t \leq t_1$. Then $\widetilde\gamma \in \Opt(k, t_1, \gamma(t_1))$.
\end{enumerate}
\end{prop}

\subsection{Hamilton--Jacobi equation}
\label{SecHJ}

In this section, we collect further results on the value function $\varphi$ under the additional assumption that Hypothesis \ref{MainHypo}\ref{HypoXOverlineOmega} holds. Before establishing the main result of this section, Theorem \ref{MainTheoHJ}, which provides a Hamilton--Jacobi equation for $\varphi$, we recall the definition of superdifferential of a function and obtain a lower bound on the time component of any vector on the superdifferential of $\varphi$ as consequence of Proposition \ref{PropMonotoneOptimalTime}.

\begin{defi}\label{DefiSuperdifferential}
Let $A \subset \mathbb R^d$, $w: A \to \mathbb R$, and $x \in A$. The \emph{superdifferential} of $w$ at $x$ is the set $D^+ w(x)$ defined by
\[
D^+ w(x) = \left\{p \in \mathbb R^d \midsuchthat \limsup_{y \to x} \frac{w(y) - w(x) - p \cdot (y - x)}{\abs{y - x}} \leq 0\right\}.
\]
\end{defi}

\begin{lemm}
\label{LemmP0GreaterThanMinusOne}
Consider the optimal control problem $\OCP(X, \Gamma, k)$ and assume that Hypotheses \ref{MainHypo}\ref{HypoXCompact}--\ref{HypoXOverlineOmega} hold. There exists $c > 0$ such that, for every $(t_0, x_0) \in \mathbb R_+^\ast \times \overline\Omega$, and $(p_0, p_1) \in D^+ \varphi(t_0,\allowbreak x_0)$, one has $p_0 \geq c - 1$.
\end{lemm}

\begin{proof}
Let $c > 0$ be as in Proposition \ref{PropMonotoneOptimalTime}. Since $(p_0, p_1) \in D^+\varphi(t_0, x_0)$, one has
\begin{equation}
\label{P0P1InDPlusVarphi}
\limsup_{\substack{t \to t_0 \\ x \to x_0}} \frac{\varphi(t, x) - \varphi(t_0, x_0) - p_0(t - t_0) - p_1 \cdot (x - x_0)}{\left(\abs{t - t_0}^2 + \abs{x - x_0}^2\right)^{1/2}} \leq 0.
\end{equation}
For $h > 0$ small, one has, using Proposition \ref{PropMonotoneOptimalTime},
\begin{multline*}
\sup_{\substack{t \in (t_0 - h, t_0 + h) \\ x \in B(x_0, h) \\ (t, x) \neq (t_0, x_0)}} \frac{\varphi(t, x) - \varphi(t_0, x_0) - p_0(t - t_0) - p_1 \cdot (x - x_0)}{\left(\abs{t - t_0}^2 + \abs{x - x_0}^2\right)^{1/2}} \\ \geq \sup_{t \in (t_0, t_0 + h)} \frac{\varphi(t, x_0) - \varphi(t_0, x_0) - p_0(t - t_0)}{t - t_0} \geq c - 1 - p_0,
\end{multline*}
and thus, taking the limit as $h \searrow 0$, one obtains from \eqref{P0P1InDPlusVarphi} that $c - 1 - p_0 \leq 0$.
\end{proof}

Our next result provides a Hamilton--Jacobi equation for $\varphi$. Its proof is based on classical techniques on optimal control and is omitted here (see, e.g., \cite[Chapter IV, Proposition 2.3]{Bardi1997Optimal}).

\begin{theo}
\label{MainTheoHJ}
Consider the optimal control problem $\OCP(X, \Gamma, k)$, assume that Hypotheses \ref{MainHypo}\ref{HypoXCompact}\allowbreak--\ref{HypoXOverlineOmega} hold, and let $\Omega$ be as in Hypothesis \ref{MainHypo}\ref{HypoXOverlineOmega} and $\varphi$ be the value function from Definition \ref{DefiVarphi}. Consider the Hamilton--Jacobi equation on $\mathbb R_+ \times \overline\Omega$
\begin{equation}
\label{EqHJB}
-\partial_t \varphi(t, x) + \abs{\nabla_x \varphi(t, x)} k(t, x) - 1 = 0.
\end{equation}
Then $\varphi$ is a viscosity subsolution of \eqref{EqHJB} on $\mathbb R_+ \times \Omega$, a viscosity supersolution of \eqref{EqHJB} on $\mathbb R_+ \times (\overline\Omega \setminus \Gamma)$, and satisfies $\varphi(t, x) = 0$ for $(t, x) \in \mathbb R_+ \times \Gamma$.
\end{theo}

Proposition \ref{PropMonotoneOptimalTime} yields a lower bound on the time derivative of $\varphi$, which can be used to obtain information on the gradient of $\varphi$ and on the optimal control of optimal trajectories thanks to the Hamilton--Jacobi equation \eqref{EqHJB}.

\begin{coro}
\label{Coro1Varphi}
Consider the optimal control problem $\OCP(X, \Gamma, k)$, assume that Hypotheses \ref{MainHypo}\allowbreak\ref{HypoXCompact}\allowbreak--\ref{HypoXOverlineOmega} hold, and let $\Omega$ be as in Hypothesis \ref{MainHypo}\ref{HypoXOverlineOmega} and $(t_0, x_0) \in \mathbb R_+ \times \overline\Omega$.
\begin{enumerate}
\item\label{CoroGradNotZero} If $x_0 \notin \Gamma$ and $\varphi$ is differentiable at $(t_0, x_0)$, then $\partial_t \varphi(t, x) > -1$ and $\nabla_x \varphi(t, x) \neq 0$.

\item\label{CoroAlmostVelocityField} Let $\gamma \in \Opt(k, t_0, x_0)$ and assume that $t \in [t_0, t_0 + \varphi(t_0,\allowbreak x_0))$ is such that $\gamma$ is differentiable at $t$, $\gamma(t) \in \Omega$, and $\varphi$ is differentiable at $(t, \gamma(t))$. Then
\begin{equation}
\label{NablaTauQIsOptimalControl}
\dot\gamma(t) = -k(t, \gamma(t)) \frac{\nabla_x \varphi(t, \gamma(t))}{\abs{\nabla_x \varphi(t, \gamma(t))}}.
\end{equation}
\end{enumerate}
\end{coro}

\begin{proof}
To prove \ref{CoroGradNotZero}, notice first that, for $(t_0, x_0) \in \mathbb R_+ \times (\overline\Omega \setminus \Gamma)$ at which $\varphi$ is differentiable, it follows from \eqref{DtTauQGreaterThanMinusOne} that $\partial_t \varphi(t_0, x_0) \geq c - 1 > -1$. Since $\varphi$ is a viscosity supersolution of \eqref{EqHJB} on $\mathbb R_+ \times (\overline\Omega \setminus \Gamma)$, then
\[
-\partial_t \varphi(t_0, x_0) + \abs{\nabla_x \varphi(t_0, x_0)} k(t_0, x_0) - 1 \geq 0
\]
(see, e.g., \cite[Corollary I.6]{Crandall1983Viscosity}). Since $\partial_t \varphi(t_0, x_0) > -1$ and $k(t_0, x_0) \geq K_{\min} > 0$, one obtains that $\abs{\nabla_x \varphi(t_0, x_0)} > 0$, i.e., $\nabla_x \varphi(t_0, x_0) \neq 0$.

In order to prove \ref{CoroAlmostVelocityField}, notice that, since $\gamma\in \Adm(k)$ and $\gamma$ is differentiable at $t$, we can write $\dot\gamma(t) = k(t, \gamma(t)) u$ for a certain $u \in \overline B_d$. Applying Proposition \ref{Prop2Varphi}\ref{PropFlowPreservesOptimality}, one obtains, for small $h$,
\[
\varphi(t + h, \gamma(t + h)) + h = \varphi(t, \gamma(t)).
\]
Differentiating with respect to $h$ at $h = 0$ yields
\[
\partial_t \varphi(t, \gamma(t)) + \nabla_x \varphi(t, \gamma(t)) \cdot \dot\gamma(t) + 1 = 0.
\]
On the other hand, since $\varphi$ is differentiable at $(t, \gamma(t))$, \eqref{EqHJB} holds pointwisely at  $(t, \gamma(t))$ (see, e.g., \cite[Corollary I.6]{Crandall1983Viscosity}). Then, comparing the two expressions, we obtain
\[
\nabla_x \varphi(t, \gamma(t)) \cdot \dot\gamma(t) + \abs{\nabla_x \varphi(t, \gamma(t))} k(t, \gamma(t)) = 0.
\]
Using $\dot\gamma(t) = k(t, \gamma(t)) u$ and $k(t, \gamma(t)) > 0$, we get
\[
\nabla_x \varphi(t, \gamma(t)) \cdot u = - \abs{\nabla_x \varphi(t, \gamma(t))}.
\]
Since $u \in \overline B_d$ and, by \ref{CoroGradNotZero}, $\nabla_x \varphi(t, \gamma(t)) \neq 0$, this implies that
\[
u = -\frac{\nabla_x \varphi(t, \gamma(t))}{\abs{\nabla_x \varphi(t, \gamma(t))}},
\]
and thus \eqref{NablaTauQIsOptimalControl} holds.
\end{proof}

Notice that, since $\varphi$ is Lipschitz continuous, one obtains as a consequence of Corollary \ref{Coro1Varphi}\ref{CoroGradNotZero} that $\partial_t \varphi(t, x) > -1$ and $\nabla_x \varphi(t, x) \not = 0$ for almost every $(t, x) \in \mathbb R_+ \times \Omega$.

\begin{defi}
\label{DefiOptControl}
Consider the optimal control problem $\OCP(X, \Gamma, k)$ and assume that Hypothesis \ref{MainHypo}\ref{HypoXOverlineOmega} holds. Given $(t_0, x_0) \in \mathbb R_+ \times \overline\Omega$ and $\gamma \in \Opt(k, t_0, x_0)$, we say that a measurable function $u: \mathbb R_+ \to \overline B_d$ is an \emph{optimal control} associated with $\gamma$ if $\dot\gamma(t) = k(t, \gamma(t)) u(t)$ for almost every $t \in \mathbb R_+$.
\end{defi}

In terms of Definition \ref{DefiOptControl}, Corollary \ref{Coro1Varphi}\ref{CoroAlmostVelocityField} states that any optimal control $u$ associated with an optimal trajectory $\gamma$ satisfies $u(t) = - \frac{\nabla_x \varphi(t, \gamma(t))}{\abs{\nabla_x \varphi(t, \gamma(t))}}$ whenever $\gamma$ is differentiable at $t \in [t_0, t_0 + \varphi(t_0, x_0))$, $\gamma(t) \in \Omega$, and $\varphi$ is differentiable at $(t, \gamma(t))$. Even though $\varphi$ and $\gamma$ are both Lipschitz continuous, and hence differentiable almost everywhere, $\varphi$ may be nowhere differentiable along a given optimal trajectory when $d \geq 2$, and thus Corollary \ref{Coro1Varphi}\ref{CoroAlmostVelocityField} is not sufficient to characterize the optimal control for every optimal trajectory.

\subsection{Consequences of Pontryagin Maximum Principle}
\label{SecPMP}

As a first step towards providing a characterization of the optimal control associated with an optimal trajectory $\gamma \in \Opt(k, t_0, x_0)$, we apply Pontryagin Maximum Principle to $\OCP(X, \Gamma, k)$ to obtain a relation between the optimal control and the costate from Pontryagin Principle and deduce a differential inclusion for the optimal control. To do so, we will assume the more restrictive Hypothesis \ref{MainHypo}\ref{HypoExitOnPartialOmega} on the target set $\Gamma$. 
Notice that, under Hypothesis \ref{MainHypo}\ref{HypoXOverlineOmega}, $\OCP(\overline\Omega, \Gamma, k)$ is an optimal control problem with the state constraint $\gamma(t) \in \overline\Omega$ for every $t \in \mathbb R_+$, but such a state constraint becomes redundant when one assumes that $\Gamma = \partial\Omega$, since optimal trajectories starting at $\overline\Omega$ stop as soon as they reach the target set $\partial\Omega$, meaning that they will automatically always remain in $\overline\Omega$.

Even though versions of Pontryagin Maximum Principle for non-smooth dynamics and state constraints are available \cite[Theorem 5.2.3]{Clarke1990Optimization}, as well as techniques for adapting the unconstrained maximum principle to the constrained case \cite{Cannarsa2008Regularity}, which could be used to study $\OCP(\overline\Omega, \Gamma, k)$ without Hypothesis \ref{MainHypo}\ref{HypoExitOnPartialOmega}, we prefer to state its conclusions under Hypothesis \ref{MainHypo}\ref{HypoExitOnPartialOmega} for simplicity since this assumption will be needed in the sequel in Section \ref{SecNormGradientAndOptimalControl}. However, notice the need for non-smooth statements of the Pontryagin Maximum Principle (which involve differential inclusions), the reason being that we are not assuming    Hypothesis \ref{MainHypo}\ref{HypokGlobC11}. 

In the next result, $\pi_t: \mathbb R \times \mathbb R^d \to \mathbb R$ and $\pi_x: \mathbb R \times \mathbb R^d \to \mathbb R^d$ denote the canonical projections onto the factors of the product $\mathbb R \times \mathbb R^d$.

\begin{prop}
\label{PropPMPTauQ}
Consider the optimal control problem $\OCP(X, \Gamma, k)$, assume that Hypotheses \ref{MainHypo}\ref{HypoXCompact}\allowbreak--\ref{HypokGlobLip} hold, and let $\Omega$ be as in Hypothesis \ref{MainHypo}\ref{HypoXOverlineOmega}. Let $(t_0, x_0) \in \mathbb R_+ \times \overline\Omega$, $\gamma \in \Opt(k, t_0, x_0)$, $T = \varphi(t_0, x_0)$, and $u: \mathbb R_+ \to \overline B_d$ be a measurable optimal control associated with $\gamma$. Then there exist $\lambda \in \{0, 1\}$ and absolutely continuous functions $p: [t_0, t_0 + T] \to \mathbb R^d$ and $h: [t_0, t_0 + T] \to \mathbb R$ such that
\begin{enumerate}
\item\label{PMPAdjointEq} For almost every $t \in [t_0, t_0 + T]$,
\begin{equation}
\label{PMPSystHP}
\left\{
\begin{aligned}
\dot h(t) & \in   \abs{p(t)} \pi_t \gengrad k(t, \gamma(t)), \\
\dot p(t) & \in - \abs{p(t)} \pi_x \gengrad k(t, \gamma(t)).
\end{aligned}
\right.
\end{equation}

\item\label{PMPOptimalU} One has
\[
u(t) = \frac{p(t)}{\abs{p(t)}}
\]
almost everywhere on $\left\{t \in [t_0, t_0 + T] \midsuchthat p(t) \neq 0\right\}$.

\item\label{PMPhHamiltonian} For almost every $t \in [t_0, t_0 + T]$, one has $h(t) = \abs{p(t)} k(t, \gamma(t)) - \lambda$, and $h(t_0 + T) = 0$.

\item\label{PMPTransverseP} $-p(t_0 + T) \in N_{\partial\Omega}(\gamma(t_0 + T))$.

\item\label{PMPNotIdenticallyZero} $\lambda + \max_{t \in [t_0, t_0 + T]} \abs{p(t)} > 0$.
\end{enumerate}
\end{prop}

Proposition \ref{PropPMPTauQ} can be obtained from \cite[Theorem 5.2.3]{Clarke1990Optimization} using a classical technique of state augmentation to transform \eqref{AdmissibleIsControlSystem} into an autonomous control system on the augmented state variable $z = (t, \gamma)$.

Notice that Proposition \ref{PropPMPTauQ}\ref{PMPOptimalU} characterizes the optimal control in terms of the costate $p$ whenever the costate is non-zero. Our next result states that this happens everywhere on $[t_0, t_0 + T]$.

\begin{lemm}
\label{LemmPNeqZero}
Consider the optimal control problem $\OCP(X, \Gamma, k)$, assume that Hypotheses \ref{MainHypo}\ref{HypoXCompact}\allowbreak--\ref{HypokGlobLip} hold, and let $\Omega$, $(t_0, x_0)$, $\gamma$, $T$, $\lambda$, $p$, and $h$ be as in the statement of Proposition \ref{PropPMPTauQ}. Then $p(t) \neq 0$ for every $t \in [t_0, t_0 + T]$.
\end{lemm}

\begin{proof}
Let $M$ be a Lipschitz constant for $k$ and $\beta: [t_0, t_0 + T] \to \mathbb R^d$ be a measurable function such that $\beta(t) \in \pi_x \gengrad k(t, \gamma(t))$ and $\dot p(t) = - \abs{p(t)} \beta(t)$ for almost every $t \in [t_0, t_0 + T]$. Then $\abs{\beta(t)} \leq M$ for almost every $t \in [t_0, t_0 + T]$ (see, e.g., \cite[Proposition 2.1.2]{Clarke1990Optimization}) and thus, for every $t, t_1 \in [t_0, t_0 + T]$,
\[
\abs{p(t)} \leq \abs{p(t_1)} + M \int_{\min\{t_1, t\}}^{\max\{t_1, t\}} \abs{p(s)} \diff s.
\]
Hence, by Gronwall's inequality, for every $t, t_1 \in [t_0, t_0 + T]$,
\[
\abs{p(t)} \leq \abs{p(t_1)} e^{M \abs{t - t_1}}.
\]
One then concludes that, if there exists $t_1 \in [t_0, t_0 + T]$ such that $p(t_1) = 0$, then $p(t) = 0$ for every $t \in [t_0, t_0 + T]$. Thus, by Proposition \ref{PropPMPTauQ}\ref{PMPhHamiltonian}, one has $h(t) = -\lambda$ for every $t \in [t_0, t_0 + T]$, and, since $h(t_0 + T) = 0$, it follows that $\lambda = 0$, contradicting Proposition \ref{PropPMPTauQ}\ref{PMPNotIdenticallyZero}. Thus $p(t) \neq 0$ for every $t \in [t_0, t_0 + T]$.
\end{proof}

Combining Lemma \ref{LemmPNeqZero} with the differential inclusion for $p$ from \eqref{PMPSystHP}, one obtains the following differential inclusion for $u$.

\begin{coro}
\label{CoroSystGammaU}
Consider the optimal control problem $\OCP(X, \Gamma, k)$, assume that Hypotheses \ref{MainHypo}\ref{HypoXCompact}\allowbreak--\ref{HypokGlobLip} hold, and let $\Omega$, $(t_0, x_0)$, $\gamma$, $T$, and $u$ be as in the statement of Proposition \ref{PropPMPTauQ}. Then $\gamma \in \mathcal C^{1, 1}([t_0,\allowbreak t_0 + T], \mathbb R^d)$, $u \in \Lip([t_0,\allowbreak t_0 + T], \mathbb S^{d-1})$, and $(\gamma, u)$ solves the system
\begin{equation}
\label{SystGammaU}
\left\{
\begin{aligned}
\dot \gamma(t) & = k(t, \gamma(t)) u(t), \\
\dot u(t) & \in - \Proj^\perp_{u(t)} \pi_x \gengrad k(t, \gamma(t)),
\end{aligned}
\right.
\end{equation}
where, for $x \in \mathbb S^{d-1}$, $\Proj^\perp_{x}: \mathbb R^d \to T_{x} \mathbb S^{d-1}$ is the projection of $v$ onto the tangent space to $x$ of $\mathbb S^{d-1}$, defined for $v \in \mathbb R^d$ by $\Proj^\perp_{x} v = v - (x \cdot v) x$.
\end{coro}

\begin{proof}
Let $p$ be as in the statement of Proposition \ref{PropPMPTauQ}. Thanks to Lemma \ref{LemmPNeqZero}, it follows from Proposition \ref{PropPMPTauQ}\ref{PMPOptimalU} that $u(t) = \frac{p(t)}{\abs{p(t)}}$, and in particular $u$ is absolutely continuous and takes values in $\mathbb S^{d-1}$ for every $t \in [t_0, t_0 + T]$. Let $\beta: [t_0, t_0 + T] \to \mathbb R^d$ be a measurable function such that $\beta(t) \in \pi_x \gengrad k(t, \gamma(t))$ and $\dot p(t) = - \abs{p(t)} \beta(t)$ for almost every $t \in [t_0, t_0 + T]$. Then, for almost every $t \in [t_0, t_0 + T]$,
\begin{equation*}
\dot u(t) = \frac{\dot p(t) \abs{p(t)} - \frac{p(t) \cdot \dot p(t)}{\abs{p(t)}} p(t)}{\abs{p(t)}^2} = - \beta(t) + \left[u(t) \cdot \beta(t)\right] u(t),
\end{equation*}
which yields the differential inclusion for $u$ in \eqref{SystGammaU}. Since $k$ is Lipschitz continuous, one obtains that $\abs{\dot u(t)}$ is bounded for almost every $t \in [t_0, t_0 + T]$, and thus $u$ is Lipschitz continuous. It follows from the differential equation on $\gamma$ in \eqref{SystGammaU} that $\dot\gamma$ is Lipschitz continuous on $[t_0, t_0 + T]$, and hence $\gamma \in \mathcal C^{1, 1}([t_0,\allowbreak t_0 + T], \mathbb R^d)$.
\end{proof}

Corollary \ref{CoroSystGammaU} states that every optimal trajectory satisfies, together with its optimal control, \eqref{SystGammaU}. However, given $(t_0, x_0) \in \mathbb R_+ \times \overline\Omega$, $u_0 \in \mathbb S^{d-1}$, a solution $(\gamma, u)$ of \eqref{SystGammaU} with initial condition $\gamma(t_0) = x_0$ and $u(t_0) = u_0$ may not yield an optimal trajectory. In order to understand when a solution of \eqref{SystGammaU} is an optimal trajectory, we introduce the following definition.

\begin{defi}\label{DefiOptDirections}
Let $(t_0, x_0) \in \mathbb R_+ \times \Omega$. We define the set $\mathcal U_{(t_0, x_0)}$ of \emph{optimal directions} at $(t_0, x_0)$ as the set of all $u_0 \in \mathbb S^{d-1}$ such that there exists a solution $(\gamma, u)$ of \eqref{SystGammaU} with $\gamma(t_0) = x_0$ and $u(t_0) = u_0$ satisfying $\gamma \in \Opt(k, t_0, x_0)$.
\end{defi}

Thanks to Proposition \ref{Prop1Varphi} and Corollary \ref{CoroSystGammaU}, $\mathcal U_{(t_0, x_0)}$ is non-empty. Our next result shows that, along an optimal trajectory $\gamma$, $\mathcal U_{(t, \gamma(t))}$ is a singleton, except possibly at its initial and final points.

\begin{prop}
\label{PropSingletonAfterStartingTime}
Consider the optimal control problem $\OCP(X, \Gamma, k)$ and assume that Hypotheses \ref{MainHypo}\ref{HypoXCompact}\allowbreak--\ref{HypokGlobLip} hold. Let $(t_0, x_0) \in \mathbb R_+ \times \Omega$, $\gamma \in \Opt(k, t_0, x_0)$. Then, for every $t \in (t_0, t_0 + \varphi(t_0, x_0))$, $\mathcal U_{(t, \gamma(t))}$ contains exactly one element.
\end{prop}

\begin{proof}
In order to simplify the notations, set $T = \varphi(t_0, x_0)$, fix $t_1 \in (t_0, t_0 + T)$ and let $x_1 = \gamma(t_1)$. Let $u$ be an optimal control associated with $\gamma$. By Corollary \ref{CoroSystGammaU}, $u \in \Lip([t_0, t_0 + T], \mathbb S^{d-1})$. Thanks to Proposition \ref{Prop2Varphi}\ref{LemmOptimalAllAlong}, $u(t_1) \in \mathcal U_{(t_1, x_1)}$; suppose to have another element $u_1 \in \mathcal U_{(t_1, x_1)}$: we will prove that $u_1 = u(t_1)$.

Since $u_1 \in \mathcal U_{(t_1, x_1)}$, there exists a solution $(\widetilde\gamma, \widetilde u)$ of \eqref{SystGammaU} with $\widetilde\gamma(t_1) = x_1$, $\widetilde u(t_1) = u_1$, and $\widetilde\gamma \in \Opt(k, t_1, x_1)$. Let $(\widehat\gamma, \widehat u)$ be defined on $\mathbb R_+$ by
\[
\widehat\gamma(t) = 
\begin{dcases*}
\gamma(t), & if $t < t_1$, \\
\widetilde\gamma(t), & if $t \geq t_1$,
\end{dcases*} \qquad
\widehat u(t) = 
\begin{dcases*}
u(t), & if $t < t_1$, \\
\widetilde u(t), & if $t \geq t_1$.
\end{dcases*}
\]
We can see that $\widehat\gamma \in \Opt(k, t_0, x_0)$ and $\widehat u$ is an optimal control associated with $\widehat\gamma$. Hence, by Corollary \ref{CoroSystGammaU}, $\widehat\gamma \in \mathcal C^{1, 1}([t_0, t_0 + T], \mathbb R^d)$ and $\widehat u \in \Lip([t_0, t_0 + T], \mathbb S^{d-1})$. In particular, $\widehat u$ is continuous, which proves that $u(t_1) = \widetilde u(t_1) = u_1$ and that $u(t_1)$ is the unique element of $\mathcal U_{(t_1, x_1)}$.
\end{proof}

\subsection{Normalized gradient and the optimal control}
\label{SecNormGradientAndOptimalControl}

In this section, we are interested in the additional properties one gets when assuming Hypothesis \ref{MainHypo}\ref{HypokGlobC11}, the main goal being to characterize the optimal control $u$ in a similar way to Corollary \ref{Coro1Varphi}\ref{CoroAlmostVelocityField} but for all optimal trajectories and all times. A first result obtained from the extra regularity from Hypothesis \ref{MainHypo}\ref{HypokGlobC11} is the following immediate consequence of Corollary \ref{CoroSystGammaU}.

\begin{coro}
\label{CoroSystGammaURegular}
Consider the optimal control problem $\OCP(X, \Gamma, k)$, assume that Hypotheses \ref{MainHypo}\ref{HypoXCompact}\allowbreak--\ref{HypokGlobC11} hold, and let $\Omega$, $(t_0, x_0)$, $\gamma$, $T$, and $u$ be as in the statement of Proposition \ref{PropPMPTauQ}. Then $\gamma \in \mathcal C^{2, 1}([t_0,\allowbreak t_0 + T], \mathbb R^d)$, $u \in \mathcal C^{1, 1}([t_0,\allowbreak t_0 + T], \mathbb S^{d-1})$, and $(\gamma, u)$ solves the system
\begin{equation}
\label{SystGammaURegular}
\left\{
\begin{aligned}
\dot \gamma(t) & = k(t, \gamma(t)) u(t), \\
\dot u(t) & = - \Proj^\perp_{u(t)} \nabla_x k(t, \gamma(t)).
\end{aligned}
\right.
\end{equation}
\end{coro}

Thanks to the extra regularity assumption from Hypothesis \ref{MainHypo}\ref{HypokGlobC11} and to Hypothesis \ref{MainHypo}\ref{HypoPartialOmegaESP}, one may also use classical results in optimal control to conclude that $\varphi$ is a semiconcave function.

\begin{prop}
\label{PropVarphiSemiconcave}
Consider the optimal control problem $\OCP(X, \Gamma, k)$ and assume that Hypotheses \ref{MainHypo}\ref{HypoXCompact}\allowbreak--\ref{HypoPartialOmegaESP} hold. Then $\varphi$ is semiconcave in $\mathbb R_+ \times \Omega$.
\end{prop}

\begin{proof}
This is a classical fact and we will use \cite[Theorem 8.2.7]{Cannarsa2004Semiconcave}. This Theorem is stated on the whole time-space $\R\times\R^d$, and requires an autonomous control system. 

Hence, we will set
$\mathcal K = \mathbb R^d \setminus \Omega$: $\mathcal K$ is closed and, due to the uniform exterior sphere property of $\Omega$, $\mathcal K$ satisfies a uniform interior sphere property. Consider the optimal control problem of reaching $\mathcal K$ in minimal time with dynamics \eqref{AdmissibleIsControlSystem}, starting at time $t \in \mathbb R$ from a point $x \in \mathbb R^d$, and denote by $\widetilde\varphi$ its value function (the only difference from $\OCP(\overline\Omega, \partial\Omega, k)$ is that we allow negative initial times). Of course, $\widetilde\varphi$ and $\varphi$ coincide on $\mathbb R_+ \times \overline\Omega$.

We rewrite \eqref{AdmissibleIsControlSystem} as an autonomous control system on the variable $z = (t, \gamma)$ and apply \cite[Theorem 8.2.7]{Cannarsa2004Semiconcave} to the corresponding autonomous optimal control problem with target set $\mathbb R \times \mathcal K$ to conclude that $\widetilde\varphi$ is locally semiconcave on $\mathbb R \times \Omega$. Notice that, since one has an upper bound on the exit time for all optimal trajectories thanks to Proposition \ref{Prop1Varphi}\ref{PropBoundTau} (and its immediate generalization to the optimal control problem for $\widetilde\varphi$), it follows from the proof of \cite[Theorem 8.2.7]{Cannarsa2004Semiconcave} that $\widetilde\varphi$ is semiconcave on $\mathbb R \times \Omega$. Hence, $\varphi$ is semiconcave on $\mathbb R_+ \times \Omega$.
\end{proof}

\begin{defi} Let $A \subset \mathbb R^d$, $B \subset \mathbb R$.
\begin{enumerate}
\item Let $w: A \to \mathbb R$ and $x \in A$. We say that $p \in \mathbb R^d$ is a \emph{reachable gradient} of $w$ at $x$ if there exists a sequence $(x_n)_{n \in \mathbb N}$ in $A \setminus \{x\}$ such that, for every $n \in \mathbb N$, $w$ is differentiable at $x_n$, and $x_n \to x$ and $D w(x_n) \to p$ as $n \to \infty$. The set of all reachable gradients of $w$ at $x$ is denoted by $D^\ast w(x)$.

\item Let $w: B \times A \to \mathbb R$ and $(t, x) \in B \times A$. We say that $\omega \in \mathbb S^{d-1}$ is the \emph{normalized gradient} of $w$ with respect to $x$ at $(t, x)$ if $D^+ w(t, x)$ is non-empty and, for every $(p_0, p_1) \in D^+ w(t, x)$, one has $p_1 \neq 0$ and $\frac{p_1}{\abs{p_1}} = \omega$. In this case, $\omega$ is denoted by $\widehat{\nabla w}(t, x)$\label{WidehatNabla}.
\end{enumerate}
\end{defi}

Recall that, for a semiconcave function $w$ defined on an open set $A \subset \mathbb R^d$, for every $x \in A$, one has $D^+ w(x) = \conv D^\ast w(x) = \gengrad w(x)$ and these sets are non-empty (see, e.g., \cite[Proposition 3.3.4 and Theorem 3.3.6]{Cannarsa2004Semiconcave}). Here, $\conv B$ denotes the convex hull of $B$.

\begin{theo}
\label{TheoNormGradIffSingleton}
Consider the optimal control problem $\OCP(X, \Gamma, k)$ and assume that Hypotheses \ref{MainHypo}\ref{HypoXCompact}\allowbreak--\ref{HypoPartialOmegaESP} hold. Let $(t_0, x_0) \in \mathbb R_+^\ast \times \Omega$. Then $\varphi$ admits a normalized gradient at $(t_0, x_0)$ if and only if $\mathcal U_{(t_0, x_0)}$ contains only one element. In this case, $\mathcal U_{(t_0, x_0)} = \{-\widehat{\nabla\varphi}(t_0, x_0)\}$.
\end{theo}

\begin{proof}
Notice first that, thanks to Proposition \ref{PropVarphiSemiconcave}, $\varphi$ is semiconcave in $\mathbb R_+ \times \Omega$. Assume that $\varphi$ admits a normalized gradient at $(t_0, x_0)$ and take $u_0 \in \mathcal U_{(t_0, x_0)}$. Let $(\gamma, u)$ be a solution of \eqref{SystGammaU} with $\gamma(t_0) = x_0$, $u(t_0) = u_0$, and $\gamma \in \Opt(k, t_0, x_0)$. By Proposition \ref{Prop2Varphi}\ref{PropFlowPreservesOptimality}, for every $h \in [0, \varphi(t_0, x_0)]$, one has
\[
\varphi(t_0 + h, \gamma(t_0 + h)) + h = \varphi(t_0, x_0).
\]
Let $\psi: [0, \varphi(t_0, x_0)] \to \mathbb R_+ \times \overline\Omega$ be given by $\psi(h) = (t_0 + h, \gamma(t_0 + h))$. Then $\varphi \circ \psi(h) + h = \varphi(t_0, x_0)$. Since $h \mapsto \varphi(t_0, x_0) - h$ is differentiable on $[0, \varphi(t_0, x_0)]$, with derivative equal to $-1$, one obtains that $\varphi \circ \psi$ is differentiable on $[0, \varphi(t_0, x_0)]$, with $(\varphi \circ \psi)^\prime(h) = -1$ for every $h \in [0, \varphi(t_0, x_0)]$. Moreover, $\psi \in \mathcal C^1([0, \varphi(t_0, x_0)], \mathbb R_+ \times \overline\Omega)$, with $\psi^\prime(h) = (1, \dot\gamma(t_0 + h))$. Hence, by the chain rule for generalized gradients (see, e.g., \cite[Theorem 2.3.10]{Clarke1990Optimization}), one has, for every $h \in (0, \varphi(t_0, x_0))$,
\[
\{-1\} = \gengrad (\varphi \circ \psi)(h) \subset \gengrad \varphi(t_0 + h, \gamma(t_0 + h)) \cdot (1, \dot\gamma(t_0 + h)).
\]
Thus there exist $(p_{0, h}, p_{1, h}) \in \gengrad \varphi(t_0 + h, \gamma(t_0 + h))$ such that
\begin{equation}
\label{EqP0P1WithH}
-1 = p_{0, h} + p_{1, h} \cdot \dot\gamma(t_0 + h).
\end{equation}
Notice moreover that, since $\varphi$ is semiconcave, one has $(p_{0, h}, p_{1, h}) \in D^+ \varphi(t_0 + h, \gamma(t_0 + h))$. Since $\varphi$ is Lipschitz continuous, the family $((p_{0, h}, p_{1, h}))_{h \in (0, \varphi(t_0, x_0))}$ is bounded, and thus there exists a sequence $(h_n)_{n \in \mathbb N}$ with $h_n \to 0$ as $n \to \infty$ and a point $(p_0, p_1) \in \mathbb R \times \mathbb R^d$ such that, as $n \to \infty$, $p_{0, h_n} \to p_0$ and $p_{1, h_n} \to p_1$. Moreover, $(p_0, p_1) \in D^+\varphi(t_0, x_0)$ (see, e.g., \cite[Proposition 3.3.4(a)]{Cannarsa2004Semiconcave}). One also has that $\dot\gamma(t_0 + h) = k(t_0 + h, \gamma(t_0 + h)) u(t_0 + h)$ and thus, taking $h = h_n$ in \eqref{EqP0P1WithH} and letting $n \to \infty$, one obtains that
\begin{equation}
\label{EqP0P1}
-1 = p_0 + p_1 \cdot u_0 k(t_0, x_0).
\end{equation}
Letting $c > 0$ be as in Lemma \ref{LemmP0GreaterThanMinusOne}, one obtains that $-p_1 \cdot u_0 k(t_0, x_0) \geq c$, and thus, in particular, $p_1 \neq 0$. Since $\varphi$ admits a normalized gradient at $(t_0, x_0)$, one then concludes that $\frac{p_1}{\abs{p_1}} = \widehat{\nabla\varphi}(t_0, x_0)$.

Now, since $(p_0, p_1) \in D^+\varphi(t_0, x_0)$, there exists a neighborhood $V$ of $(t_0, x_0)$ in $\mathbb R_+ \times \Omega$ and a function $\phi \in \mathcal C^1(V, \mathbb R)$ such that $\phi(t_0, x_0) = \varphi(t_0, x_0)$, $\phi(t, x) \geq \varphi(t, x)$ for every $(t, x) \in V$, and $p_0 = \partial_t \phi(t_0, x_0)$, $p_1 = \nabla_x \phi(t_0, x_0)$. Since $\varphi$ is a viscosity subsolution of \eqref{EqHJB} on $\mathbb R_+ \times \Omega$, one concludes that
\[
-p_0 + \abs{p_1} k(t_0, x_0) - 1 \leq 0.
\]
Combining the above inequality with \eqref{EqP0P1}, one obtains that
\[
\abs{p_1} k(t_0, x_0) + p_1 \cdot u_0 k(t_0, x_0) \leq 0,
\]
i.e., $p_1 \cdot u_0 \leq - \abs{p_1}$. Since $u_0 \in \mathbb S^{d-1}$, this means that $u_0 = - \frac{p_1}{\abs{p_1}} = - \widehat{\nabla\varphi}(t_0, x_0)$. Since this holds for every $u_0 \in \mathcal U_{(t_0, x_0)}$, one concludes that the unique element in $\mathcal U_{(t_0, x_0)}$ is $- \widehat{\nabla\varphi}(t_0, x_0)$.

Assume now that $\mathcal U_{(t_0, x_0)}$ contains only one element, which we denote by $u_0$. Since $D^+ \varphi(t_0, x_0)\allowbreak = \conv D^\ast \varphi(t_0, x_0)$ is non-empty, $D^\ast \varphi(t_0, x_0) \neq \emptyset$. Take $p = (p_0, p_1) \in D^\ast \varphi(t_0, x_0)$. Hence there exists a sequence $((t_{0, n}, x_{0, n}))_{n \in \mathbb N}$ such that $\varphi$ is differentiable at $(t_{0, n}, x_{0, n})$ for every $n \in \mathbb N$ and $t_{0, n} \to t_0$, $x_{0, n} \to x_0$, and $D \varphi(t_{0, n}, x_{0, n}) \to p$ as $n \to \infty$. For $n \in \mathbb N$, write $(p_{0, n}, p_{1, n}) = D \varphi(t_{0, n}, x_{0, n})$, and let $\gamma_n \in \Opt(k, t_{0, n}, x_{0, n})$ and $u_n$ be an optimal control associated with $\gamma_n$. By Proposition \ref{PropMonotoneOptimalTime} and Corollary \ref{Coro1Varphi}, there exists $c > 0$ such that, for every $n \in \mathbb N$, $\abs{p_{1, n}} \geq c$ and $u_n(t_{0, n}) = - \frac{p_{1, n}}{\abs{p_{1, n}}}$ In particular, $\abs{p_1} \geq c$, and thus $p_1 \neq 0$. Using Corollary \ref{CoroSystGammaURegular}, one obtains that $(\gamma_n, u_n)$ solves
\[
\left\{
\begin{aligned}
\dot\gamma_n(t) & = k(t, \gamma_n(t)) u_n(t), & & t \in [t_{0, n}, t_{0, n} + \varphi(t_{0, n}, x_{0, n})], \\
\dot u_n(t) & = -\Proj^\perp_{u_n(t)} \nabla_x k(t, \gamma_n(t)), & & t \in [t_{0, n}, t_{0, n} + \varphi(t_{0, n}, x_{0, n})], \\
\gamma_n(t_{0, n}) & = x_{0, n}, \\
u_n(t_{0, n}) & = -\frac{p_{1, n}}{\abs{p_{1, n}}}.
\end{aligned}
\right.
\]
We modify $u_n$ outside of the interval $[t_{0, n}, t_{0, n} + \varphi(t_{0, n}, x_{0, n})]$ by setting $u_n(t) = u_n(t_{0, n})$ for $t \in [0, t_{0, n})$ and $u_n(t) = u_n(t_{0, n} +\allowbreak \varphi(t_{0, n},\allowbreak x_{0, n}))$ for $t > t_{0, n} + \varphi(t_{0, n},\allowbreak x_{0, n})$. We have $\gamma_n \in \Lip_{K_{\max}}(\overline\Omega)$, $u_n \in \Lip_M(\mathbb S^{d-1})$, where $M$ is a Lipschitz constant for $k$, and thus, by Arzelà--Ascoli Theorem, there exist subsequences, which we still denote by $(\gamma_n)_{n \in \mathbb N}$ and $(u_n)_{n \in \mathbb N}$ for simplicity, and elements $\gamma \in \Lip_{K_{\max}}(\overline\Omega)$, $u \in \Lip_M(\mathbb S^{d-1})$ such that $\gamma_n \to \gamma$ and $u_n \to u$ uniformly on compact time intervals.

Since $\varphi$ is Lipschitz continuous, one has $\varphi(t_{0, n}, x_{0, n}) \to \varphi(t_0, x_0)$ as $n \to \infty$. Since $\gamma_n(t) = \gamma_n(t_{0, n})$ and $u_n(t) = u_n(t_{0, n})$ for $t \in [0, t_{0, n}]$ and $\gamma_n(t) = \gamma_n(t_{0, n} + \varphi(t_{0, n}, x_{0, n}))$ and $u_n(t) = u_n(t_{0, n} + \varphi(t_{0, n}, x_{0, n}))$ for $t \geq t_{0, n} + \varphi(t_{0, n}, x_{0, n})$, one obtains that $\gamma(t) = \gamma(t_0)$ and $u(t) = u(t_0)$ for $t \in [0, t_0]$ and $\gamma(t) = \gamma(t_0 + \varphi(t_0, x_0))$, $u(t) = u(t_0 + \varphi(t_0, x_0))$ for $t \geq t_0 + \varphi(t_0, x_0)$. Since $\gamma_n(t_{0, n} + \varphi(t_{0, n}, x_{0, n})) \allowbreak \in \partial\Omega$, one also obtains that $\gamma(t_0 + \varphi(t_0, x_0)) \in \partial\Omega$, and in particular its exit time satisfies $\tau(t_0, \gamma) \leq \varphi(t_0, x_0)$. From $\dot\gamma_n= k(\cdot, \gamma_n) u_n$, the local uniform convergence of $(\gamma_n)_{n \in \mathbb N}$ and $(u_n)_{n \in \mathbb N}$ provides, at the limit, $\dot\gamma = k(\cdot, \gamma) u$. Hence $\gamma \in \Opt(k, t_0, x_0)$ and $u$ is an optimal control associated with $\gamma$.

Since $u(t_0) = \lim_{n \to \infty} u_n(t_{0, n}) = \lim_{n \to \infty} \left(-\frac{p_{1, n}}{\abs{p_{1, n}}}\right) = -\frac{p_1}{\abs{p_1}}$, one obtains that $-\frac{p_1}{\abs{p_1}} \in \mathcal U_{(t_0, x_0)}$, and thus, by assumption, $-\frac{p_1}{\abs{p_1}} = u_0$. We have thus proved that, for every $p = (p_0, p_1) \in D^\ast \varphi(t_0, x_0)$, one has $p_1 \neq 0$ and $-\frac{p_1}{\abs{p_1}} = u_0$. Since $D^+ \varphi(t_0, x_0) = \conv D^\ast \varphi(t_0, x_0)$, one immediately checks that $p_1 \neq 0$ and $-\frac{p_1}{\abs{p_1}} = u_0$ for every $p = (p_0, p_1) \in D^+ \varphi(t_0, x_0)$, which proves that $\varphi$ admits a normalized gradient at $(t_0, x_0)$ and $- \widehat{\nabla\varphi}(t_0, x_0) = u_0$, i.e., $\mathcal U_{(t_0, x_0)} = \{-\widehat{\nabla\varphi}(t_0, x_0)\}$.
\end{proof}

As an immediate consequence of Proposition \ref{PropSingletonAfterStartingTime} and Theorem \ref{TheoNormGradIffSingleton}, one obtains the following characterization of the optimal control.

\begin{coro}
\label{CoroDiffEqnWithNormGrad}
Consider the optimal control problem $\OCP(X, \Gamma, k)$ and assume that Hypotheses \ref{MainHypo}\ref{HypoXCompact}\allowbreak--\ref{HypoPartialOmegaESP} hold. Let $(t_0, x_0) \in \mathbb R_+ \times \Omega$, $\gamma \in \Opt(k, t_0, x_0)$, and $u$ be the optimal control associated with $\gamma$. Then, for every $t \in (t_0, t_0 + \varphi(t_0, x_0))$, $\varphi$ admits a normalized gradient at $(t, \gamma(t))$ and $u(t) = - \widehat{\nabla\varphi}(t, \gamma(t))$, i.e.,
\[
\dot\gamma(t) = - k(t, \gamma(t)) \widehat{\nabla\varphi}(t, \gamma(t)).
\]
\end{coro}

Corollary \ref{CoroDiffEqnWithNormGrad} can be seen as an improved version of Corollary \ref{Coro1Varphi}\ref{CoroAlmostVelocityField}, where we now characterize the optimal control for every optimal trajectory and every time, replacing the normalization of the gradient of $\varphi$ by its normalized gradient.

Notice that, combining Corollaries \ref{CoroSystGammaURegular} and \ref{CoroDiffEqnWithNormGrad}, we obtain that $t \mapsto \widehat{\nabla\varphi}(t, \gamma(t))$ is $\mathcal C^{1, 1}$ for every optimal trajectory $\gamma$, as long as $\gamma(t) \in \Omega$ and $t$ is larger than the initial time of $\gamma$. However, this provides no information on the regularity of $(t, x) \mapsto \widehat{\nabla\varphi}(t, x)$. We are now interested in proving continuity of this map in a set where $\widehat{\nabla\varphi}$ is defined.

Set
\begin{equation}
\label{DefiUpsilon}
\Upsilon = \{(t, x) \in \mathbb R_+^\ast \times \Omega \mid \exists t_0 \in [0, t),\; \exists x_0 \in \Omega,\; \exists \gamma \in \Opt(k, t_0, x_0) \text{ s.t.\ } \gamma(t) = x\}.
\end{equation}
This set (which can be easily seen to be a countable union of compact sets, and in particular a Borel measurable set) contains all points $(t, x) \in \mathbb R_+^\ast \times \Omega$ which are not starting points of optimal trajectories. In particular, it follows from Corollary \ref{CoroDiffEqnWithNormGrad} that $\widehat{\nabla\varphi}(t, x)$ exists for every $(t, x) \in \Upsilon$.

\begin{prop}
\label{PropNormGradContinuous}
Consider the optimal control problem $\OCP(X, \Gamma, k)$ and assume that Hypotheses \ref{MainHypo}\ref{HypoXCompact}\allowbreak--\ref{HypoPartialOmegaESP} hold. Then the normalized gradient $\widehat{\nabla\varphi}$ of $\varphi$ is continuous on $\Upsilon$.
\end{prop}

\begin{proof}
Let $(t_\ast, x_\ast) \in \Upsilon$ and $((t_n, x_n))_{n \in \mathbb N}$ be a sequence in $\Upsilon$ converging to $(t_\ast, x_\ast)$ as $n \to \infty$. For $n \in \mathbb N$, let $t_{0, n} \in [0, t_n)$, $x_{0, n} \in \Omega$ and $\gamma_n \in \Opt(k, t_{0, n}, x_{0, n})$ be such that $\gamma_n(t_n) = x_n$. Let $u_n$ be an optimal control associated with $\gamma_n$, and modify $u_n$ outside of $[t_{0, n}, t_{0, n} + \varphi(t_{0, n}, x_{0, n})]$ for $u_n$ to be constant on $[0, t_{0, n}]$ and on $[t_{0, n} + \varphi(t_{0, n}, x_{0, n}), +\infty)$. Notice that, since $x_n \in \Omega$, one has $t_n \in (t_{0, n}, t_{0, n} + \varphi(t_{0, n}, x_{0, n}))$ and, by Corollary \ref{CoroDiffEqnWithNormGrad}, $u_n(t_n) = -\widehat{\nabla\varphi}(t_n, x_n)$. We want to prove that $u_n(t_n) \to -\widehat{\nabla\varphi}(t_\ast, x_\ast)$ as $n \to \infty$. Since $(u_n(t_n))_{n \in \mathbb N}$ is bounded, it is enough to prove that every convergent subsequence of $(u_n(t_n))_{n \in \mathbb N}$ converges to $-\widehat{\nabla\varphi}(t_\ast, x_\ast)$ as $n \to \infty$. 

Assume that $(u_{n_k}(t_{n_k}))_{k \in \mathbb N}$ is a converging subsequence of $(u_n(t_n))_{n \in \mathbb N}$, and denote by $u_\ast$ its limit. Up to subsequences (using Arzelà--Ascoli Theorem together with Corollary \ref{CoroSystGammaU}), we can assume that $\gamma_{n_k} $ and $u_{n_k} $ also converge, uniformly on every compact time interval, and that $(t_{0, n_k})_{k \in \mathbb N}$ and $(x_{0, n_k})_{k \in \mathbb N}$ also converge. Let $\gamma = \lim_{k \to \infty} \gamma_{n_k}$, $u = \lim_{k \to \infty} u_{n_k}$, $t_0 = \lim_{k \to \infty} t_{0, n_k}$, and $x_0 = \lim_{k \to \infty} x_{0, k}$.

Since $\gamma_{n}(t) = x_{0, n}$ for $t \in [0, t_{0, n}]$, $u_n(t) = u_n(t_{0, n})$ for $t \in [0, t_{0, n}]$, $\gamma_n(t) = \gamma_n(t_{0, n} + \varphi(t_{0, n}, x_{0, n}))$ for $t \geq t_{0, n} + \varphi(t_{0, n}, x_{0, n})$, and $u_n(t) = u_n(t_{0, n} + \varphi(t_{0, n}, x_{0, n}))$ for $t \geq t_{0, n} + \varphi(t_{0, n}, x_{0, n})$, one obtains that $\gamma(t) = x_{0}$ for $t \in [0, t_{0}]$, $u(t) = u(t_{0})$ for $t \in [0, t_{0}]$, $\gamma(t) = \gamma(t_{0} + \varphi(t_{0}, x_{0}))$ for $t \geq t_{0} + \varphi(t_{0}, x_{0})$, and $u(t) = u(t_{0} + \varphi(t_{0}, x_{0}))$ for $t \geq t_{0} + \varphi(t_{0}, x_{0})$, which proves in particular that $\tau(t_0, \gamma) \leq \varphi(t_0, x_0)$, and one will obtain that $\gamma \in \Opt(k,\allowbreak t_0, x_0)$ as soon as one proves that $\gamma \in \Adm(k)$. This can be done exactly as in the proof of Theorem \ref{TheoNormGradIffSingleton}, also obtaining that $u$ is an optimal control associated with $\gamma$. Notice moreover that $\gamma(t_\ast) = \lim_{k \to \infty} \gamma_{n_k} (t_{n_k}) = \lim_{k \to \infty} x_{n_k} = x_\ast$.

Since $(t_\ast, x_\ast) \in \Upsilon$, there exist $t_{0, \ast} \in [0, t_\ast)$, $x_{0, \ast} \in \Omega$, and $\gamma_\ast \in \Opt(k, t_{0, \ast}, x_{0, \ast})$ such that $\gamma_\ast(t_\ast) = x_\ast$. Let $u_\ast$ be an optimal control associated with $\gamma_\ast$, assumed to be constant on $[0, t_{0, \ast}]$ and on $[t_{0, \ast}+ \varphi(t_{0, \ast}, x_{0, \ast}), +\infty)$. Since $x_\ast \in \Omega$, one has $t_\ast \in (t_{0, \ast}, t_{0, \ast} + \varphi(t_{0, \ast}, x_{0, \ast}))$, and one also has, from Corollary \ref{CoroDiffEqnWithNormGrad}, that $u_\ast(t_\ast) = -\widehat{\nabla\varphi}(t_\ast,\allowbreak x_\ast)$. Define $\widetilde\gamma: \mathbb R_+ \to \overline\Omega$ and $\widetilde u: \mathbb R_+ \to \mathbb S^{d-1}$ for $t \in \mathbb R_+$ by
\[
\widetilde\gamma(t) = 
\begin{dcases*}
\gamma_\ast(t), & if $t < t_{\ast}$, \\
\gamma(t), & if $t \geq t_\ast$,
\end{dcases*} \qquad \widetilde u(t) = 
\begin{dcases*}
u_\ast(t), & if $t < t_{\ast}$, \\
u(t), & if $t \geq t_\ast$.
\end{dcases*}
\]
Since $\gamma_\ast \in \Opt(k, t_{0, \ast}, x_{0, \ast})$, $\gamma \in \Opt(k, t_0, x_0)$, and $\gamma_\ast(t_\ast) = \gamma(t_\ast) = x_\ast$, one obtains that $\widetilde\gamma \in \Opt\allowbreak(k, \allowbreak t_{0, \ast}, \allowbreak x_{0, \ast})$, and $\widetilde u$ is an optimal control associated with $\widetilde\gamma$. Hence, by Corollary \ref{CoroSystGammaURegular}, the restrictions of $\widetilde\gamma$ and $\widetilde u$ to $[t_{0, \ast}, t_{0, \ast} + \varphi(t_{0, \ast}, x_{0, \ast})]$, still denoted by $\widetilde\gamma$ and $\widetilde u$ for simplicity, satisfy $\widetilde\gamma \in \mathcal C^{2, 1}([t_{0, \ast}, t_{0, \ast} + \varphi(t_{0, \ast}, x_{0, \ast})], \overline\Omega)$ and $\widetilde u \in \mathcal C^{1, 1}([t_{0, \ast}, t_{0, \ast} + \varphi(t_{0, \ast}, x_{0, \ast})], \mathbb S^{d-1})$. In particular, since $t_\ast \in (t_{0, \ast}, t_{0, \ast} + \varphi(t_{0, \ast}, x_{0, \ast}))$, $\widetilde u$ is continuous at $t_\ast$, which finally proves that $-\widehat{\nabla\varphi}(t_\ast, x_\ast) = u_\ast(t_\ast) = u(t_\ast) = \lim_{k \to \infty} u_{n_k}(t_{n_k}) = u_\ast$.
\end{proof}

\section{Existence of equilibria for minimal-time MFG}
\label{SecExistence}

After the preliminary study of the optimal control problem $\OCP(X, \Gamma, k)$ from Section \ref{SecOCP}, we finally turn to our main problem, the mean field game $\MFG(X, \Gamma, K)$. This section is concerned with the existence of equilibria for $\MFG(X, \Gamma, K)$. We shall need here only some elementary results on the value function from Section \ref{SecValueFunction}, and our main result is the following.

\begin{theo}
\label{MainTheoExist}
Consider the mean field game $\MFG(X, \Gamma, K)$ and assume that Hypotheses \ref{MainHypo}\ref{HypoXCompact}, \ref{HypoXDist}, and \ref{HypoK2Lip} hold. Then, for every $m_0 \in \mathcal P(X)$, there exists an equilibrium $Q \in \mathcal P(\mathcal C_X)$ for $\MFG(X,\allowbreak \Gamma,\allowbreak K)$ with initial condition $m_0$.
\end{theo}

The remainder of the section provides the proof of Theorem \ref{MainTheoExist}, which is decomposed in several steps. First, for a fixed $Q \in \mathcal P(\mathcal C_X)$, we consider in Section \ref{SecExist-Lipschitz} the optimal control problem $\OCP(X, \Gamma, k)$ where $k$ is given by $k(t, x) = K({e_t}_{\#} Q, x)$, proving first that $k$ is Lipschitz continuous, in order to be able to apply the results from Section \ref{SecValueFunction}, and then that the value function $\varphi$ is also Lipschitz continuous with respect to $Q$. Such properties are then used in Section \ref{SecExist-MeasurableSelection} to prove that one may choose an optimal trajectory for each starting point in a measurable way. Finally, Section \ref{SecExist-FixedPoint} characterizes equilibria of $\MFG(X, \Gamma, K)$ as fixed points of a suitable multi-valued function and concludes the proof of Theorem \ref{MainTheoExist} by the application of a suitable fixed point theorem.

\subsection{Lipschitz continuity of the value function}
\label{SecExist-Lipschitz}

Notice that, under Hypothesis \ref{MainHypo}\ref{HypoK2Lip}, any equilibrium $Q$ of $\MFG(X, \Gamma, K)$ should be concentrated on $K_{\max}$-Lipschitz continuous trajectories. Define
\[
\mathcal Q\label{DefiMathcalQ} = \{Q \in \mathcal P(\mathcal C_X) \mid Q(\Lip_{K_{\max}}(X)) = 1\}.
\]
Notice that $\mathcal Q$ is convex, and one can prove by standard arguments that it is also tight and closed, obtaining, as a consequence of Prohorov's Theorem (see, e.g., \cite[Chapter 1, Theorem 5.1]{Billingsley1999Convergence}), that $\mathcal Q$ is compact. We now prove that the push-forward of a measure $Q \in \mathcal Q$ by the evaluation map $e_t$ is Lipschitz continuous in both variables $t$ and $Q$.

\begin{prop}
\label{PropEtQLip}
Assume that Hypothesis \ref{MainHypo}\ref{HypoXCompact} holds.
\begin{enumerate}
\item\label{etQLipT} For every $Q \in \mathcal Q$, the map $\mathbb R_+ \ni t \mapsto {e_t}_\# Q \in \mathcal P(X)$ is $K_{\max}$-Lipschitz continuous.
\item\label{etQLipQ} For every $t \in \mathbb R_+$, the map $\mathcal P(\mathcal C_X) \ni Q \mapsto {e_t}_\# Q \in \mathcal P(X)$ is $2^{t+1}$-Lipschitz continuous.
\end{enumerate}
\end{prop}

\begin{proof}
To prove \ref{etQLipT}, take $t, s \in \mathbb R_+$. By \eqref{WassersteinDual} and using the fact that $Q(\Lip_{K_{\max}}(X)) = 1$, we have
\begin{align*}
W_1({e_t}_{\#} Q, {e_s}_{\#} Q) & = \sup_{\substack{\phi: X \to \mathbb R \\ \text{1-Lipschitz}}} \int_X \phi(x) \diff({e_t}_{\#} Q - {e_s}_{\#} Q)(x) \\
 & = \sup_{\substack{\phi: X \to \mathbb R \\ \text{1-Lipschitz}}} \int_{\Lip_{K_{\max}}(X)} (\phi(\gamma(t)) - \phi(\gamma(s))) \diff Q(\gamma) \leq K_{\max} \abs{t - s},
\end{align*}
and thus $t \mapsto {e_t}_{\#} Q$ is $K_{\max}$-Lipschitz continuous.

In order to prove \ref{etQLipQ}, we just need to observe that, due to the definition of $\dist_{\mathcal C_X}$, the evaluation map $e_t$ is $2^{n}$-Lipschitz continuous for every $n\in\mathbb N$ with $n\geq t$. Hence, this map is $2^{t+1}$-Lipschitz continuous and this implies  \ref{etQLipQ}.
\end{proof}

As an immediate consequence of Proposition \ref{PropEtQLip}, we obtain the following.

\begin{coro}
\label{CorokFromK}
Let $Q \in \mathcal Q$ and assume that Hypothesis \ref{MainHypo}\ref{HypoXCompact} holds and $K: \mathcal P(X) \times X \to \mathbb R_+$ satisfies Hypothesis \ref{MainHypo}\ref{HypoK2Lip}. Define $k: \mathbb R_+ \times X \to \mathbb R_+$ by $k(t, x) = K({e_t}_{\#} Q, x)$. Then $k$ satisfies Hypothesis \ref{MainHypo}\ref{Hypok1Lip}, i.e., $k$ is Lipschitz continuous on $\mathbb R_+ \times X$ and $k(t, x) \in [K_{\min}, K_{\max}]$ for every $(t, x) \in \mathbb R_+ \times X$.
\end{coro}

In particular, when $k$ is constructed from $K$ as in Corollary \ref{CorokFromK}, the results of Section \ref{SecValueFunction} apply to $k$. Notice that the value function $\varphi$ from Definition \ref{DefiVarphi} depends on $Q$ through $k$. In the sequel, whenever needed, we make this dependence explicit in the notation of the value function by writing it as $\varphi_Q$.

\begin{prop}
\label{PropTauLipQ}
Consider the mean field game $\MFG(X, \Gamma, K)$ and assume that Hypotheses \ref{MainHypo}\ref{HypoXCompact}, \ref{HypoXDist}, and \ref{HypoK2Lip} hold. Let $(t, x) \in \mathbb R_+ \times X$. Then the map $\mathcal Q \ni Q \mapsto \varphi_Q(t, x)$ is Lipschitz continuous in $\mathcal Q$, its Lipschitz constant being independent of $x$ and a non-decreasing function of $t$.
\end{prop}

\begin{proof}
Consider $\OCP(X, \Gamma, k)$ with $k$ defined from $K$ as in Corollary \ref{CorokFromK}. Let $M > 0$ be as in Proposition \ref{Prop1Varphi}\ref{PropBoundTau}. Let $Q_1, Q_2 \in \mathcal Q$. Let $\gamma_1 \in \Opt(Q_1, t, x)$ and define $\phi: \mathbb [t, +\infty) \to \mathbb R$ as the unique function satisfying
\begin{equation}
\label{IVP-Varphi}
\left\{
\begin{aligned}
\dot\phi(s) & = \frac{K({e_{s}}_{\#} Q_2, \gamma_1(\phi(s)))}{K({e_{\phi(s)}}_{\#} Q_1, \gamma_1(\phi(s)))}, \\
\phi(t) & = t.
\end{aligned}
\right.
\end{equation}
Similarly to the proof of Proposition \ref{PropTauLipT}, $\phi$ is a well-defined $\mathcal C^1$ function whose derivative takes values in $\left[\frac{K_{\min}}{K_{\max}}, \frac{K_{\max}}{K_{\min}}\right]$. Moreover, $\phi$ is strictly increasing and maps $[t, +\infty)$ onto itself, its inverse $\phi^{-1}$ being defined on $[t, +\infty)$. Define $\gamma_2: \mathbb R_+ \to X$ by
\[
\gamma_2(s) = 
\begin{dcases*}
x, & if $s \leq t$, \\
\gamma_1(\phi(s)), & if $s > t$.
\end{dcases*}
\]
Then $\gamma_2 \in \Lip(\mathbb R_+, X)$, $\gamma_2(s) = x$ for every $s \leq t$, $\gamma_2(s) = \gamma_2(\phi^{-1}(t + \varphi_{Q_1}(t, x))) \in \Gamma$ for every $s \geq \phi^{-1}(t + \varphi_{Q_1}(t, x))$, and $\abs{\dot\gamma_2}(s) \leq K({e_s}_{\#} Q_2, \gamma_2(s))$ for almost every $s \geq t$. Thus $\gamma_2 \in \Adm(Q_2)$ and $t + \varphi_{Q_2}(t, x) \leq \phi^{-1}(t + \varphi_{Q_1}(t, x))$, i.e., $\phi(t + \varphi_{Q_2}(t, x)) \leq t + \varphi_{Q_1}(t, x)$.

By Hypothesis \ref{MainHypo}\ref{HypoK2Lip} and Proposition \ref{PropEtQLip}, there exists $M_0 > 0$, independent of $(t, x)$, such that, for every $s \in \mathbb R_+$,
\[
\abs{K({e_{s}}_{\#} Q_2, \gamma_1(\phi(s))) - K({e_{\phi(s)}}_{\#} Q_1, \gamma_1(\phi(s)))} \leq M_0 \abs{\phi(s) - s} + 2^{s+1} M_0 W_1(Q_1, Q_2).
\]
One has, from \eqref{IVP-Varphi},
\[
\phi(s) - s = \int_t^s \left[\frac{K({e_{\sigma}}_{\#} Q_2, \gamma_1(\phi(\sigma)))}{K({e_{\phi(\sigma)}}_{\#} Q_1, \gamma_1(\phi(\sigma)))} - 1\right] \diff\sigma,
\]
and thus, for every $s \in [t, t + M]$,
\[
\abs{\phi(s) - s} \leq \frac{1}{K_{\min}}\left[2^{t+M+1} M M_0 W_1(Q_1, Q_2) + M_0 \int_{t}^s \abs{\phi(\sigma) - \sigma} \diff\sigma\right],
\]
which yields, by Gronwall's inequality, that, for every $s \in [t, t + M]$,
\[\abs{\phi(s) - s} \leq C_t W_1(Q_1, Q_2),\]
where $C_t$ is non-decreasing on $t$. We then conclude that
\[
t + \varphi_{Q_2}(t, x) \leq \phi(t + \varphi_{Q_2}(t, x)) + C_t W_1(Q_1, Q_2) \leq t + \varphi_{Q_1}(t, x) + C_t W_1(Q_1, Q_2),
\]
and finally
\[
\varphi_{Q_2}(t, x) - \varphi_{Q_1}(t, x) \leq C_t W_1(Q_1, Q_2).
\]
By exchanging $Q_1$ and $Q_2$ in the above inequality, one obtains the Lipschitz continuity of $Q \mapsto \varphi_Q(t, x)$ with Lipschitz constant $C_t$.
\end{proof}

\subsection{Measurable selection of optimal trajectories}
\label{SecExist-MeasurableSelection}

Proposition \ref{Prop1Varphi}\ref{PropExistOptim} and Corollary \ref{CorokFromK} imply that, given $Q \in \mathcal Q$, for every $x \in X$, there exists an optimal trajectory $\gamma$ for $(Q, 0, x)$. Such a trajectory depends on $x$, and one may thus construct a map that, to each point $x \in X$, associates some optimal trajectory for $(Q, 0, x)$. The main result of this section, Proposition \ref{PropMeasSelect}, asserts that it is possible to construct such a map in a measurable way.

For $Q \in \mathcal Q$, let $G_Q$ be the graph of the set-valued map $x \mapsto \Opt(Q, 0, x)$, given by
\begin{equation}
\label{GraphGamma}
G_Q = \{(x, \gamma) \in X \times \mathcal C_X \suchthat \gamma \in \Opt(Q, 0, x)\}.
\end{equation}

\begin{lemm}
\label{LemmClosedGraph}
Consider the mean field game $\MFG(X, \Gamma, K)$ and assume that Hypotheses \ref{MainHypo}\ref{HypoXCompact}, \ref{HypoXDist}, and \ref{HypoK2Lip} hold. Let $Q \in \mathcal Q$. Then the graph $G_Q$ is compact.
\end{lemm}

\begin{proof}
Let us first prove that $G_Q$ is a closed subset of $X \times \mathcal C_X$. Let $((x_n, \gamma_n))_{n \in \mathbb N}$ be a sequence in $G_Q$ such that, as $n \to \infty$, $x_n \to x$ for some $x \in X$ and $\gamma_n \to \gamma$ for some $\gamma \in \mathcal C_X$ (uniformly on compact time intervals). Since $\gamma_n \in \Opt(Q, 0, x_n)$, one has that $\gamma_n \in \Lip_{K_{\max}}(X)$, and thus, since $\Lip_{K_{\max}}(X)$ is closed, one has $\gamma \in \Lip_{K_{\max}}(X)$.

Since $\gamma_n(0) = x_n$, it follows that $\gamma(0) = x$. The continuity of the value function $\varphi_Q$ also implies that $\varphi_Q(0, x_n) \to \varphi_Q(0, x)$ as $n \to \infty$. Since $\gamma_n(\varphi_Q(0, x_n)) \in \Gamma$ and $\Gamma$ is closed, one concludes that $\gamma(\varphi_Q(0, x)) \in \Gamma$. Moreover, if $t \in \mathbb R_+$ is such that $t > \varphi_Q(0, x) $, then $t > \varphi_Q(0, x_n)$ for $n$ large enough, and thus $\gamma_n(t) = \gamma_n(\varphi_Q(0, x_n))$ for $n$ large enough, which shows that $\gamma(t) = \gamma(\varphi_Q(0, x))$ for every $t > \varphi_Q(0, x)$. Notice, in particular, that the exit time of $\gamma$ satisfies $\tau(0, \gamma) \leq \varphi_Q(0, x)$, and thus, to conclude that $\gamma \in \Opt(Q, 0, x)$, one is only left to prove that $\gamma \in \Adm(Q)$.

For $n \in \mathbb N$, since $\gamma_n \in \Adm(Q)$, one has $\abs{\dot\gamma_n}(t) \leq K({e_t}_{\#} Q, \gamma_n(t))$, a condition which passes to the limit as $n\to \infty$, providing $\abs{\dot\gamma}(t) \leq K({e_t}_{\#} Q, \gamma(t))$ for almost every $t \in \mathbb R_+$. Hence $\gamma \in \Adm(Q)$, which concludes the proof that $G_Q$ is closed.

Since $\Opt(Q, 0, x) \subset \Lip_{K_{\max}}(X)$ for every $x \in X$, one has $G_Q \subset X \times \Lip_{K_{\max}}(X)$ and, since $X \times \Lip_{K_{\max}}(X)$ is compact and $G_Q$ is closed, one concludes that $G_Q$ is compact.
\end{proof}

As a consequence of Lemma \ref{LemmClosedGraph} and the fact that $\Opt(Q, 0, x)$ is non-empty for every $x \in X$, one obtains the following measurable selection of optimal trajectories immediately from von Neumann--Aumann Selection Theorem (see, e.g., \cite[Theorem III.22]{Castaing1977Convex} or \cite[Theorem 3]{SainteBeuve1974Extension}).

\begin{prop}
\label{PropMeasSelect}
Consider the mean field game $\MFG(X, \Gamma, K)$ and assume that Hypotheses \ref{MainHypo}\ref{HypoXCompact}, \ref{HypoXDist}, and \ref{HypoK2Lip} hold. Let $Q \in \mathcal Q$, $m_0 = {e_0}_\# Q \in \mathcal P(X)$, and $\widehat{\mathfrak B}$ be the completion of the Borel $\sigma$-algebra of $X$ with respect to $m_0$. Then there exists a function $\pmb\gamma_Q: X \to \mathcal C_X$ such that $\pmb\gamma_Q(x) \in \Opt(Q, 0, x)$ for every $x \in X$ and $\pmb\gamma_Q$ is measurable with respect to $\widehat{\mathfrak B}$ in $X$ and to the Borel $\sigma$-algebra in $\mathcal C_X$.
\end{prop}

\subsection{Fixed point formulation for equilibria}
\label{SecExist-FixedPoint}

To prove Theorem \ref{MainTheoExist}, we reformulate the question of the existence of an equilibrium for $\MFG(X, \Gamma, K)$ into a question of the existence of a fixed point for a set-valued map. Let $m_0 \in \mathcal P(X)$ and define
\[\mathcal Q_{m_0} = \{Q \in \mathcal Q \mid {e_0}_\# Q = m_0\}.\]

\begin{lemm}
\label{LemmQm0}
Assume that Hypothesis \ref{MainHypo}\ref{HypoXCompact} holds. For every $m_0 \in \mathcal P(X)$, $\mathcal Q_{m_0}$ is non-empty, convex, and compact.
\end{lemm}

\begin{proof}
The convexity of $\mathcal Q_{m_0}$ follows immediately from its definition. To prove that it is non-empty, denote by $\psi: X \to \mathcal C_X$ the function that associates with each $x \in X$ the curve $\psi(x) \in \mathcal C_X$ which remains at $x$ for every time $t \geq 0$. Then $\psi_\# m_0 \in \mathcal Q_{m_0}$, and hence $\mathcal Q_{m_0}$ is non-empty. One also has that $\mathcal Q_{m_0}$ is closed in $\mathcal Q$, since, by Proposition \ref{PropEtQLip}, $Q \mapsto {e_0}_\# Q$ is continuous. In particular, since $\mathcal Q_{m_0} \subset \mathcal Q$ and $\mathcal Q$ is compact, one obtains that $\mathcal Q_{m_0}$ is compact.
\end{proof}

We now define a set-valued map $F_{m_0}: \mathcal Q_{m_0} \rightrightarrows \mathcal Q_{m_0}$ by setting, for $Q \in \mathcal Q_{m_0}$,
\[
F_{m_0}(Q) = \left\{\widetilde Q \in \mathcal Q_{m_0} \midsuchthat \widetilde Q\text{-almost every } \gamma \in \mathcal C_X \text{ satisfies }\gamma \in \Opt(Q, 0, \gamma(0))\right\}.
\]
By the definition of equilibria for $\MFG(X, \Gamma, K)$, $Q \in \mathcal Q_{m_0}$ is an equilibrium for $\MFG(X,\allowbreak \Gamma,\allowbreak K)$ with initial condition $m_0$ if and only if $Q \in F_{m_0}(Q)$. Introduce the set $\OOpt(Q)$ of all optimal trajectories for $Q$ starting at time $0$, i.e.,
\begin{equation}
\label{EqDefiOOpt}
\OOpt(Q) = \bigcup_{x \in X}\Opt(Q, 0, x).
\end{equation}
Notice that $\OOpt(Q)$ is compact as the projection into $\mathcal C_X$ of the compact set $G_Q$ from \eqref{GraphGamma}. The set $F_{m_0}(Q)$ can be rewritten in terms of $\OOpt(Q)$ as
\begin{equation}
\label{DefiFQ2}
F_{m_0}(Q) = \left\{\widetilde Q \in \mathcal Q_{m_0} \midsuchthat \widetilde Q(\OOpt(Q)) = 1\right\}.
\end{equation}
Let us provide some properties of the set-valued map $F_{m_0}$.

\begin{lemm}
\label{LemmFQNonEmptyAndCompact}
Consider the mean field game $\MFG(X, \Gamma, K)$ and assume that Hypotheses \ref{MainHypo}\ref{HypoXCompact}, \ref{HypoXDist}, and \ref{HypoK2Lip} hold. Then, for every $Q \in \mathcal Q_{m_0}$, $F_{m_0}(Q)$ is non-empty and compact.
\end{lemm}

\begin{proof}
Let $\pmb\gamma_Q: X \to \mathcal C_X$ be the function from Proposition \ref{PropMeasSelect}. Write $\widehat m_0$ for the completion of the measure $m_0$, which is defined in the $\sigma$-algebra $\widehat{\mathfrak B}$ from Proposition \ref{PropMeasSelect}. Let $\widetilde Q = {\pmb\gamma_Q}_{\#} \widehat m_0$, which is well-defined thanks to Proposition \ref{PropMeasSelect}.

Notice that $\widetilde Q \in \mathcal Q_{m_0}$. Indeed, $\widetilde Q\left(\Lip_{K_{\max}}(X)\right) = \widehat m_0\left(\pmb\gamma_Q^{-1}\left(\Lip_{K_{\max}}(X)\right)\right) = \widehat m_0(X) = 1$, since $\pmb\gamma_Q(x) \in \Opt(Q, 0, x) \subset \Lip_{K_{\max}}(X)$ for every $x \in X$. Moreover, for every Borel set $B$ in $X$, one has ${e_0}_\# \widetilde Q(B) = \widetilde Q\left(e_0^{-1}(B)\right) = \widehat m_0 \left(\pmb\gamma_Q^{-1}\left(e_0^{-1}(B)\right)\right) = \widehat m_0(B) = m_0(B)$ since $e_0 \circ \pmb\gamma_Q$ is the identity map in $X$ and $\widehat m_0$ and $m_0$ coincide on Borel sets. Hence $\widetilde Q \in \mathcal Q_{m_0}$.

Since $\pmb\gamma_Q(X) \subset \OOpt(Q)$, one has $\widetilde Q(\OOpt(Q)) \geq \widetilde Q(\pmb\gamma_Q(X)) = \widehat m_0(X) = 1$, which proves that $\widetilde Q \in F_{m_0}(Q)$, and thus $F_{m_0}(Q)$ is non-empty.

Let us now prove that $F_{m_0}(Q)$ is compact. Since $F_{m_0}(Q) \subset \mathcal Q_{m_0}$ and $\mathcal Q_{m_0}$ is compact, it suffices to prove that $F_{m_0}(Q)$ is closed in $\mathcal Q_{m_0}$. Let $(\widetilde Q_n)_{n \in \mathbb N}$ be a sequence in $F_{m_0}(Q)$ converging as $n \to \infty$ to some $\widetilde Q \in \mathcal Q_{m_0}$. Since $\OOpt(Q)$ is closed in $\mathcal C_X$, one obtains (see, e.g., \cite[Chapter 1, Theorem 2.1]{Billingsley1999Convergence}) that
\[
\widetilde Q(\OOpt(Q)) \geq \limsup_{n \to \infty} \widetilde Q_n(\OOpt(Q)) = 1,
\]
which proves that $\widetilde Q(\OOpt(Q)) = 1$, and thus $\widetilde Q \in F_{m_0}(Q)$. Hence $F_{m_0}(Q)$ is closed.
\end{proof}

Recall that, for $A, B$ two topological spaces, a set-valued map $H: A \rightrightarrows B$ is said to be \emph{upper semi-continuous} if, for every open set $W \subset B$, the set $\{x \in A \suchthat H(x) \subset W\}$ is open in $A$.

\begin{lemm}
\label{LemmPmbGammaUSC}
Consider the mean field game $\MFG(X, \Gamma, K)$ and assume that Hypotheses \ref{MainHypo}\ref{HypoXCompact}, \ref{HypoXDist}, and \ref{HypoK2Lip} hold. Let $m_0 \in \mathcal P(X)$ and $\OOpt: \mathcal Q_{m_0} \rightrightarrows \Lip_{K_{\max}}(X)$ be the set-valued map defined for $Q \in \mathcal Q_{m_0}$ by \eqref{EqDefiOOpt}. Then $\OOpt$ is upper semi-continuous.
\end{lemm}

\begin{proof}
Since $\OOpt(Q)$ is a non-empty compact subset of the compact set $\Lip_{K_{\max}}(X)$ for every $Q \in \mathcal Q_{m_0}$, it suffices to show that the graph of $\OOpt$ is closed (see, e.g., \cite[Proposition 1.4.8]{Aubin2009Set}).

Let $(Q_n)_{n \in \mathbb N}$ be a sequence in $\mathcal Q_{m_0}$ with $Q_n \to Q$ as $n \to \infty$ for some $Q \in \mathcal Q_{m_0}$ and $(\gamma_n)_{n \in \mathbb N}$ be a sequence in $\Lip_{K_{\max}}(X)$ with $\gamma_n \in \OOpt(Q_n)$ for every $n \in \mathbb N$ and $\gamma_n \to \gamma$ as $n \to \infty$ for some $\gamma \in \mathcal C_X$. Notice that $\gamma \in \Lip_{K_{\max}}(X)$ since $\Lip_{K_{\max}}(X)$ is closed. Since $\gamma_n \in \OOpt(Q_n)$, then $\gamma_n(t) = \gamma_n(\varphi_{Q_n}(0, \gamma_n(0))) \in \Gamma$ for $t \geq \varphi_{Q_n}(0, \gamma_n(0))$ and $\abs{\dot\gamma_n}(t) \leq K({e_t}_{\#} Q_n, \gamma_n(t))$ for almost every $t \geq 0$.

By Propositions \ref{PropTauLip} and \ref{PropTauLipQ}, one has that $\varphi_{Q_n}(0, \gamma_n(0)) \to \varphi_Q(0, \gamma(0))$ as $n \to \infty$, which implies that $\gamma(t) = \gamma(\varphi_Q(0, \gamma(0))) \in \Gamma$ for every $t \geq \varphi_Q(0, \gamma(0))$. In particular, the exit time of $\gamma$ satisfies $\tau(0, \gamma) \leq \varphi_Q(0, \gamma(0))$. The condition $\abs{\dot\gamma_n}(t) \leq K({e_t}_{\#} Q_n, \gamma_n(t))$ passes to the limit as $n \to \infty$, proving that $\gamma \in \Adm(Q)$. Hence $\gamma \in \Opt(Q, 0, \gamma(0)) \subset \OOpt(Q)$, and thus the graph of $\OOpt$ is closed.
\end{proof}

\begin{lemm}
\label{LemmFUSC}
Consider the mean field game $\MFG(X, \Gamma, K)$ and assume that Hypotheses \ref{MainHypo}\ref{HypoXCompact}, \ref{HypoXDist}, and \ref{HypoK2Lip} hold. Let $m_0 \in \mathcal P(X)$. Then the set-valued map $F_{m_0}: \mathcal Q_{m_0} \rightrightarrows \mathcal Q_{m_0}$ defined in \eqref{DefiFQ2} is upper semi-continuous.
\end{lemm}

\begin{proof}
Using Lemma \ref{LemmFQNonEmptyAndCompact} and \cite[Proposition 1.4.8]{Aubin2009Set}, it suffices to show that the graph of $F_{m_0}$ is closed. Let $(Q_n)_{n \in \mathbb N}$ be a sequence in $\mathcal Q_{m_0}$ with $Q_n \to Q$ as $n \to \infty$ for some $Q \in \mathcal Q_{m_0}$ and $(\widetilde Q_n)_{n \in \mathbb N}$ be a sequence in $\mathcal Q_{m_0}$ with $\widetilde Q_n \in F_{m_0}(Q_n)$ for every $n \in \mathbb N$ and $\widetilde Q_n \to \widetilde Q$ as $n \to \infty$ for some $\widetilde Q \in \mathcal Q_{m_0}$.

For $k \in \mathbb N^\ast$, let $V_k = \left\{\gamma \in \mathcal C_X \midsuchthat\dist(\gamma, \OOpt(Q)) \leq \frac{1}{k}\right\}$. Since $V_k$ is a neighborhood of $\OOpt(Q)$, it follows from Lemma \ref{LemmPmbGammaUSC} that there exists a neighborhood $W_k$ of $Q$ in $\mathcal Q_{m_0}$ such that $\OOpt(\widehat Q) \subset V_k$ for every $\widehat Q \in W_k$. Since $Q_n \to Q$ as $n \to \infty$, there exists $N_k \in \mathbb N$ such that, for every $n \geq N_k$, one has $Q_n \in W_k$, and thus $\OOpt(Q_n) \subset V_k$. Since $\widetilde Q_n(\OOpt(Q_n)) = 1$, one obtains that $\widetilde Q_n(V_k) = 1$ for every $n \geq N_k$. Since $\widetilde Q_n \to \widetilde Q$ as $n \to \infty$ and $V_k$ is closed, it follows that $\widetilde Q(V_k) \geq \limsup_{n \to \infty} \widetilde Q_n(V_k) = 1$, and thus $\widetilde Q(V_k) = 1$. Since this holds for every $k \in \mathbb N^\ast$ and $(V_k)_{k \in \mathbb N^\ast}$ is a non-increasing family of sets with $\bigcap_{k \in \mathbb N^\ast} V_k = \OOpt(Q)$, one concludes that $\widetilde Q(\OOpt(Q)) = \lim_{k \to \infty} \widetilde Q(V_k) = 1$. Hence $\widetilde Q \in F_{m_0}(Q)$, which proves that the graph of $F_{m_0}$ is closed.
\end{proof}

We can now conclude the proof of Theorem \ref{MainTheoExist}.

\begin{proof}[Proof of Theorem \ref{MainTheoExist}]
By Lemmas \ref{LemmQm0}, \ref{LemmFQNonEmptyAndCompact}, and \ref{LemmFUSC}, $\mathcal Q_{m_0}$ is convex, $F_{m_0}: \mathcal Q_{m_0} \rightrightarrows \mathcal Q_{m_0}$ is upper semi-continuous, and $F_{m_0}(Q)$ is non-empty, compact, and convex for every $Q \in \mathcal Q_{m_0}$. This means that $F_{m_0}$ is a Kakutani map (see, e.g., \cite[\S 7, Definition 8.1]{Granas2003Fixed}), and hence, thanks to Kakutani fixed point theorem (see, e.g., \cite[\S 7, Theorem 8.6]{Granas2003Fixed}), it admits a fixed point, i.e., an element $Q \in \mathcal Q_{m_0}$ such that $Q \in F_{m_0}(Q)$, which means that $Q$ is an equilibrium for $\MFG(X, \Gamma, K)$ with initial condition $m_0$.
\end{proof}

\section{The MFG system}
\label{SecCharacterization}

In most of the works on mean field games, such as the classical references \cite{Lasry2006JeuxI, Lasry2006JeuxII, Lasry2007Mean, Huang2006Large}, equilibria are characterized as solutions of a system of partial differential equations: the time evolution of the distribution of agents $m: \mathbb R_+ \to \mathcal P(X)$ satisfies a certain continuity equation, coupled with a Hamilton--Jacobi equation characterizing the value function $\varphi$ of the optimal control problem solved by each agent. This section provides such a characterization for the mean field game $\MFG(X, \Gamma, K)$:

\begin{theo}
\label{MainTheoCharact}
Consider the mean field game $\MFG(X, \Gamma, K)$ and assume that Hypotheses \ref{MainHypo}\ref{HypoXCompact}, \ref{HypoXDist}--\ref{HypoExitOnPartialOmega}, and \ref{HypoPartialOmegaESP}--\ref{HypoKConvolution} hold. Let $m_0 \in \mathcal P(\overline\Omega)$, $Q \in \mathcal Q$ be an equilibrium of $\MFG(\overline\Omega, \partial\Omega, K)$ with initial condition $m_0$, $k$ be defined from $K$ and $Q$ as in Corollary \ref{CorokFromK}, $\varphi$ be the value function defined in \eqref{EqDefiVarphi}, and $m = \mu^Q$. Then $(m, \varphi)$ solve the \emph{MFG system}
\begin{equation}
\label{SystMFG}
\left\{
\begin{aligned}
& \partial_t m(t, x) - \diverg\left[m(t, x) K(m_t, x) \widehat{\nabla\varphi}(t, x)\right] = 0, & \quad & (t, x) \in \mathbb R_+^\ast \times \Omega, \\
& -\partial_t \varphi(t, x) + \abs{\nabla_x \varphi(t, x)} K(m_t, x) - 1 = 0, & & (t, x) \in \mathbb R_+ \times \Omega, \\
& m(0, \cdot) = m_0, \\
& \varphi(t, x) = 0, & & (t, x) \in \mathbb R_+ \times \partial\Omega,
\end{aligned}
\right.
\end{equation}
where the first and second equations are satisfied, respectively, in the sense of distributions and in the viscosity sense.
\end{theo}

Recall that $\mu^Q: \mathbb R_+ \to \mathcal P(X)$ is the measure defined from $Q$ by $\mu^Q_t = {e_t}_{\#} Q$. System \eqref{SystMFG} is composed of a continuity equation on $m$ and a Hamilton--Jacobi equation on the value function $\varphi$. The Hamilton--Jacobi equation on $\varphi$ and its boundary condition follow immediately from Theorem \ref{MainTheoHJ} and Corollary \ref{CorokFromK}, and need only Hypotheses \ref{MainHypo}\ref{HypoXCompact}, \ref{HypoXDist}, \ref{HypoXOverlineOmega}, and \ref{HypoK2Lip}.

Concerning the continuity equation on $m$, it is clear that $m$ satisfies an equation of the form $\partial_t m + \diverg(m v) = 0$ for some velocity field $v$ (Proposition \ref{PropEtQLip}\ref{etQLipT} can be easily adapted to show that $t \mapsto {e_t}_{\#} Q$ is Lipschitz continuous with respect to the $W_p$ Wasserstein distance in $\mathcal P(X)$ for every $p \geq 1$, and thus one can apply \cite[Theorem 8.3.1]{Ambrosio2005Gradient}). The point of Theorem \ref{MainTheoCharact} is to identify this velocity field as $v = - K \widehat{\nabla\varphi}$.

This identification is possible thanks to Corollary \ref{CoroDiffEqnWithNormGrad} when $k$ is constructed from $K$ and $Q$ as in Corollary \ref{CorokFromK}. However, for a function $K$ satisfying Hypothesis \ref{MainHypo}\ref{HypoK2Lip}, $k$ does not necessarily satisfies Hypothesis \ref{MainHypo}\ref{HypokGlobC11}, which is needed for applying Corollary \ref{CoroDiffEqnWithNormGrad}. The goal of this section is to prove that, if $K$ satisfies the more restrictive Hypothesis \ref{MainHypo}\ref{HypoKConvolution} and $Q$ is an equilibrium of $\MFG(X, \Gamma, K)$, then $k$ satisfies Hypothesis \ref{MainHypo}\ref{HypokGlobC11}, and thus Corollary \ref{CorokFromK} applies. We shall also prove that the measures $m_t$ are concentrated on the set $\Upsilon$ introduced in \eqref{DefiUpsilon}, on which $\widehat{\nabla\varphi}$ is continuous.

\begin{prop}
\label{PropK0C11}
Consider the mean field game $\MFG(X, \Gamma, K)$ and assume that Hypotheses \ref{MainHypo}\ref{HypoXCompact}, \ref{HypoXDist}, \ref{HypoExitOnPartialOmega}, and \ref{HypoKConvolution} hold. Let $Q$ be an equilibrium of $\MFG(X, \Gamma, K)$ and define $k: \mathbb R_+ \times \mathbb R^d \to \mathbb R_+$ by $k(t, x) = K({e_t}_{\#} Q, x)$. Then $k \in \mathcal C^{1, 1}(\mathbb R_+ \times \mathbb R^d, \mathbb R_+)$ and it admits an extension to a $\mathcal C^{1, 1}$ function on $\mathbb R \times \mathbb R^d$. 
\end{prop}

\begin{proof}
Notice that, as in Corollary \ref{CorokFromK}, one has $k \in \Lip(\mathbb R_+ \times \mathbb R^d, \mathbb R_+)$, and one may extend $k$ to a Lipschitz continuous function $\widehat k$ defined on $\mathbb R \times \mathbb R^d$ by setting $\widehat k(t, x) = k(0, x)$ for $t < 0$. This allows one to apply Corollary \ref{CoroSystGammaU} with $\widehat k$ to obtain in particular that, for every $\gamma \in \OOpt(Q)$, one has $\gamma \in \mathcal C^{1, 1}([0, \varphi(0, \gamma(0))], \mathbb R^d)$. We extend each curve $\gamma \in \OOpt(Q)$ to a curve $\widetilde\gamma$ in $\mathcal C^{1, 1}((-\infty, \varphi(0, \gamma(0))], \mathbb R^d)$ by setting $\widetilde\gamma(t) = \gamma(0) + t \dot\gamma(0)$ for $t < 0$, and we denote by $\widetilde Q$ the measure on $\mathcal C(\mathbb R, \mathbb R^d)$ defined as the pushforward of $Q$ by such extension operator. We define $\widetilde k(t, x) = K({e_t}_{\#} \widetilde Q, x)$, which is defined on $\mathbb R \times \mathbb R^d$ and coincides with $k$ on $\mathbb R_+ \times \mathbb R^d$.

Let $E: \mathcal P(\overline\Omega) \times \mathbb R^d \to \mathbb R_+$ be as in Hypothesis \ref{MainHypo}\ref{HypoKConvolution}. Notice that $E$ can be easily extended to a function defined on $\mathcal P(\mathbb R^d) \times \mathbb R^d$, that we still denote by $E$. We define $E_0: \mathbb R \times \mathbb R^d \to \mathbb R_+$ by setting $E_0(t, x) = E({e_t}_{\#} \widetilde Q, x)$ for every $(t, x) \in \mathbb R \times \mathbb R^d$. Notice that, for every $(t, x) \in \mathbb R \times \mathbb R^d$, one has
\[
\nabla_x E_0 (t, x) = \int_{\mathbb R^d} \nabla \chi(x - y) \eta(y) \diff {e_t}_{\#}\widetilde Q (y),
\]
and in particular, proceeding as in the proof of Proposition \ref{PropFConvolution} and using Proposition \ref{PropEtQLip}, we obtain that $\nabla_x E_0$ is locally Lipschitz continuous on $\mathbb R \times \mathbb R^d$.
Let us consider now $\partial_t E_0$. One has
\[
E_0(t, x) = \int_{\mathbb R^d} \chi(x - y) \eta(y) \diff {e_t}_{\#} \widetilde Q(y) = \int_{\mathcal C(\mathbb R, \mathbb R^d)} \chi(x - \widetilde\gamma(t)) \eta(\widetilde\gamma(t)) \diff \widetilde Q(\widetilde\gamma).
\]
For given $x \in \mathbb R^d$, since $\widetilde Q$ is concentrated on curves $\widetilde\gamma$ belonging to $\mathcal C^{1, 1}((-\infty, \varphi(0, \widetilde\gamma(0))], \mathbb R^d)$ and constant for $t \geq \varphi(0, \widetilde\gamma(0))$, then, for $\widetilde Q$-almost every $\widetilde\gamma$, the function $t \mapsto \chi(x - \widetilde\gamma(t)) \eta(\widetilde\gamma(t))$ is differentiable everywhere, except possibly at $t = \varphi(0, \widetilde\gamma(0))$, with
\begin{equation}
\label{DerivativeE0Time}
\frac{\diff}{\diff t} \left[\chi(x - \widetilde\gamma(t)) \eta(\widetilde\gamma(t))\right] = -\nabla\chi(x - \widetilde\gamma(t)) \cdot \dot{\widetilde\gamma}(t) \eta(\widetilde\gamma(t)) + \chi(x - \widetilde\gamma(t)) \nabla\eta(\widetilde\gamma(t)) \cdot \dot{\widetilde\gamma}(t).
\end{equation}
Since $\eta(x) = 0$ and $\nabla\eta(x) = 0$ for $x \in \partial\Omega$ and $\widetilde\gamma(t) \in \partial\Omega$ for $t = \varphi(0, \widetilde\gamma(0))$, one can also prove that the above function is differentiable at $t = \varphi(0, \widetilde\gamma(0))$. Moreover, thanks to \eqref{DerivativeE0Time}, its derivative is Lipschitz continuous and upper bounded, and thus $\partial_t E_0(t, x)$ exists, with
\[
\partial_t E_0(t, x) = \int_{\mathcal C(\mathbb R, \mathbb R^d)} \left[-\nabla\chi(x - \widetilde\gamma(t)) \cdot \dot{\widetilde\gamma}(t) \eta(\widetilde\gamma(t)) + \chi(x - \widetilde\gamma(t)) \nabla\eta(\widetilde\gamma(t)) \cdot \dot{\widetilde\gamma}(t)\right] \diff \widetilde Q(\widetilde\gamma),
\] 
and one immediately verifies using the previous assumptions that $\partial_t E_0$ is Lipschitz continuous in $\mathbb R \times \mathbb R^d$. Together with the corresponding property for $\nabla_x E_0$, we obtain that $E_0 \in \mathcal C^{1, 1}(\mathbb R \times \mathbb R^d, \mathbb R_+)$. Since $g \in \mathcal C^{1, 1}(\mathbb R_+, \mathbb R_+^\ast)$, we conclude that $\widetilde k \in \mathcal C^{1, 1}(\mathbb R \times \mathbb R^d, \mathbb R_+)$.
\end{proof}

Proposition \ref{PropK0C11} allows one to apply all the results of Section \ref{SecOCP}, and in particular Corollary \ref{CoroDiffEqnWithNormGrad}, to the optimal control problem $\OCP(\overline\Omega, \partial\Omega, k)$ when $k$ is constructed from $K$ as in Proposition \ref{PropK0C11}. To conclude the proof of Theorem \ref{MainTheoCharact}, we show that the set where $\widehat{\nabla\varphi}$ is discontinuous has $m_t$ measure zero.

\begin{prop}
\label{PropMtConcentratedOnUpsilon}
Consider the mean field game $\MFG(X, \Gamma, K)$ and assume that Hypotheses \ref{MainHypo}\ref{HypoXCompact}, \ref{HypoXDist}, \ref{HypoExitOnPartialOmega}, and \ref{HypoKConvolution} hold. Let $Q \in \mathcal Q$ be an equilibrium of $\MFG(\overline\Omega, \partial\Omega, K)$, $m = \mu^Q$, $k$ be defined from $K$ and $Q$ as in Proposition \ref{PropK0C11}, and $\Upsilon_t$ be defined for $t \geq 0$ via $\Upsilon_t = \{x \in \Omega \mid (t, x) \in \Upsilon\}$. Then, for every $t >0$, $m_t(\Omega\setminus \Upsilon_t)=0$.
\end{prop}

\begin{proof}
Notice that, since $Q$ is an equilibrium of $\MFG(\overline\Omega, \partial\Omega, K)$, then $Q(\OOpt(Q)) = 1$. It follows easily from the definition of $\Upsilon$ that, for $t > 0$, $\{\gamma \in \OOpt(Q) \mid \gamma(t) \in \Omega \setminus \Upsilon_t\} = \emptyset$, and then $m_t (\Omega\setminus \Upsilon_t) = Q(\{\gamma \in \OOpt(Q) \mid \gamma(t) \in \Omega\setminus\Upsilon_t\})=Q(\emptyset)=0$.
\end{proof}


Thanks to Propositions \ref{PropK0C11} and \ref{PropMtConcentratedOnUpsilon}, one can easily obtain from Corollary \ref{CoroDiffEqnWithNormGrad} and Proposition \ref{PropNormGradContinuous} that $m$ satisfies the continuity equation in \eqref{SystMFG}. For the sake of completeness, we detail the proof of this fact.

\begin{proof}[Proof of Theorem \ref{MainTheoCharact}]
The Hamilton--Jacobi equation on $\varphi$ and its boundary condition follow immediately from Theorem \ref{MainTheoHJ} and Corollary \ref{CorokFromK}, and the initial condition on $m$ follows immediately from its definition. One is then left to show that $m$ satisfies the continuity equation in \eqref{SystMFG}.

By Proposition \ref{PropNormGradContinuous}, $\widehat{\nabla\varphi}$ is continuous on the set $\Upsilon$ defined in \eqref{DefiUpsilon}; moreover, it can easily be extended into a Borel-measurable function defined everywhere on $\mathbb R_+ \times \overline\Omega$, that we will still denote by $\widehat{\nabla\varphi}$.

Let $\phi \in \mathcal C_{\mathrm c}^{\infty}(\mathbb R_+^\ast \times \Omega, \mathbb R)$ be a test function. Take $\gamma \in \OOpt(Q)$ and $x_0 \in \overline\Omega$ such that $\gamma \in \Opt(Q,\allowbreak 0, x_0)$ and let $T = \varphi(0, x_0)$. 
By Corollary \ref{CoroDiffEqnWithNormGrad}, one has $\dot\gamma(t) = - K(m_t, \gamma(t)) \widehat{\nabla\varphi}(t, \gamma(t))$
for every $t \in (0, T)$. Hence, for every $t \in (0, T)$,
\begin{equation}
\label{EqToIntegrate}
\frac{\diff}{\diff t}\left[\phi(t, \gamma(t))\right] = \partial_t\phi(t, \gamma(t)) - \nabla_x \phi(t, \gamma(t)) \cdot \widehat{\nabla\varphi}(t, \gamma(t)) K(m_t, \gamma(t)).
\end{equation}
Notice that the right-hand side of \eqref{EqToIntegrate} is a continuous function of $(t, \gamma) \in \mathbb R_+^\ast \times \OOpt(Q)$. The only non-trivial term is $\nabla_x \phi(t, \gamma(t)) \cdot\widehat{\nabla\varphi}(t, \gamma(t))$. Take $(\widetilde t, \widetilde \gamma) \in \mathbb R_+^\ast \times \OOpt(Q)$. If $\widetilde\gamma(\widetilde t) \in \partial\Omega$ the continuity at $(\widetilde t, \widetilde \gamma)$ is guaranteed by the fact that $\phi$ is compactly supported on $\mathbb R_+^\ast \times \Omega$. Otherwise, if $\widetilde\gamma(\widetilde t) \in \Omega$, let $(t_n, \gamma_n)_{n \in \mathbb N}$ be a sequence with $t_n \to \widetilde t$ and $\gamma_n \to \widetilde \gamma$ in $\mathcal C_{\overline\Omega}$ as $n \to \infty$. For $n$ large enough, we have $\gamma_n(t_n) \in \Omega$ and $t_n > 0$. Since $\gamma_n \in \OOpt(Q)$, there exists $x_{0, n} \in \overline\Omega$ such that $\gamma_n \in \Opt(Q, 0, x_{0, n})$. Let $\widetilde x_0 = \widetilde \gamma(0) = \lim_{n \to \infty} \gamma_n(0) = \lim_{n \to \infty} x_{0, n}$ and notice that $\widetilde \gamma \in \Opt(Q, 0, \widetilde x_0)$. Since $\widetilde \gamma(\widetilde t) \in \Omega$ and $\gamma_n(t_n) \in \Omega$, one has $\widetilde x_{0} \in \Omega$ and $x_{0, n} \in \Omega$. Thus $(\widetilde t, \widetilde \gamma(\widetilde t)), (t_n, \gamma_n(t_n)) \in \Upsilon$, and then, using Proposition \ref{PropNormGradContinuous}, one concludes that $ \lim_{n \to \infty} \widehat{\nabla\varphi}(t_n, \gamma_n(t_n)) = \widehat{\nabla\varphi}(\widetilde t, \widetilde \gamma(\widetilde t))$, yielding the desired continuity.

Thanks to the continuity of its right-hand side in $\mathbb R_+ \times \OOpt(Q)$, one can integrate \eqref{EqToIntegrate} over this set, yielding
\begin{multline*}
\int_0^{+\infty} \int_{\OOpt(Q)} \frac{\diff}{\diff t}\left[\phi(t, \gamma(t))\right] \diff Q(\gamma) \diff t = \int_0^{+\infty} \int_{\OOpt(Q)} \partial_t \phi(t, \gamma(t)) \diff Q(\gamma) \diff t \\ {} - \int_0^{+\infty} \int_{\OOpt(Q)} \nabla_x \phi(t, \gamma(t)) \cdot \widehat{\nabla\varphi}(t, \gamma(t)) K(m_t, \gamma(t)) \diff Q(\gamma) \diff t.
\end{multline*}
Since $\varphi$ is compactly supported, the left-hand side of the above equality is zero. Moreover, 
using Proposition \ref{PropMtConcentratedOnUpsilon} and the facts that $\gamma(t) = e_t(\gamma)$ and $m_t = {e_t}_{\#} Q$, one finally deduces that, for every $\phi \in \mathcal C^\infty_{\mathrm c}(\mathbb R_+^\ast \times \Omega, \mathbb R)$,
\[
\int_0^{+\infty} \int_{\Omega} \partial_t \phi(t, x) \diff m_t(x) \diff t = \int_0^{+\infty} \int_{\Omega} \nabla_x \phi(t, x) \cdot \widehat{\nabla\varphi}(t, x) K(m_t, x) \diff m_t(x) \diff t,
\]
which is precisely the weak formulation of the continuity equation in \eqref{SystMFG}.
\end{proof}

\begin{remk}
Theorem \ref{MainTheoCharact} shows that, if $Q$ is an equilibrium of $\MFG(\overline\Omega, \partial\Omega, K)$, then $m = \mu^Q$ and the corresponding value function $\varphi$ solve the MFG system \eqref{SystMFG}. One may also prove the converse statement, i.e., that any solution $(m, \varphi)$ to \eqref{SystMFG} comes from an equilibrium $Q$ of $\MFG(\overline\Omega, \partial\Omega, K)$. Let us provide a brief idea of such a proof.

Using classical techniques in optimal control (cf., e.g., \cite[Chapter IV, Corollary 4.3]{Bardi1997Optimal}), one may prove that any viscosity solution $\varphi$ of the Hamilton--Jacobi equation in \eqref{SystMFG} satisfying also the corresponding boundary condition is the value function of $\OCP(\overline\Omega, \partial\Omega, k)$, where $k$ is obtained from $K$ and $m_t$ as in Corollary \ref{CorokFromK}. Moreover, using the superposition principle (cf., e.g., \cite[Theorem 3.2]{Ambrosio2008Transport}), one may prove that any solution $m$ in the sense of distributions of the continuity equation in \eqref{SystMFG} is a superposition solution, i.e., $m = \mu^Q$ for some $Q \in \mathcal P(\mathcal C_{\overline\Omega})$ concentrated on the solutions of $\dot\gamma(t) = - K(m_t, \gamma(t)) \widehat{\nabla\varphi}(t, \gamma(t))$. Finally, one may again use classical techniques in optimal control to show that, since $\varphi$ is the value function of $\OCP(\overline\Omega, \partial\Omega, k)$, any such $\gamma$ is necessarily an optimal trajectory of $\OCP(\overline\Omega, \partial\Omega, k)$, implying that $Q$ is an equilibrium of $\MFG(\overline\Omega, \partial\Omega, K)$, as required.
\end{remk}

\section{Examples}
\label{SecExamples}

\subsection{Mean field game on a segment}
\label{SecMFGDimOne}

Consider the mean field game $\MFG(X, \Gamma, K)$ with $\Omega = (0, 1) \subset \mathbb R$, $X = \overline\Omega$, $\Gamma = \partial\Omega = \{0, 1\}$, and $K$ satisfying the assumptions from Hypothesis \ref{MainHypo}\ref{HypoKConvolution}. This situation can be interpreted as the movement of agents on a long corridor, with one exit at each end, and where we assume that agents have no preference on which exit they take. We first remark that the equilibrium of such mean field game consists on agents on a certain segment $(0, \varkappa)$ going left and agents on $(\varkappa, 1)$ going right, the position $\varkappa$ of the splitting point depending on the initial distribution of agents $m_0$.

\begin{prop}
\label{PropDimensionOneSplit}
Let $m_0 \in \mathcal P([0, 1])$, $K$ satisfy Hypothesis \ref{MainHypo}\ref{HypoKConvolution}, $Q \in \mathcal Q$ be an equilibrium of $\MFG([0,\allowbreak 1],\allowbreak \{0, 1\},\allowbreak K)$ with initial condition $m_0$, and $m = \mu^Q$. Then there exists $\varkappa \in (0, 1)$ such that, for every $\gamma \in \OOpt(Q)$, $\gamma$ satisfies, for every $t \in [0, \varphi(0, \gamma(0)))$,
\begin{equation}
\label{Dimension1-Direction}
\begin{dcases*}
\dot\gamma(t) = - K(m_t, \gamma(t)), & if $\gamma(0) \in (0, \varkappa)$, \\
\dot\gamma(t) = K(m_t, \gamma(t)), & if $\gamma(0) \in (\varkappa, 1)$.
\end{dcases*}
\end{equation}
\end{prop}

The fact that agents move in a constant direction is a consequence of Corollary \ref{CoroSystGammaU}, since the optimal control $u$ is Lipschitz continuous and takes values in $\{-1, 1\}$, being thus necessarily constant. Existence of $\varkappa$ can be determined by defining, for a given equilibrium $Q$, the functions $T_l, T_r: [0, 1] \to \mathbb R_+$ by setting, for $x \in [0, 1]$, $T_l(x)$ and $T_r(x)$ as the times needed for solutions of $\dot\gamma(t) = - K(m_t, \gamma(t))$ and $\dot\gamma(t) = K(m_t, \gamma(t))$, respectively, to reach the boundary when their initial condition is $\gamma(0) = x$. Thanks to a characterization of $T_l$ and $T_r$ as implicit functions in terms of the flows of such differential equations, one can prove that $T_l^\prime(x) > 0$ and $T_r^\prime(x) < 0$ for every $x \in (0, 1)$, yielding, thanks to the fact that $T_l(0) = T_r(1) = 0$, the existence of a unique $\varkappa \in (0, 1)$ such that $T_l(\varkappa) = T_r(\varkappa)$. One can then verify that such $\varkappa$ satisfies the required properties.

We consider in the sequel the particular case where all agents are concentrated at some point $\ell \in [0, 1]$ at the starting time, i.e., $m_0 = \delta_\ell$, in which one can completely describe Lagrangian equilibria. This is one of the interesting points of the non-local congested dynamics $K$, the fact that it allows to study the case of Dirac masses (local dynamics in MFG are usually not well-defined for non-absolutely continuous measures).

\begin{prop}
\label{PropQSplitsInTwoDeltas}
Let $\ell \in [0, 1]$, $K$ satisfy Hypothesis \ref{MainHypo}\ref{HypoKConvolution}, $Q \in \mathcal Q$ be an equilibrium of $\MFG([0, 1],\allowbreak \{0, 1\},\allowbreak K)$ with initial condition $\delta_\ell$, and $m = \mu^Q$. Then there exists $\alpha \in [0, 1]$ such that
\begin{equation}
\label{Dimension1-DecomposeQ}
Q = \alpha \delta_{\gamma_l} + (1 - \alpha) \delta_{\gamma_r},
\end{equation}
where $\gamma_l, \gamma_r \in \OOpt(Q)$ are the unique solutions of
\begin{equation}
\label{DefGammaLGammaR}
\dot\gamma_l(t) = 
\begin{dcases*}
- K(m_t, \gamma_l(t)), & if $0 \leq t < \varphi(0, \ell)$, \\
0, & if $t > \varphi(0, \ell)$,
\end{dcases*} \quad
\dot\gamma_r(t) = 
\begin{dcases*}
K(m_t, \gamma_r(t)), & if $0 \leq t < \varphi(0, \ell)$, \\
0, & if $t > \varphi(0, \ell)$,
\end{dcases*}
\end{equation}
with initial condition $\gamma_l(0) = \gamma_r(0) = \ell$. In particular,
\begin{equation}
\label{Dimension1-DecomposeM}
m_t = \alpha \delta_{\gamma_l(t)} + (1 - \alpha) \delta_{\gamma_r(t)}
\end{equation}
for every $t \geq 0$.
\end{prop}

\begin{proof}
By definition of equilibrium, one has $Q(\OOpt(Q)) = 1$. Moreover, since ${e_0}_\# Q = \delta_\ell$, one has $Q(\{\gamma \in \OOpt(Q) \mid \gamma(0) = \ell\}) = 1$. If $\gamma \in \OOpt(Q)$ is such that $\gamma(0) = \ell$, then, by Corollary \ref{CoroSystGammaU}, its associated optimal control $u$ satisfies $u \in \Lip([0, \varphi(0, \ell)], \{-1, 1\})$, implying that $u$ is constant on $[0, \varphi(0, \ell)]$. Hence, $\gamma$ solves either $\dot\gamma(t) = - K(m_t, \gamma(t))$ or $\dot\gamma(t) = K(m_t, \gamma(t))$ on $[0, \varphi(0, \ell))$ with initial condition $\gamma(0) = \ell$, and, since $\gamma \in \OOpt(Q)$, one has $\dot\gamma(t) = 0$ for $t > \varphi(0, \ell)$. Thus one has either $\gamma = \gamma_l$ or $\gamma = \gamma_r$, which proves that $\{\gamma \in \OOpt(Q) \mid \gamma(0) = \ell\} \subset \{\gamma_l, \gamma_r\}$, and then $Q(\{\gamma_l, \gamma_r\}) = 1$, yielding the existence of $\alpha \in [0, 1]$ such that \eqref{Dimension1-DecomposeQ} holds. In particular, \eqref{Dimension1-DecomposeM} also holds.
\end{proof}

\begin{remk}
\label{RemkEquilibriumNonUnique}
In general, one may have several possible equilibria. Indeed, if $\ell = \frac{1}{2}$ and $K$ is defined by $K(\mu, x) = 1$ for every $(\mu, x) \in \mathcal P([0, 1]) \times \mathbb R$, one can immediately verify that the functions $\gamma_l$, $\gamma_r$ from \eqref{DefGammaLGammaR} are given by $\gamma_l(t) = \max\left(\frac{1}{2} - t, 0\right)$ and $\gamma_r(t) = \min\left(\frac{1}{2} + t, 1\right)$ for every $t \geq 0$, and $Q = \alpha \delta_{\gamma_l} + (1 - \alpha) \delta_{\gamma_r}$ is an equilibrium for every $\alpha \in [0, 1]$.
\end{remk}

Proposition \ref{PropQSplitsInTwoDeltas} states that every equilibrium of $\MFG([0, 1], \{0, 1\}, K)$ has the form \eqref{Dimension1-DecomposeQ}. Our next result provides sufficient conditions for a measure $Q$ given by \eqref{Dimension1-DecomposeQ} to be an equilibrium of $\MFG([0, 1], \{0, 1\}, K)$.

\begin{prop}
\label{PropDimensionOneCNS}
Let $K$ be defined from $g \in \mathcal C^{1, 1}(\mathbb R_+, \mathbb R_+^\ast)$, $\chi \in \mathcal C^{1, 1}(\mathbb R, \mathbb R_+)$, and $\eta \in \mathcal C^{1, 1}(\mathbb R, \mathbb R_+)$ as in Hypothesis \ref{MainHypo}\ref{HypoKConvolution}. Assume further that $\eta(x) = 0$ for every $x \in \mathbb R \setminus (0, 1)$. Let $\ell, \alpha \in [0, 1]$ and $\widetilde\gamma_l, \widetilde\gamma_r: \mathbb R_+ \to \mathbb R$ be the unique solutions of
\begin{equation}
\label{SystTildeGamma}
\left\{
\begin{aligned}
\dot{\widetilde\gamma_l}(t) & =           - g\left[\alpha \chi(0) \eta(\widetilde\gamma_l(t)) + (1 - \alpha) \chi(\widetilde\gamma_l(t) - \widetilde\gamma_r(t)) \eta(\widetilde\gamma_r(t))\right], & \quad & t \geq 0, \\
\dot{\widetilde\gamma_r}(t) & = \hphantom{-}g\left[\alpha \chi(\widetilde\gamma_r(t) - \widetilde\gamma_l(t)) \eta(\widetilde\gamma_l(t)) + (1 - \alpha) \chi(0) \eta(\widetilde\gamma_r(t))\right], & & t \geq 0, \\
\widetilde\gamma_l(0) & = \widetilde\gamma_r(0) = \ell.
\end{aligned}
\right.
\end{equation}
Define $T_l, T_r \in \mathbb R_+$ and $\gamma_l, \gamma_r: \mathbb R_+ \to [0, 1]$ by
\begin{equation}
\label{DefiTlTr}
\begin{aligned}
T_l & = \inf\{t \geq 0 \suchthat \widetilde\gamma_l(t) = 0\}, & \qquad \gamma_l(t) & = \max(\widetilde\gamma_l(t), 0)\quad \text{ for } t \geq 0, \\
T_r & = \inf\{t \geq 0 \suchthat \widetilde\gamma_r(t) = 1\}, & \gamma_r(t) & = \min(\widetilde\gamma_r(t), 1)\quad \text{ for } t \geq 0,
\end{aligned}
\end{equation}
and let $Q \in \mathcal P(\mathcal C_{[0, 1]})$ be given by $Q = \alpha \delta_{\gamma_l} + (1 - \alpha) \delta_{\gamma_r}$. Then $Q$ is an equilibrium for $\MFG([0, 1],\allowbreak \{0, 1\},\allowbreak K)$ with initial condition $\delta_\ell$ if and only if one of the following conditions hold:
\begin{enumerate}
\item\label{QEquilDim1-CaseAlpha01} $\alpha \in (0, 1)$ and $T_l = T_r$; or
\item\label{QEquilDim1-CaseAlpha0} $\alpha = 0$ and $T_l \geq T_r$; or
\item\label{QEquilDim1-CaseAlpha1} $\alpha = 1$ and $T_l \leq T_r$.
\end{enumerate}
\end{prop}

When $Q$ is an equilibrium, the fact that one of the conditions \ref{QEquilDim1-CaseAlpha01}--\ref{QEquilDim1-CaseAlpha1} is satisfied can be deduced from Proposition \ref{PropQSplitsInTwoDeltas}. Conversely, to prove that $Q$ is an equilibrium when one of \ref{QEquilDim1-CaseAlpha01}--\ref{QEquilDim1-CaseAlpha1} is satisfied, it suffices to show that $\gamma_l$ and $\gamma_r$ are optimal trajectories, which can be easily done after observing that any other $Q$-admissible trajectory $\gamma$ must satisfy $\gamma_l(t) \leq \gamma(t) \leq \gamma_r(t)$ for every $t \geq 0$.

For fixed $\ell \in [0, 1]$, one can consider $T_l$ and $T_r$ defined in \eqref{DefiTlTr} as functions of $\alpha \in [0, 1]$. Notice that, denoting by $\widetilde\gamma_l(\cdot; \alpha)$ and $\widetilde\gamma_r(\cdot; \alpha)$ the solutions of \eqref{SystTildeGamma} for a given $\alpha \in [0, 1]$, $T_l(\alpha)$ and $T_r(\alpha)$ can be characterized by the equations $\widetilde\gamma_l(T_l(\alpha); \alpha) = 0$ and $\widetilde\gamma_r(T_r(\alpha); \alpha) = 1$, and it follows from the implicit function theorem that $T_l$ and $T_r$ belong to $C^1([0, 1], \mathbb R_+)$ (after a suitable extension to a neighborhood of $[0, 1]$).

Given $\ell \in [0, 1]$, Proposition \ref{PropDimensionOneCNS} allows one to numerically approximate equilibria of $\MFG([0,\allowbreak 1],\allowbreak \{0, 1\},\allowbreak K)$ with initial condition $\delta_\ell$. Indeed, one can obtain approximations for $T_l(0)$, $T_r(0)$, $T_l(1)$, $T_r(1)$ by numerically integrating \eqref{SystTildeGamma} with $\alpha = 0$ and $\alpha = 1$ and check if conditions \ref{QEquilDim1-CaseAlpha0} or \ref{QEquilDim1-CaseAlpha1} from Proposition \ref{PropDimensionOneCNS} are satisfied. If this is the case, then one has found an equilibrium of $\MFG([0, 1],\allowbreak \{0, 1\},\allowbreak K)$. Otherwise, one has $T_l(0) < T_r(0)$ and $T_l(1) > T_r(1)$, and one can thus find $\alpha \in (0, 1)$ such that $T_l(\alpha) = T_r(\alpha)$ by using standard numerical methods.

This simulation has been done by approximating the solution of \eqref{SystTildeGamma} using an explicit Euler method with time step $\Delta t = 10^{-4}$. We have chosen $g$, $\chi$, and $\eta$ as
\begin{gather}
\label{SimulGChi}
g(x) = \frac{1}{1 + \left(\frac{2x}{15}\right)^4}, \qquad \chi(x) = 
\begin{dcases*}
\frac{1}{2 \varepsilon}\left[1 + \cos\left(\frac{\pi x}{\varepsilon}\right)\right], & if $\abs{x} < \varepsilon$, \\
0, & otherwise,
\end{dcases*} \displaybreak[0] \\ \eta(x) = 
\begin{dcases*}
\frac{1}{2}\left[1 - \cos\left(\frac{\pi\dist(x, \{0, 1\})}{\varepsilon}\right)\right], & if $\dist(x, \{0, 1\}) < \varepsilon$ and $x \in [0, 1]$, \\
1, & if $\dist(x, \{0, 1\}) \geq \varepsilon$ and $x \in [0, 1]$, \\
0, & if $x \in \mathbb R \setminus [0, 1]$,
\end{dcases*}\notag
\end{gather}
where $\varepsilon = \frac{1}{10}$. With this choice, $\eta$ is a smooth function equal to $1$ on $[\varepsilon, 1 - \varepsilon]$ and vanishing outside of the interval $[0, 1]$, allowing one to not take into account in $K$ agents who already left the domain $[0, 1]$ and discounting agents who are close to leaving. The convolution kernel $\chi$ is smooth and vanishes outside of $[-\varepsilon, \varepsilon]$, meaning that an agent only takes into account in their congestion other agents at a distance at most $\varepsilon$. Finally, the function $g$ is decreasing, meaning that higher concentrations of agents yield slower velocities. Its precise form allows one to have an important variation on the velocities for the range of agent concentrations for this problem. The functions $T_l$ and $T_r$ are represented in Figure \ref{FigTlTr} for $\ell = 0.2$ and $\ell = 0.4$.

\begin{figure}[ht]
\centering
\begin{tabular}{@{} >{\centering} m{0.5\textwidth} @{} >{\centering} m{0.5\textwidth} @{}}
\includegraphics[width=0.5\textwidth]{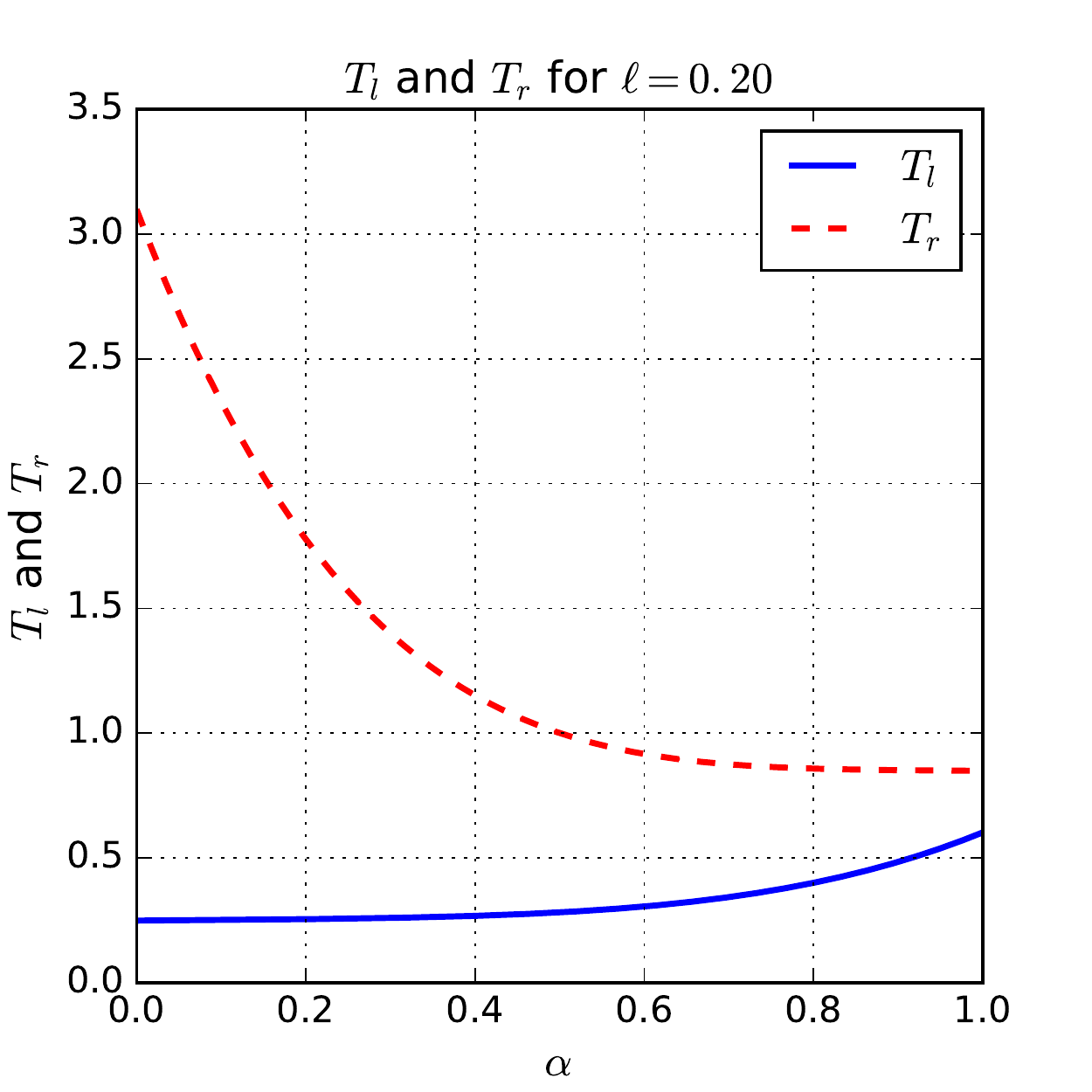} & \includegraphics[width=0.5\textwidth]{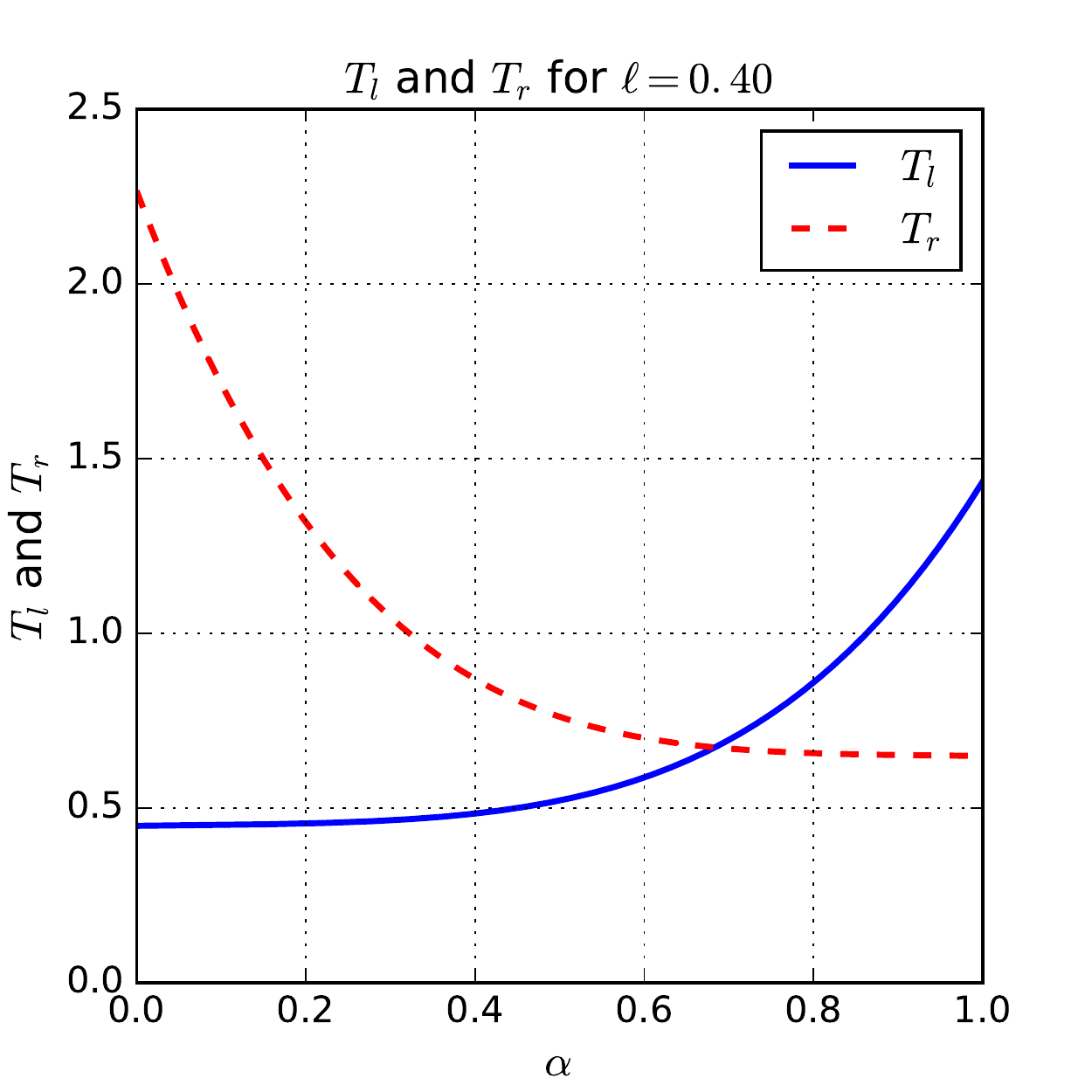} \tabularnewline
(a) & (b) \tabularnewline
\end{tabular}
\caption{Functions $T_l$ (solid line) and $T_r$ (dashed line) for (a) $\ell = 0.2$ and (b) $\ell = 0.4$.}
\label{FigTlTr}
\end{figure}

For $\ell = 0.2$, Figure \ref{FigTlTr}(a) shows that $T_l(\alpha) < T_r(\alpha)$ for every $\alpha \in [0, 1]$, and thus the unique equilibrium in this case is $Q = \delta_{\gamma_l}$, which corresponds to all agents moving left. For $\ell = 0.4$, Figure \ref{FigTlTr}(b) shows that $T_l(0) < T_r(0)$, $T_l(1) > T_r(1)$, and that there exists a unique $\alpha \approx 0.683$ such that $T_l(\alpha) = T_r(\alpha)$, giving thus a unique equilibrium where approximately $68.3 \%$ of the agents move left and $31.7 \%$ of the agents move right.

Using this method, one can compute, for each $\ell \in [0, 1]$, a value $\alpha$ for which \eqref{Dimension1-DecomposeQ} is an equilibrium, and the corresponding minimal exit time $\varphi_Q(0, \ell)$. Even though $\alpha$ need not be unique in general, as seen in Remark \ref{RemkEquilibriumNonUnique}, it seems from the simulations that, for every $\ell$, $T_l$ is increasing and $T_r$ is decreasing, and thus one has uniqueness of the equilibrium in the framework of our simulations. Figure \ref{FigAlphaT} presents the values of $\alpha$ and $\varphi_Q(0, \ell)$ at the equilibrium as functions of $\ell$ obtained from our simulations.

\begin{figure}[ht]
\centering
\begin{tabular}{@{} >{\centering} m{0.5\textwidth} @{} >{\centering} m{0.5\textwidth} @{}}
\includegraphics[width=0.5\textwidth]{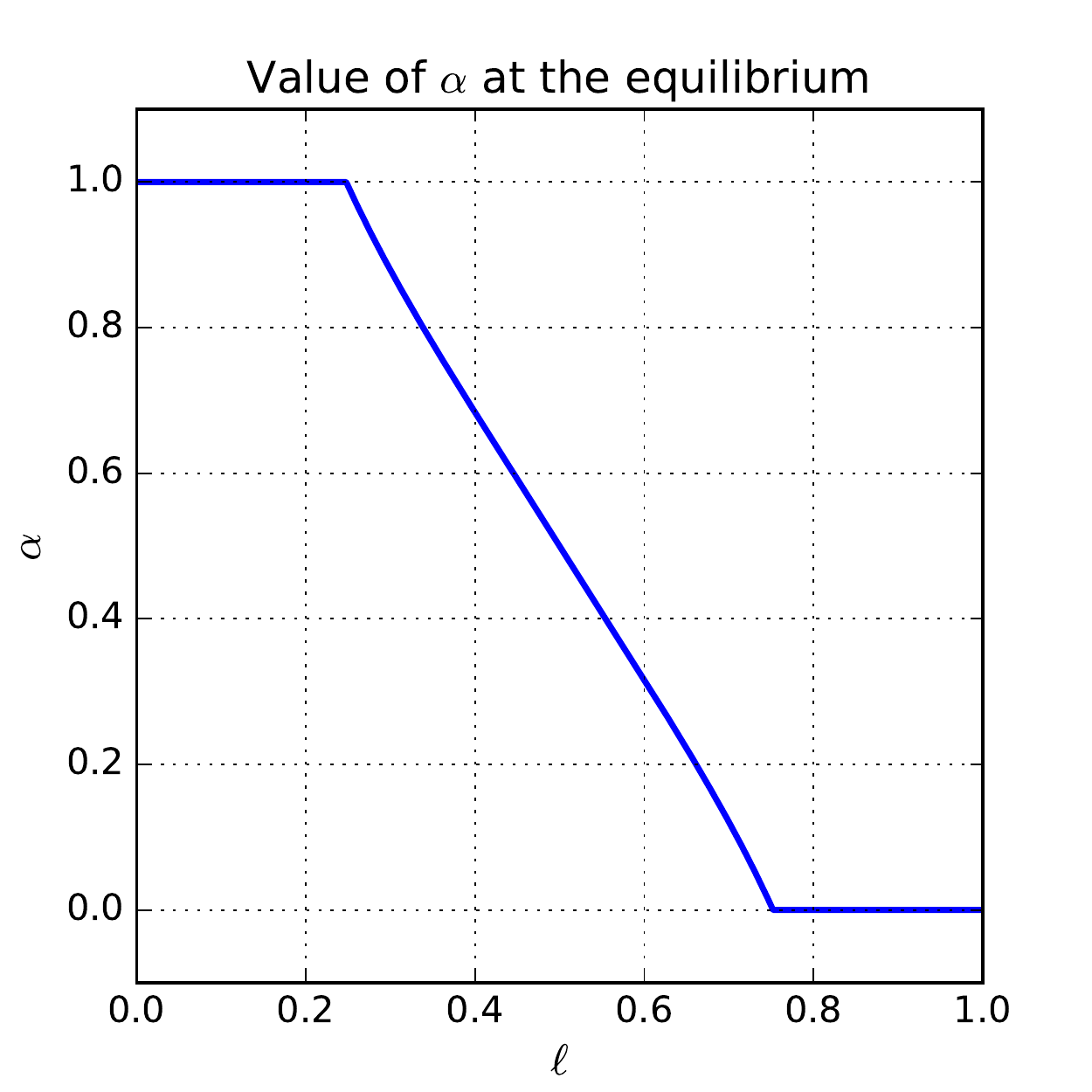} & \includegraphics[width=0.5\textwidth]{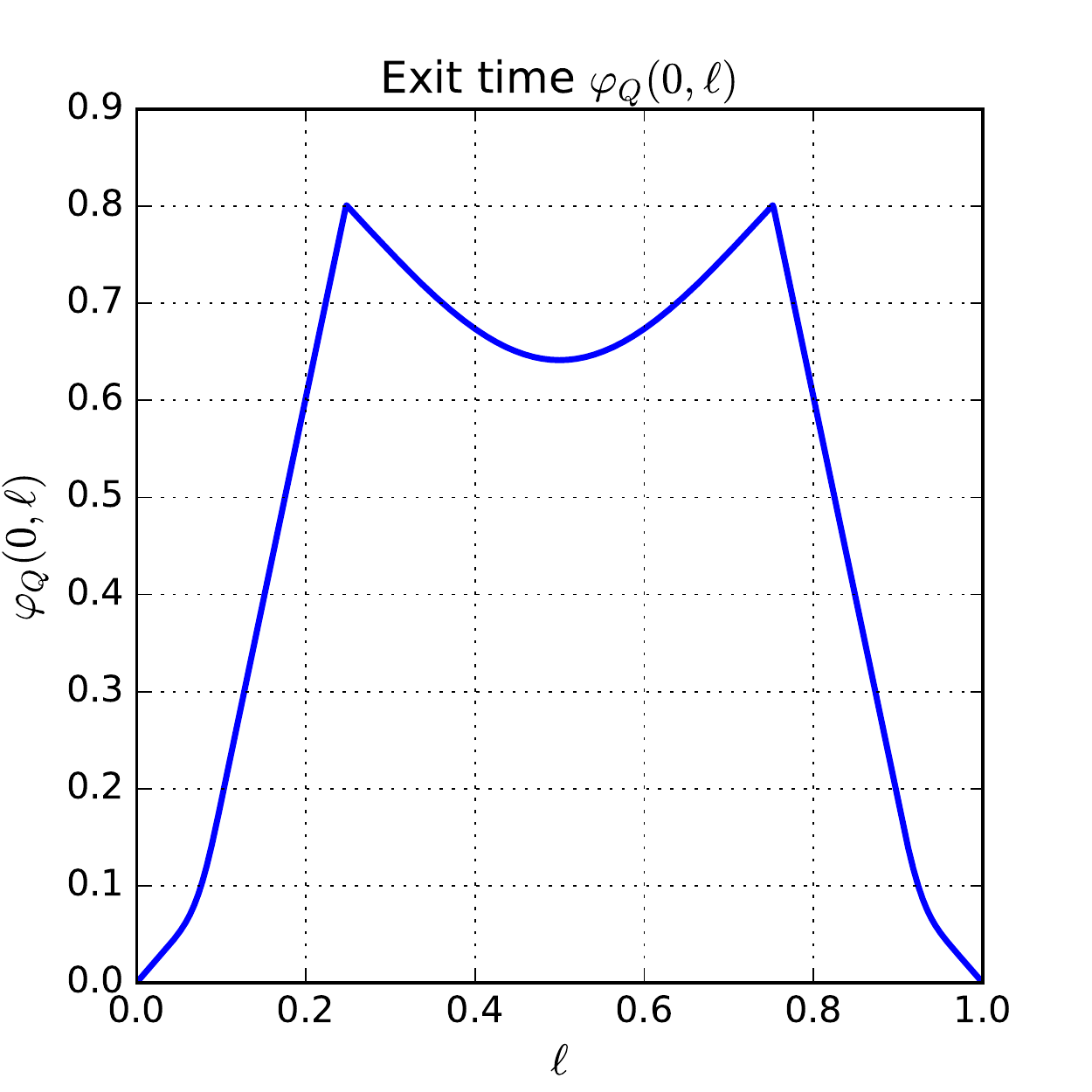} \tabularnewline
(a) & (b) \tabularnewline
\end{tabular}
\caption{Values of (a) $\alpha$ and (b) $\varphi_Q(0, \ell)$ at the equilibrium in function of $\ell$.}
\label{FigAlphaT}
\end{figure}

Figure \ref{FigAlphaT}(a) corresponds to the behavior one might intuitively expect. The represented curve is symmetric around the point $\left(\frac{1}{2}, \frac{1}{2}\right)$, which one expects since this mean field game model is symmetric with respect to the transformation $x \mapsto 1 - x$ of the interval $[0, 1]$ and by exchanging $\alpha$ and $1 - \alpha$. When $\ell$ is close to $0$, meaning that all agents are initially much closer to the exit at $0$ than to the exit at $1$, all agents move left to the exit at $0$, with a symmetric situation when $\ell$ is close to $1$. For intermediate values of $\ell$, agents split in two parts, according to Proposition \ref{PropQSplitsInTwoDeltas}, with a higher proportion of agents moving to the closer exit and a smaller proportion of agents moving to the further exit.

From Figure \ref{FigAlphaT}(b), one remarks that, close to the boundary $\{0, 1\}$, the exit time is small, increasing as one gets further away from the boundary, up to the point where agents start splitting instead of going all in the same direction. At these points where splitting starts to occur, a seemingly counter-intuitive situation happens: the further agents start from the boundary, the faster they arrive at the exit, and the points where splitting starts to occur are maxima of the exit time.

As a final remark for these simulations, notice that Figure \ref{FigAlphaT}(b) is \emph{not} the graph of the value function at time $0$, since the equilibrium $Q$ used to compute $\varphi_Q(0, \ell)$ depends on $\ell$.

\subsection{A Braess-type paradox}
\label{SecBraess}

Consider the mean field game $\MFG(X, \Gamma, K)$ where $X$ is a network whose set of vertices and edges are, respectively, $V = \{A, B, C, D\}$ and $E = \{\{A, C\},\allowbreak \{A, D\},\allowbreak \{C, D\},\allowbreak \{B, C\},\allowbreak \{B, D\}\}$, the lengths $L_e$ of the edges $e \in E$ being $L_{\{A, C\}} = L_{\{B, D\}} = 1$, $L_{\{A, D\}} = L_{\{B, C\}} = L > 1$, and $L_{\{C, D\}} = \ell \in (0, 1)$, and $\Gamma = \{B\}$. We identify each edge $e \in E$ with a segment of length $L_e$ and consider $X$ as the union of such segments with the suitable endpoints identified, endowed with its natural distance $\dist$ along the edges. This network is represented in Figure \ref{FigNetwBraess}. We are interested in numerically simulating equilibria of $\MFG(X, \Gamma, K)$ with initial condition $\delta_A$.

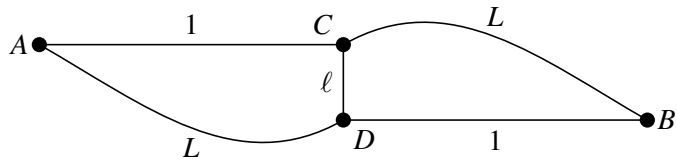
\begin{figure}[ht]
\centering
\begin{tikzpicture}
\draw[semithick] (0, -0.5) node[below right] {$D$} -- node[midway, left] {$\ell$} (0, 0.5) node[above left] {$C$} -- node[midway, above] {$1$} (-4, 0.5) node[left] {$A$} to[out=-30, in=-150] node[midway, below] {$L$} (0, -0.5) -- node[midway, below] {$1$} (4, -0.5) node[right] {$B$} to[out=150, in=30] node[midway, above] {$L$} (0, 0.5);
\fill (4, -0.5) circle[radius=0.1];
\fill (0, -0.5) circle[radius=0.1];
\fill (0,  0.5) circle[radius=0.1];
\fill (-4, 0.5) circle[radius=0.1];
\end{tikzpicture}
\caption{Network considered in the example of Section \ref{SecBraess}.}
\label{FigNetwBraess}
\end{figure}

The motivation for considering this model comes from Braess' paradox in traffic flow. The original model studied by Braess in \cite{Braess1968Paradoxon} is a static traffic flow on an oriented network similar to $X$ where agents move from $A$ to $B$ and, instead of prescribing lengths, one prescribes the travel time of an edge as an increasing affine function of its vehicular flow. The paradox consists on the fact that the total travel time from $A$ to $B$ at equilibrium may be reduced (according to how the travel times on individual edges are prescribed) when the edge $\{C, D\}$ is suppressed from the network. Similar phenomena have later been observed on other traffic flow models \cite{Taguchi1982Braess, Smith1978Road} as well as on mechanical, electrical, or computer networks \cite{Roughgarden2002How, Cohen1991Paradoxical}.

We assume in this section that $K$ is given by
\begin{equation}
\label{K-Braess}
K(\mu, x) = g\left[\int_X \chi(\dist(x, y)) \diff \mu(y)\right],
\end{equation}
with $\chi: \mathbb R_+ \to \mathbb R_+$ a smooth decreasing convolution kernel sufficiently concentrated around $0$ and $g$ a smooth decreasing function. Notice that, by Proposition \ref{PropFConvolution}, this $K$ satisfies Hypothesis \ref{MainHypo}\ref{HypoK2Lip}. Contrarily to Section \ref{SecMFGDimOne}, we provide here only an informal description of the behavior of equilibria.

Let $Q \in \mathcal P(X)$ be an equilibrium of $\MFG(X, \Gamma, K)$ with initial condition $\delta_A$ and $m = \mu^Q$. For $t > 0$, the Dirac mass at $A$ splits in two, one mass $\alpha_1 \in [0, 1]$ moving along a trajectory $\gamma_{AC}$ on the edge $\{A, C\}$, and the remaining mass $1 - \alpha_1$ moving along $\gamma_{AD}$ on the edge $\{A, D\}$. If $\gamma_{AD}$ reaches $D$ before $\gamma_{AC}$ reaches $C$, then, since $g$ is decreasing, one has $1 - \alpha_1 < \alpha_1$. Since $\dist(D, B) <\dist(C, B)$ and $1 - \alpha_1 < \alpha_1$, one concludes that the mass $1 - \alpha_1$ arriving at $D$ will reach the boundary before the mass $\alpha_1$ traveling to $C$, which contradicts the fact that $Q$ is an equilibrium. Hence, $\gamma_{AC}$ reaches $C$ before $\gamma_{AD}$ reaches $D$. Let $T_1 > 0$ be the time at which $\gamma_{AC}$ reaches $C$.

The mass $\alpha_1$ arriving at $C$ splits into a mass $\alpha_1 \alpha_2$ moving along a trajectory $\gamma_{CD}$ on the edge $\{C, D\}$ and a mass $\alpha_1 (1 - \alpha_2)$ moving along $\gamma_{CB}$ on the edge $\{B, C\}$, where $\alpha_2 \in [0, 1]$. Clearly, $\gamma_{CB}$ cannot arrive at $B$ before $\gamma_{CD}$ arrives at $D$, for otherwise $Q$ would not be an equilibrium. Let $T_2 > T_1$ be the time at which $\gamma_{CD}$ reaches $D$.

If $\gamma_{AD}$ reaches $D$ before $\gamma_{CD}$ does, then any part of mass from $\gamma_{AD}$ moving on $\{B, D\}$ would be in advance with respect to a part of the mass from $\gamma_{CD}$ moving on the same edge, and would thus arrive at $B$ before, since optimal trajectories cannot merge (see Proposition \ref{PropMonotoneOptimalTime}). Similarly, $\gamma_{CD}$ cannot reach $D$ before $\gamma_{AD}$, and thus one concludes that both must arrive at $D$ at the same time $T_2$.

At time $T_2$, one has a proportion $\alpha_1 (1 - \alpha_2)$ of the mass moving on $\gamma_{CB}$ at some point in the edge $\{B, C\}$, and a proportion $1 - \alpha_1 + \alpha_1 \alpha_2$ at the point $D$. For the latter mass, it is not optimal for any part of it to take the edges $\{A, D\}$ or $\{C, D\}$, since it would certainly arrive at $B$ after the mass moving on $\gamma_{CB}$. Hence, for $t > T_2$, the mass $1 - \alpha_1 + \alpha_1 \alpha_2$ moves along a trajectory $\gamma_{DB}$ on the edge $\{B, D\}$. Since $Q$ is an equilibrium, both masses arrive at $D$ at the same time $T > T_2$.

The previous arguments show that $m = \mu^Q$ is given by
\begin{equation}
\label{BraessExplicitMt}
m_t = 
\begin{dcases*}
\alpha_1 \delta_{\gamma_{AC}(t)} + (1 - \alpha_1) \delta_{\gamma_{AD}(t)}, & if $0 \leq t < T_1$, \\
\alpha_1 \alpha_2 \delta_{\gamma_{CD}(t)} + \alpha_1 (1 - \alpha_2) \delta_{\gamma_{CB}(t)} + (1 - \alpha_1) \delta_{\gamma_{AD}(t)}, & if $T_1 \leq t < T_2$, \\
\alpha_1 (1 - \alpha_2) \delta_{\gamma_{CB}(t)} + (1 - \alpha_1 + \alpha_1 \alpha_2) \delta_{\gamma_{DB}(t)}, & if $T_2 \leq t < T$, \\
\delta_B, & if $t \geq T$.
\end{dcases*}
\end{equation}
Notice moreover that, since agents move at maximal speed,
\begin{equation}
\label{BraessDotGamma}
\dot\gamma_i(t) = K(m_t, \gamma_i(t)) u_i(t),
\end{equation}
where $i \in \{AC, AD, CD, CB, DB\}$ and $u_i(t) \in \{-1, 1\}$, according to the direction of the movement and the orientation chosen for the edges.

The expression \eqref{BraessExplicitMt} allows one to numerically simulate the equilibrium by searching for $\alpha_1, \alpha_2 \allowbreak \in [0, 1]$ such that, when solving \eqref{BraessExplicitMt}--\eqref{BraessDotGamma}, one obtains $\gamma_{CD}(T_2) = \gamma_{AD}(T_2) = D$ and $\gamma_{DB}(T) = \gamma_{CB}(T) = B$. We describe a first method of implementing the numerical simulation, which is decomposed in two steps. As a first step, we implement a function that, for each $\alpha_2 \in [0, 1]$, finds $\alpha_1 \in [0, 1]$ such that, similarly to Proposition \ref{PropDimensionOneCNS}, one is in one of the following situations:
\begin{itemize}
\item $\alpha_1 \in (0, 1)$ and $\gamma_{CD}(T_2) = \gamma_{AD}(T_2) = D$; or
\item $\alpha_1 = 0$ and $\gamma_{AD}$ reaches $D$ before $\gamma_{CD}$; or
\item $\alpha_1 = 1$ and $\gamma_{CD}$ reaches $D$ before $\gamma_{AD}$.
\end{itemize}
The second step consists on finding $\alpha_2 \in [0, 1]$ such that, when $\alpha_1$ is computed from $\alpha_2$ using the first step, one is in one of the following situations:
\begin{itemize}
\item $\alpha_2 \in (0, 1)$ and $\gamma_{DB}(T) = \gamma_{CB}(T) = B$; or
\item $\alpha_2 = 0$ and $\gamma_{CB}$ reaches $B$ before $\gamma_{DB}$; or
\item $\alpha_2 = 1$ and $\gamma_{DB}$ reaches $B$ before $\gamma_{CB}$.
\end{itemize}
The searches for $\alpha_1$ and $\alpha_2$ in both steps are implemented using a bisection method. Notice that, in the case $\alpha_2 = 0$, the only possible equilibrium is when $\alpha_1 = \frac{1}{2}$.

A second method of implementing the numerical simulation, which is much faster, can be obtained if one is in a situation where the equilibrium is unique. Indeed, notice that, by transforming $t$ into $T - t$, one obtains an equilibrium of a mean field game on the same network with initial condition $\delta_B$ and exit at $\Gamma = \{A\}$. Up to relabeling the vertices, this is an equilibrium of $\MFG(X, \Gamma, K)$, and, by uniqueness, one concludes that $\alpha_1 = 1 - \alpha_1 + \alpha_1 \alpha_2$ and $1 - \alpha_1 = \alpha_1(1 - \alpha_2)$, yielding the relation
\begin{equation}
\label{BraessRelationAlphas}
\alpha_1 = \frac{1}{2 - \alpha_2}.
\end{equation}
Hence, the second method for the simulation consists on finding $\alpha_2 \in [0, 1]$ such that, if $\alpha_1$ is computed from $\alpha_2$ by \eqref{BraessRelationAlphas}, then, similarly to Proposition \ref{PropDimensionOneCNS}, one is in one of the following situations:
\begin{itemize}
\item $\alpha_2 \in (0, 1)$ and $\gamma_{CD}(T_2) = \gamma_{AD}(T_2) = D$; or
\item $\alpha_2 = 0$ and $\gamma_{AD}$ reaches $D$ before $\gamma_{CD}$; or
\item $\alpha_2 = 1$ and $\gamma_{CD}$ reaches $D$ before $\gamma_{AD}$.
\end{itemize}
Notice that it suffices here to guarantee that $\gamma_{CD}$ and $\gamma_{AD}$ arrive at $D$ at the same time (or one of the corresponding masses is zero), since \eqref{BraessRelationAlphas} guarantees the symmetry of the equilibrium.

We have performed this simulation for $L = 1.25$ and different values of $\ell$ on the interval $(0, 0.35]$, approximating the solutions of \eqref{BraessDotGamma} by an explicit Euler method with time step $\Delta t = 10^{-4}$. The functions $g$ and $\chi$ we chose were those from \eqref{SimulGChi} with $\varepsilon = \frac{1}{10}$. We have first selected a few values of $\ell$ and computed an equilibrium using the first method, verifying that the values of $\alpha_1$ and $\alpha_2$ at equilibrium satisfy \eqref{BraessRelationAlphas}. We have then used the second method to simulate the equilibrium for $350$ equally spaced values of $\ell$ on the interval $(0, 0.35]$. Figure \ref{FigEquilBraess} presents the results of this simulation.

\begin{figure}[ht]
\centering
\begin{tabular}{@{} >{\centering} m{0.5\textwidth} @{} >{\centering} m{0.5\textwidth} @{}}
\includegraphics[width=0.5\textwidth]{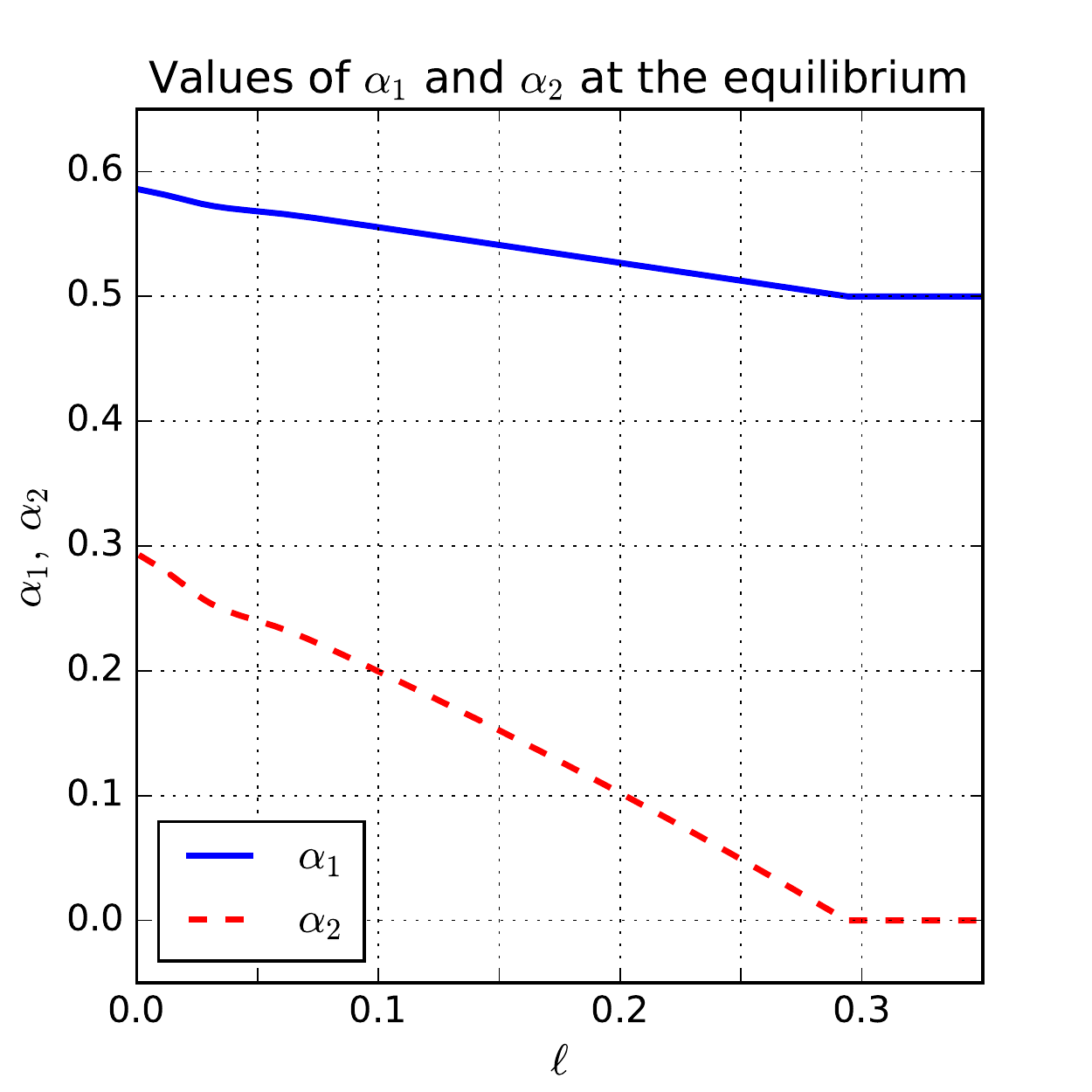} & \includegraphics[width=0.5\textwidth]{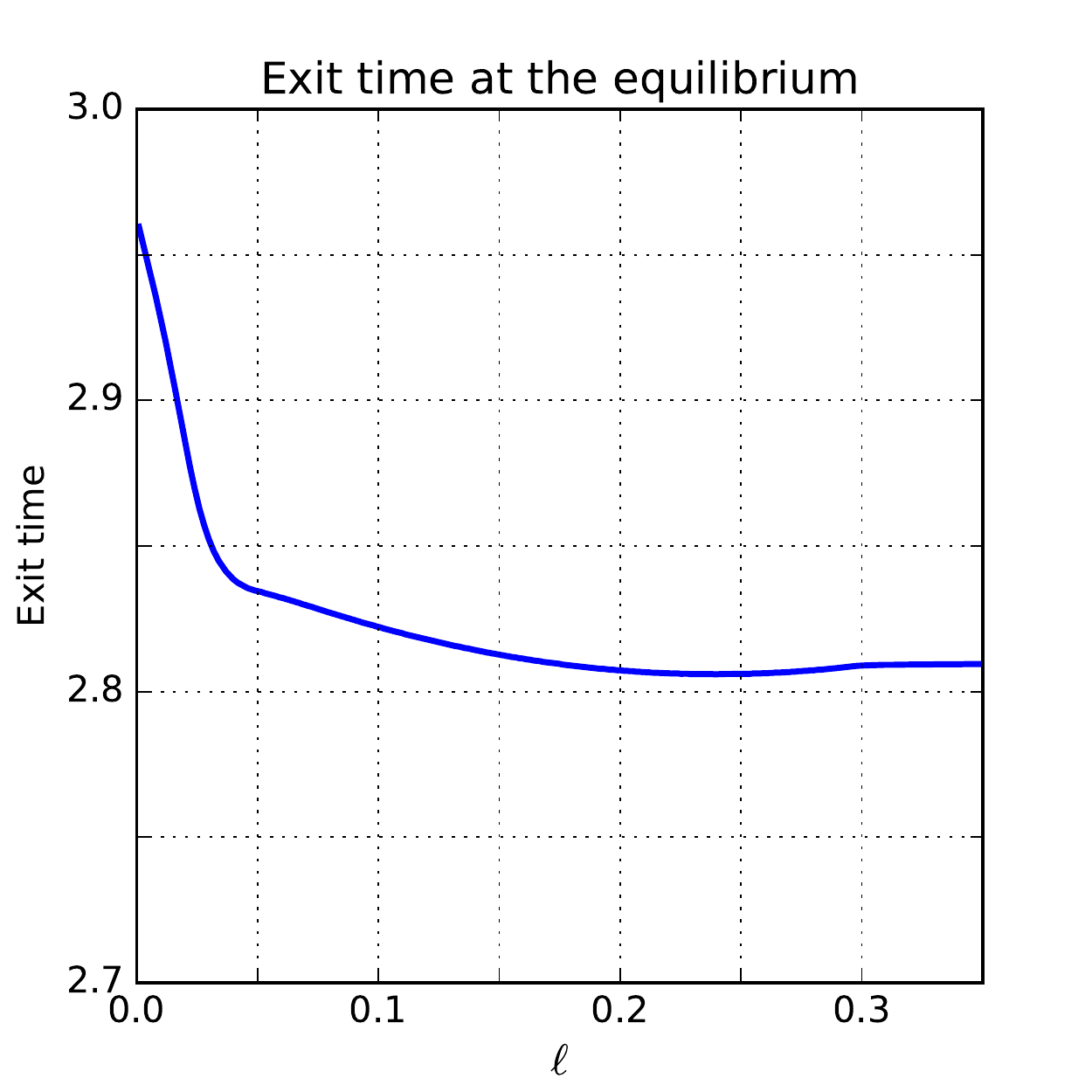} \tabularnewline
(a) & (b) \tabularnewline
\end{tabular}
\caption{Values of (a) $\alpha_1$ and $\alpha_2$ and (b) the exit time at the equilibrium in function of the length $\ell$ of the edge $CD$.}
\label{FigEquilBraess}
\end{figure}

We observe in Figure \ref{FigEquilBraess}(a) that, as expected, $\alpha_2$ decreases with $\ell$, eventually reaching zero when $\ell \approx 0.294$: as the length of $CD$ increases, fewer people will go through it, until $CD$ is so long that no agent wants to move on this edge. The value of $\alpha_1$ also decreases with $\ell$, since it increases with $\alpha_2$: when the length of $CD$ is small, people anticipate the fact that they will go through the edge $CD$ and so a higher proportion of people start by taking the edge $AC$ in order to move through $CD$ later. When $\ell \geq 0.294$, no one travels on the edge $CD$, and the equilibrium is the same as the one we would get if the edge $CD$ did not exist and we had only two edges connecting $A$ to $B$, both of length $1 + L$.

As regards the behavior of the exit time at the equilibrium, shown in Figure \ref{FigEquilBraess}(b), for $\ell \geq 0.294$, the exit time is constant, which is expected since no agents move on $CD$ and hence the presence and the length of $CD$ do not influence the equilibrium. As $\ell$ decreases, one observes a small decrease of the exit time, until it reaches its global minimum at $\ell_\ast \approx 0.239$. As $\ell$ decreases below $\ell_\ast$, the exit time increases (very sharply when $\ell < 0.05$), which may at first look counter-intuitive since decreasing $\ell$ means decreasing the length of the shortest path from $A$ to $B$, which is the sequence of edges $(AC, CD, DB)$ when $\ell < 0.25$. This exactly corresponds to a Braess-type paradox.

As in the original Braess paradox, this situation can be explained by the congestion terms. Indeed, as $\ell$ decreases below $\ell_\ast$, more agents will go through the path $CD$, which means that more agents will also move through $AC$ and $DB$, increasing congestion on those edges.

\subsection{Mean field game on a disk}

As a last example, consider the mean field game $\MFG(X, \Gamma, K)$ with $\Omega = B_d(0, 1)$, $X = \overline\Omega$, $\Gamma = \partial\Omega = \mathbb S^{d-1}$, and $K$ satisfying Hypothesis \ref{MainHypo}\ref{HypoKConvolution}. We assume further that the functions $\chi$ and $\eta$ from Hypothesis \ref{MainHypo}\ref{HypoKConvolution} are radial, i.e., there exist $\chi_0, \eta_0: \mathbb R_+ \to \mathbb R_+$ such that $\chi(x) = \chi_0(\abs{x})$ and $\eta(x) = \eta_0(\abs{x})$ for every $x \in \mathbb R^d$. When the initial condition is a Dirac mass at the center of the ball, one can obtain an explicit expression for an equilibrium, given in the following proposition.

\begin{prop}
\label{MFGBall-SymmetricEquilibrium}
Let $K$, $g$, $\chi$, and $\eta$ satisfy Hypothesis \ref{MainHypo}\ref{HypoKConvolution} and assume that there exist $\chi_0, \eta_0: \mathbb R_+ \to \mathbb R_+$ such that $\chi$ and $\eta$ are given by $\chi(x) = \chi_0(\abs{x})$ and $\eta(x) = \eta_0(\abs{x})$ for every $x \in \mathbb R^d$. Let $e \in \mathbb S^{d-1}$, $\nu$ be the normalized uniform measure on $\mathbb S^{d-1}$, and $\widetilde r: \mathbb R_+ \to \mathbb R_+$ be the unique solution of
\begin{equation}
\label{DefTildeR}
\left\{
\begin{aligned}
\dot{\widetilde r}(t) & = g\left[\eta_0(\widetilde r(t)) \int_{\mathbb S^{d-1}} \chi_0(\widetilde r(t) \abs{e - \omega}) \diff \nu(\omega)\right], \\
\widetilde r(0) & = 0.
\end{aligned}
\right.
\end{equation}
Then $\widetilde r$ does not depend on $e \in \mathbb S^{d-1}$. Define $r: \mathbb R_+ \to [0, 1]$ by $r(t) = \min\{\widetilde r(t), 1\}$ and let $R: \mathbb S^{d-1} \to \mathcal C_{\overline\Omega}$ be the function defined by $R(\omega)(t) = r(t) \omega$ for every $\omega \in \mathbb S^{d-1}$ and $t \in \mathbb R_+$. Then $Q = R_{\#} \nu$ is an equilibrium of $\MFG(X, \Gamma, K)$ with initial condition $\delta_0$, and, letting $m = \mu^Q$, $m_t$ is the normalized uniform measure on the sphere of radius $r(t)$ and centered at $0$ for every $t > 0$.
\end{prop}

As in the proof of Proposition \ref{PropDimensionOneCNS}, the main idea for the proof of Proposition \ref{MFGBall-SymmetricEquilibrium} is to show that $R(\omega)$ is an optimal trajectory for every $\omega \in \mathbb S^{d-1}$, which is done by observing that $\abs{R(\omega)(t)} = r(t)$ and $\abs{\gamma(t)} \leq r(t)$ for every admissible trajectory $\gamma$ and every $t \geq 0$.

Notice that the equilibrium from Proposition \ref{MFGBall-SymmetricEquilibrium} is not necessarily unique. Indeed, if $g$ is constant and equal to $1$, then the measure $Q_\omega$ concentrated on the curve $t \mapsto \min\{t, 1\} \omega$ is an equilibrium for every $\omega \in \mathbb S^{d-1}$.

We are now interested in the case where the initial condition is $\delta_p$ for some point $p \in B_d(0, 1)$. Notice that, when $g = 1$ and $p \neq 0$, there exists a unique equilibrium $Q$ of $\MFG(X, \Gamma, K)$ with initial condition $\delta_p$, which is concentrated on the segment joining $p$ to its closest point on the boundary with unit speed. However, when $g$ is decreasing and $\chi$ is a sufficiently concentrated radially decreasing convolution kernel, one may expect that agents will avoid congestion, concentrating on a surface evolving in time, and arriving at the boundary at a same time. Moreover, if $p$ is close to the origin, one may expect to find equilibria with closed surfaces similar to the circles from the equilibrium of Proposition \ref{MFGBall-SymmetricEquilibrium}.

Since an explicit characterization of an equilibrium similar to Proposition \ref{MFGBall-SymmetricEquilibrium} is too hard to obtain with initial condition $\delta_p$, $p \neq 0$, we validate the previous intuition of the behavior of equilibria by a numerical simulation. The idea for our simulation goes as follows.

Given an equilibrium $Q$, let $\mu$ be the measure on $\mathbb S^{d-1}$ obtained as the pushforward of $Q$ by the map that, to each optimal trajectory $\gamma$, associates the value of the corresponding optimal control at time $0$, $u(0)$. If $\mu$ is known, then one can obtain an approximation of $m_t$ using \eqref{SystGammaURegular}. Indeed, for $N \in \mathbb N$, we can approximate $\mu$ as a sum of finitely many Dirac masses, $\mu^N = \sum_{k=1}^N \mu_k \delta_{z_k}$, where $z_k \in \mathbb S^{d-1}$ and $\mu_k \in [0, 1]$, $\sum_{k = 1}^N \mu_k = 1$. One can thus compute $N$ trajectories $\gamma_1, \dotsc, \gamma_N$ and their associated optimal controls $u_1, \dotsc, u_N$ by solving \eqref{SystGammaURegular} with initial conditions $\gamma_k(0) = p$, $u_k(0) = z_k$ for $k \in \{1, \dotsc, N\}$, where we approximate $m_t$ by $m_t^N = \sum_{k=1}^N \mu_k \delta_{\gamma_k(t)}$. Since all optimal trajectories arrive at $\partial\Omega$ at the same time, one expects that the approximated trajectories $\gamma_1, \dotsc, \gamma_N$ corresponding to a weight $\mu_k > 0$ will arrive at $\partial\Omega$ at approximately the same time.

One may thus search for an equilibrium by searching for $\mu \in \mathcal P(\mathbb S^{d-1})$ such that the above approximated trajectories with positive weigth arrive at $\partial\Omega$ at the same time. This motivates the introduction of Algorithm \ref{Algo}, based on a fixed-point strategy, to compute an equilibrium for $\MFG(X, \Gamma, K)$ with initial condition $\delta_p$.

\begin{algorithm}
\caption{Algorithm used to find an equilibrium for $\MFG(\overline\Omega, \partial\Omega, K)$ with initial condition $\delta_p$.}
\label{Algo}
\centering
\tikzstyle{block} = [rectangle, rounded corners=3mm, very thick, font=\small, draw=blue!50!white, fill=blue!10!white]
\tikzstyle{decision} = [diamond, shape aspect=2.5, very thick, font=\small, draw=red!50!white, fill=red!10!white, inner sep=0mm]

\begin{tikzpicture}
\node[block] (step1) at (0, 0) {\parbox[c][\height][c]{4cm}{\centering Choose an initial measure $\mu \in \mathcal P(\mathbb S^{d-1})$ for $u(0)$.}};
\node[block] (step2) at (0, -1.8) {\parbox[c][\height][c]{4cm}{\centering Compute $N$ optimal trajectories using \eqref{SystGammaURegular}.}};
\node[decision] (test) at (0, -4.5) {\parbox[c][\height][c]{4cm}{\centering Do all trajectories with positive mass arrive at $\partial\Omega$ at the same time?}};
\node[block] (end) at (0, -7.5) {\parbox[c][\height][c]{4cm}{\centering $\mu$ yields an equilibrium.}};
\node[block] (loop) at (-5, -2.5) {\parbox[c][\height][c]{4cm}{\centering Choose another $\mu$, putting more mass on those who arrived early and less mass on those who arrived late.}};
\node at (7, 0) {};
\draw[thick, -Stealth] (step1) -- (step2);
\draw[thick, -Stealth] (step2) -- (test);
\draw[thick, -Stealth] (test) -- node[midway, right] {\small Yes} (end);
\draw[thick, -Stealth] (test) -- node[midway, below] {\small No} +(-5, 0) -- (loop);
\draw[thick, -Stealth] (loop) -- (-5, -0.7) -- (0, -0.7);
\end{tikzpicture}
\end{algorithm}

We have implemented Algorithm \ref{Algo} in dimension $d = 2$, identifying the plane $\mathbb R^2$ with the complex plane $\mathbb C$ for simplicity, and choosing $g$, $\chi_0$, and $\eta_0$ as
\begin{gather*}
g(x) = \frac{1}{1 + \left(\frac{x}{2}\right)^2}, \qquad \chi_0(x) = 
\begin{dcases*}
\frac{1}{2 \varepsilon}\left[1 + \cos\left(\frac{\pi x}{\varepsilon}\right)\right], & if $0 \leq x < \varepsilon$, \\
0, & if $x \geq \varepsilon$,
\end{dcases*} \displaybreak[0] \\ \eta_0(x) = 
\begin{dcases*}
1, & if $0 \leq x < 1 - \varepsilon$, \\
\frac{1}{2}\left[1 - \cos\left(\frac{\pi (1 - x)}{\varepsilon}\right)\right], & if $1 - \varepsilon \leq x < 1$, \\
0, & if $x \geq 1$,
\end{dcases*}
\end{gather*}
with $\varepsilon = \frac{1}{10}$. Figure \ref{FigSimulCircle} presents the results of the numerical simulations, showing the densities, with respect to the normalized uniform measure on $\mathbb S^1$, of the measures $\mu$ yielding an equilibrium for different initial conditions of the form $\delta_p$. Figure \ref{FigSimulCircleTrajectories} shows the support of the measures $m_t$ and some optimal trajectories at different times for the initial condition $\delta_{(0.1, 0)}$. Before commenting on the results shown in Figures \ref{FigSimulCircle} and \ref{FigSimulCircleTrajectories}, let us describe some details of the implementation of Algorithm \ref{Algo}.

\begin{figure}[ht]
\centering
\begin{tabular}{@{} >{\centering} m{0.5\textwidth} @{} >{\centering} m{0.5\textwidth} @{}}
\includegraphics[width=0.5\textwidth]{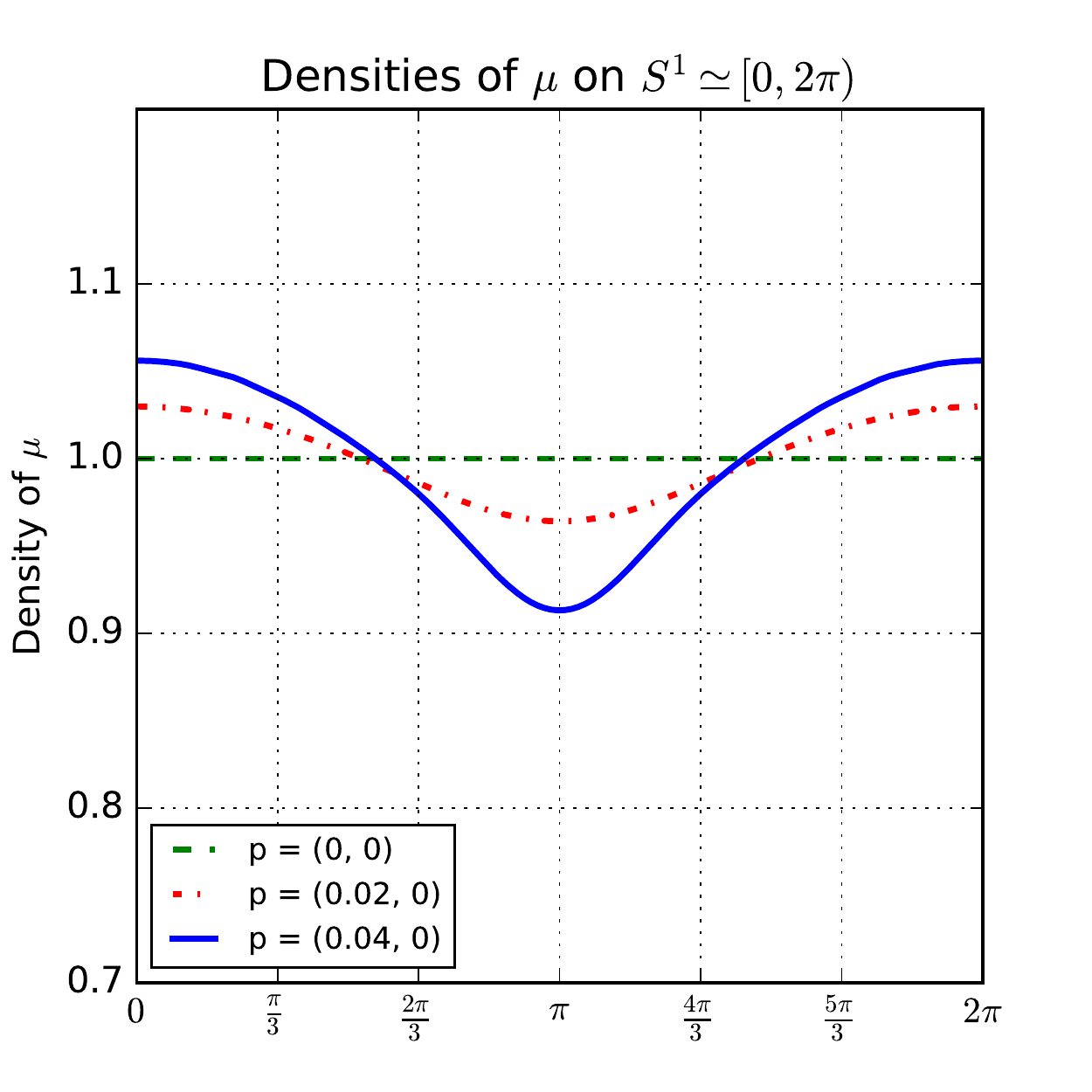} & \includegraphics[width=0.5\textwidth]{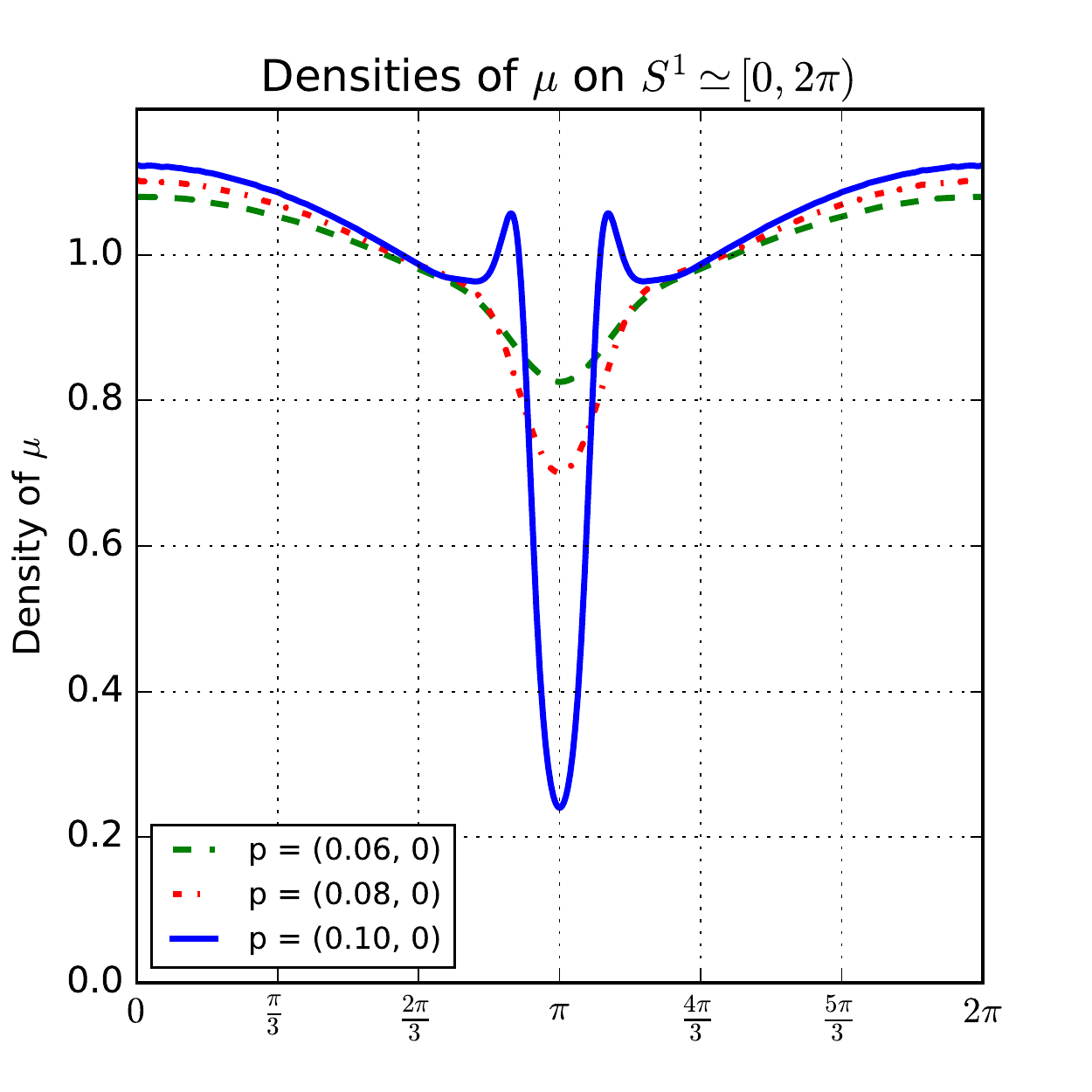} \tabularnewline
(a) & (b) \tabularnewline
\end{tabular}
\caption{Densities of the measures $\mu$ yielding an equilibrium with initial condition $\delta_p$ for (a) $p = (0, 0)$, $p = (0.02, 0)$, and $p = (0.04, 0)$, and (b) $p = (0.06, 0)$, $p = (0.08, 0)$, and $p = (0.1, 0)$. Notice that the scales of the vertical axes are not the same.}
\label{FigSimulCircle}
\end{figure}

\begin{figure}[ht]
\centering
\begin{tabular}{@{} >{\centering} m{0.5\textwidth} @{} >{\centering} m{0.5\textwidth} @{}}
\includegraphics[width=0.5\textwidth]{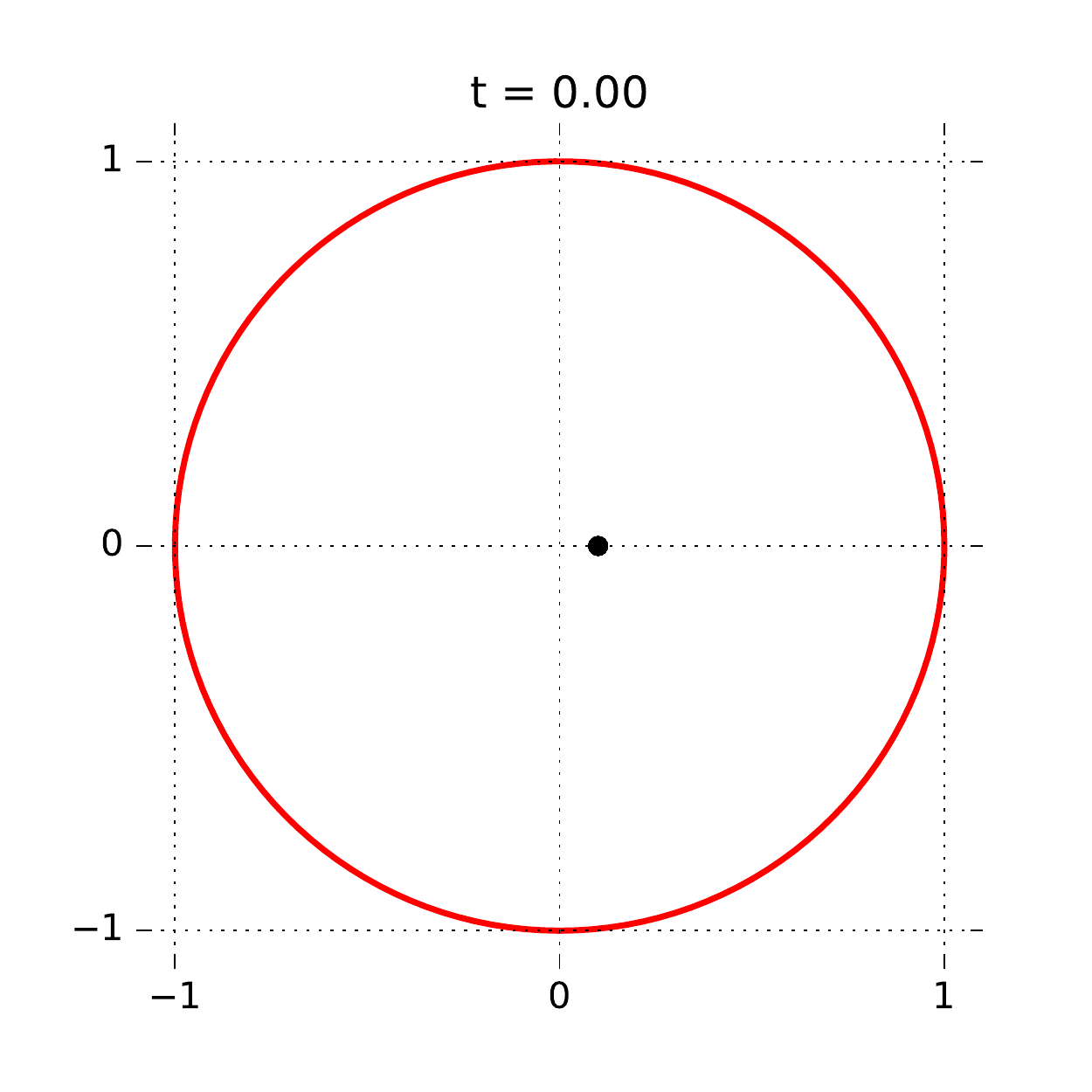} & \includegraphics[width=0.5\textwidth]{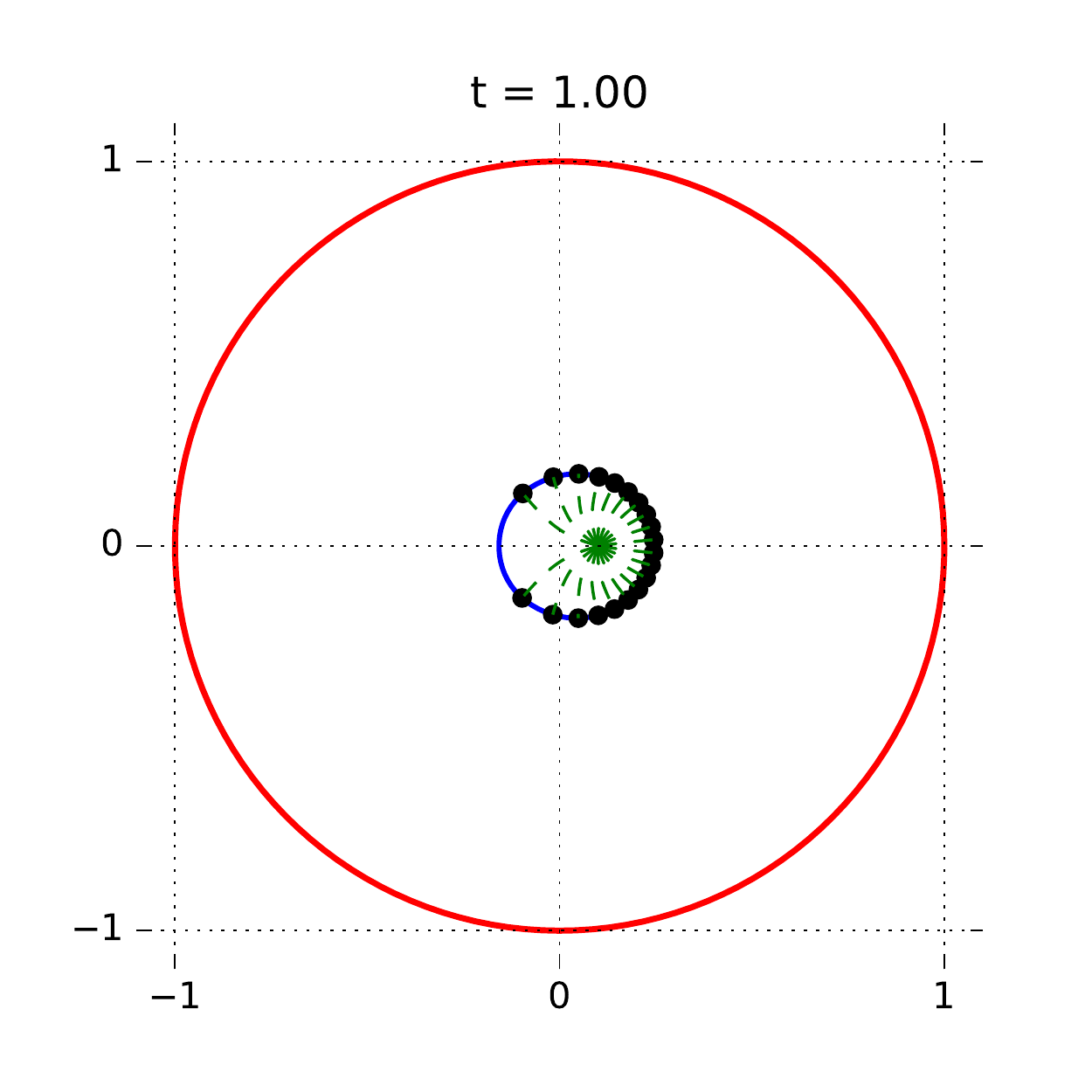} \tabularnewline
(a) & (b) \tabularnewline
\includegraphics[width=0.5\textwidth]{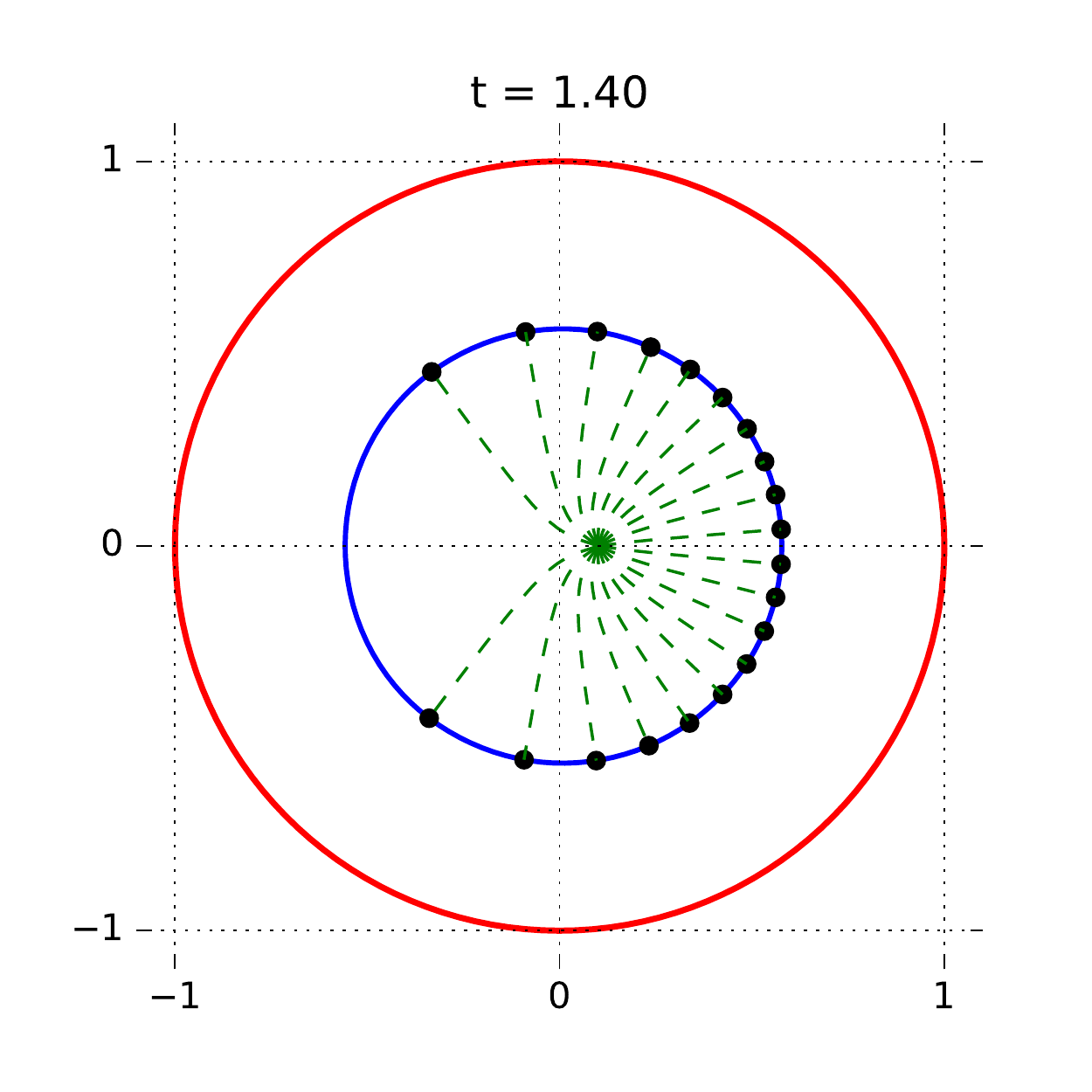} & \includegraphics[width=0.5\textwidth]{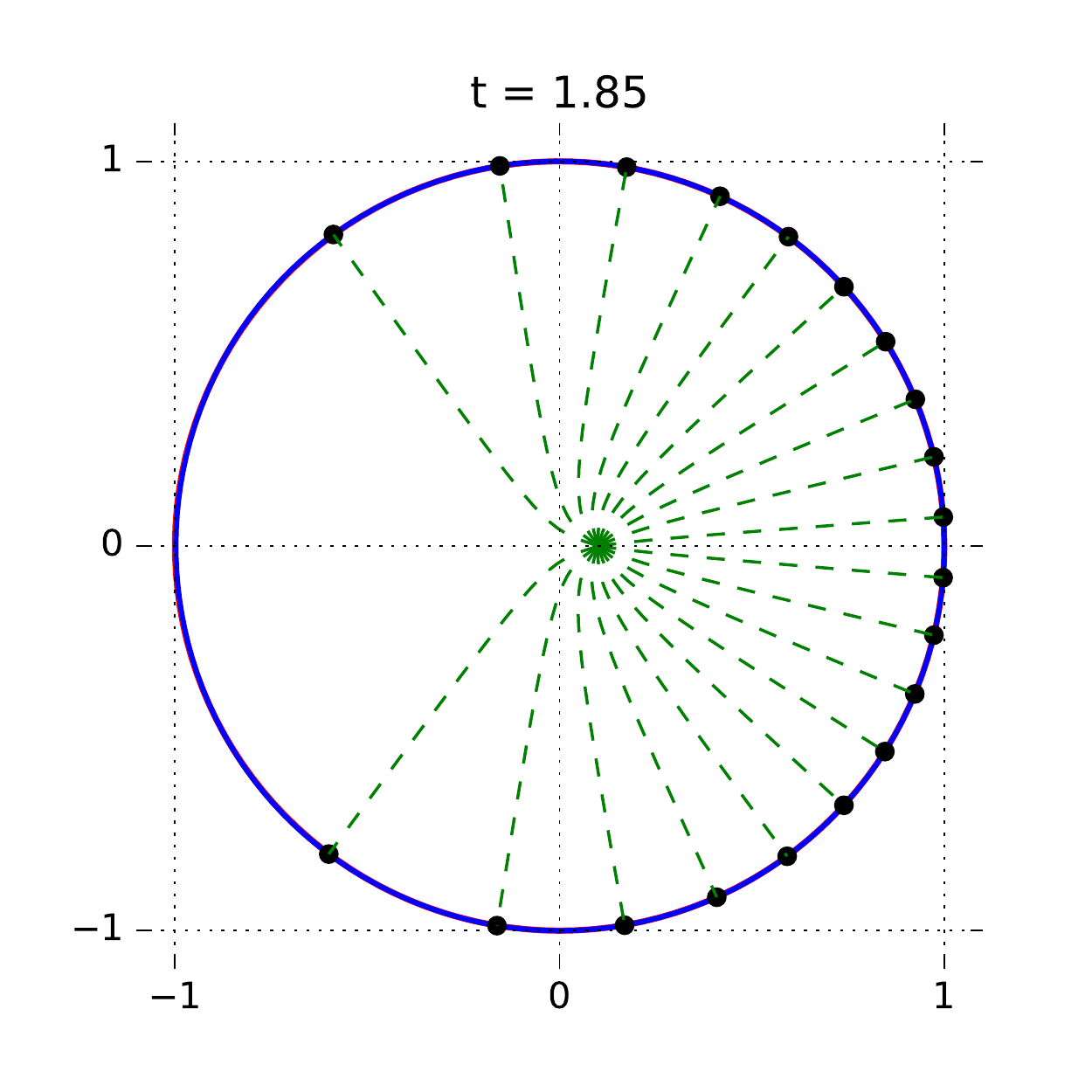} \tabularnewline
(c) & (d) \tabularnewline
\end{tabular}
\caption{Support of $m_t$ and optimal trajectories for $p = (0.1, 0)$ at times (a) $t = 0$, (b) $t = 1$, (c) $t = 1.4$, and (d) the final time $t \approx 1.85$.}
\label{FigSimulCircleTrajectories}
\end{figure}

The choice of the initial measure $\mu \in \mathcal P(\mathbb S^{d-1})$ is important: from the numerical simulations, the algorithm seems to converge only if the initial measure is sufficiently close to a measure yielding an equilibrium. When $p$ is close to $0$, a good choice, based on Proposition \ref{MFGBall-SymmetricEquilibrium}, is to take $\mu$ as the normalized uniform measure on $\mathbb S^{d-1}$. For $p$ further from $0$, a possible technique is to choose points $p_0, p_1, \dotsc, p_n$ with $p_0 = 0$, $p_n = p$, and $\abs{p_i - p_{i-1}}$ small for every $i \in \{1, \dotsc, n\}$, and apply the algorithm to compute successively an equilibrium with initial condition $\delta_{p_i}$ using as initial measure the one that yields an equilibrium with initial condition $\delta_{p_{i-1}}$ (in the spirit of the so-called {\it continuation method}).

The initial measure $\mu$ was numerically approximated by a finite sum of $N$ equally spaced Dirac masses on $\mathbb S^1$, $\mu \approx \sum_{k=1}^N \mu_k \delta_{z_k}$, where $z_k = e^{i \frac{2\pi}{N}\left(k + \frac{1}{2}\right)}$ and $\mu_k = \mu(\{e^{i\theta} \mid \frac{2\pi}{N}k \leq \theta < \frac{2\pi}{N}\left(k + 1\right)\})$ for $k \in \{1, \dotsc, N\}$. For our simulations, we chose $N = 1000$. In order to compute an equilibrium for the initial condition $\delta_p$ with $p = (0.1, 0)$, as described before, we have computed equilibria with initial conditions $\delta_{p_k}$ with $p_k = \left(\frac{k}{100}, 0\right)$ for $k \in \{1, \dotsc, 10\}$, using at each step the measure $\mu$ yielding an equilibrium for $p_k$ as an initial measure for $p_{k+1}$.

Concerning the computations of the optimal trajectories from \eqref{SystGammaURegular}, the time discretization was done by an explicit Euler method with time step $\Delta t = 10^{-2}$. The simulation is stopped at the time $T$ when the first trajectory $\gamma_k$ with positive mass reaches the boundary, and we say that all trajectories have reached the boundary at the same time if $\max_{k \in \{1, \dotsc, N\} \mid \mu_k > 0} \abs{\abs{\gamma_k(T)} - 1} \allowbreak < \zeta$, where the tolerance $\zeta$ was chosen as $\zeta = 3 \cdot 10^{-3}$.

The intuition for the construction of a new initial measure $\mu^\prime$ from $\mu$ and the final positions $(\gamma_k(T))_{k = 1}^N$ is that $\mu^\prime$ should contain less mass at the points that ended further from the boundary. To do so, our choice was to first construct a measure $\mu^{\prime\prime} = \sum_{k=1}^N \mu_{k}^{\prime\prime} \delta_{z_k}$ given by
\begin{equation}
\label{muPrimePrime}
\mu_k^{\prime\prime} = \frac{\mu_k \abs{\gamma_k(T)}}{\sum_{k=1}^N \mu_j \abs{\gamma_j(T)}},
\end{equation}
and then define $\mu^\prime = \sum_{k=1}^N \mu_k^\prime \delta_{z_k}$ by
\begin{equation}
\label{muPrime}
\mu_k^\prime = \frac{\mu_k + \mu_k^{\prime\prime}}{2}.
\end{equation}
The construction of $\mu^{\prime\prime}$ in \eqref{muPrimePrime} uses the condition $\abs{\gamma_k(T)} < 1$ for trajectories that did not arrive at the boundary as a mean to reduce the mass given to such trajectories. The renormalization will then increase the mass of trajectories that arrived first at the boundary. The average computed in \eqref{muPrime} was used in order to improve stability of the method, even though it reduces convergence speed.

Figure \ref{FigSimulCircle} shows that, as expected, the measure $\mu$ yielding an equilibrium is uniform when $p = (0, 0)$ and its density around the angle $\pi$ decreases as the $x$-component of $p$ increases, with a very sharp decrease when $p = (0.1, 0)$. Figure \ref{FigSimulCircleTrajectories} shows the behavior of the support of $m_t^N$ and 20 optimal trajectories at different times from $0$ until the final time $t \approx 1.85$. The 20 trajectories represented were chosen in such a way that the total mass of agents between any two neighbor trajectories is the same, and thus their positions at time $t$ provide an illustration of the measure $m_t^N$. One can observe that, even though $p$ is close to the origin in this case, agents are much more concentrated on the right half-ball, with a few proportion of agents moving into the left half-ball.

\bibliographystyle{abbrv}
\bibliography{Bib}

\section*{Index of notations}

The following table provides the main notations used in this paper, together with their brief descriptions and the pages of their definitions.

\renewcommand{\arraystretch}{1.33}
\begin{longtable}{@{\hspace*{0.01\textwidth}} >{\raggedright} m{0.16\textwidth} @{\hspace*{0.02\textwidth}} m{0.7\textwidth} @{\hspace*{0.02\textwidth}} >{\centering} m{0.08\textwidth} @{\hspace*{0.01\textwidth}}}
\hline\hline
Notation & Description & Page \tabularnewline
\hline\hline
\endhead
\hline\hline
\endfoot
$\mathcal C_X$ & $\mathcal C(\mathbb R_+, X)$ with the topology of uniform convergence on compact sets. & \pageref{MathcalCX} \tabularnewline
\hline
$\Lip_c(X)$ & Set of $c$-Lipschitz continuous functions in $\mathcal C_X$. & \pageref{LipcX} \tabularnewline
\hline
$e_t$ & Evaluation map, associates with a curve $\gamma$ its value at time $t$. & \pageref{DefiET} \tabularnewline
\hline
$\mathcal P(X)$ & Set of Borel probability measures on $X$. & \pageref{MathcalPX} \tabularnewline
\hline
$W_1$ & Wasserstein distance on $\mathcal P(X)$. & \pageref{DefiW1} \tabularnewline
\hline
$\abs{\dot\gamma}$ & Metric derivative of $\gamma$. & \pageref{AbsDotGamma} \tabularnewline
\hline
$\gengrad \phi$ & Generalized gradient of $\phi$. & \pageref{GenGradPhi} \tabularnewline
\hline
$\MFG(X, \Gamma, K)$ & Minimal-time mean field game in the space $X$ with target set $\Gamma$ and maximal speed given by the function $K$. & \pageref{MFGXGammaK} \tabularnewline
\hline
$\OCP(X, \Gamma, k)$ & Minimal-time optimal control problem in the space $X$ with target set $\Gamma$ and maximal speed given by the function $k$. & \pageref{OCPXGammak} \tabularnewline
\hline
$\Adm(k)$\newline $\Adm(m)$\newline $\Adm(Q)$ & Sets of $k$-, $m$-, and $\mu^Q$-admissible curves, respectively. & \pageref{DefiAdmk}, \pageref{DefiAdmm}, \pageref{DefiAdmQ} \tabularnewline
\hline
$\tau(t_0, \gamma)$ & First exit time after $t_0$ of the curve $\gamma$. & \pageref{Taut0gamma} \tabularnewline
\hline
$\Opt(k, t_0, x_0)$\newline $\Opt(m, t_0, x_0)$\newline $\Opt(Q, t_0, x_0)$ & Set of optimal curves for $k$, $m$, and $\mu^Q$, respectively, starting at the position $x_0$ at time $t_0$. & \pageref{Optkt0x0}, \pageref{Optmt0x0}, \pageref{OptQt0x0} \tabularnewline
\hline
$\varphi(t, x)$ & Value function of $\OCP(X, \Gamma, k)$. & \pageref{DefiVarphi} \tabularnewline
\hline
$D^+ w(x)$ & Superdifferential of the function $x$ at the point $x$. & \pageref{DefiSuperdifferential} \tabularnewline
\hline
$\mathcal U_{(t_0, x_0)}$ & Set of optimal directions at $(t_0, x_0)$. & \pageref{DefiOptDirections} \tabularnewline
\hline
$\widehat{\nabla w}(t, x)$ & Normalized gradient of the function $w$ with respect to $x$ at $(t, x)$. & \pageref{WidehatNabla} \tabularnewline
\hline
$\Upsilon$ & Set of points $(t, x)$ which are not starting points of optimal trajectories. & \pageref{DefiUpsilon} \tabularnewline
\hline
$\mathcal Q$ & Set of measures in $\mathcal P(\mathcal C_X)$ concentrated on $K_{\max}$-Lipschitz continuous curves. & \pageref{DefiMathcalQ} \tabularnewline
\hline
$\OOpt(Q)$ & Set of optimal trajectories for $Q$ starting at time $0$. & \pageref{EqDefiOOpt} \tabularnewline
\end{longtable}

\end{document}